\documentclass[12pt,leqno]{amsart}
\usepackage{mathtools,amsmath,stmaryrd}
\usepackage{amssymb}
\usepackage{amsthm}
\usepackage{epsf}
\usepackage{graphicx}
\usepackage{esint} 
\usepackage{dsfont}
\usepackage[pagebackref]{hyperref} 
\hypersetup{pdfpagemode=FullScreen,  colorlinks=true} 
\usepackage{enumerate} 
\usepackage{xcolor,bbm,easybmat,bm}

\usepackage{amsfonts}
\usepackage{verbatim}
\usepackage{booktabs}
\usepackage{color}
\usepackage{graphicx,float}

\usepackage{pgf,tikz}
\usetikzlibrary{cd,arrows,babel,shapes,positioning}
\usepackage{amsmath,amsfonts}
\usepackage{xcolor}
\usepackage{tikz} 
\usetikzlibrary{positioning, decorations.pathmorphing, arrows.meta, matrix}

\numberwithin{equation}{section}

\newcommand\footnoteref[1]{\protected@xdef\@thefnmark{\ref{#1}}\@footnotemark}

\DeclareRobustCommand{\SkipTocEntry}[5]{}

\newcommand{\subsectionstar}[1]{%
  \addtocontents{toc}{\SkipTocEntry}%
  \subsection*{#1}%
}

\newcommand{\EPS}{\varepsilon}
\newcommand{\Div}{\mathrm{div}}

\newcommand{\AP}{A_{\Vert}}

\def\YYint#1#2#3{{\setbox0=\hbox{$#1{#2#3}{\iint}$}
		\vcenter{\hbox{$#2#3$}}\kern-.51\wd0}}
 
\theoremstyle{plain}
\newtheorem{thm}[equation]{Theorem}
\newtheorem{lemma}[equation]{Lemma}
\newtheorem{cor}[equation]{Corollary}
\newtheorem{prop}[equation]{Proposition}

\newtheorem{defin}[equation]{Definition}

\numberwithin{equation}{section}

\theoremstyle{plain}
\newcommand{\refthm}[1]{\emph{Theorem \ref{#1}}}
\newcommand{\reflemma}[1]{\emph{Lemma \ref{#1}}}
\newcommand{\refprop}[1]{\emph{Proposition \ref{#1}}}
\newcommand{\refcor}[1]{\emph{Corollary \ref{#1}}}

\theoremstyle{definition}
\newcommand{\refdef}[1]{\emph{Definition \ref{#1}}}

\theoremstyle{remark}

\theoremstyle{definition}

\theoremstyle{remark}

\allowdisplaybreaks
\begin{document}

\title[Regularity problem for mixed \(L^1-L^\infty\) condition on \(\partial_t A\)]{Equivalence between solvability of the Dirichlet and Regularity problem under a mixed \(L^1-L^\infty\) condition on \(\partial_t A\)}


\author[Ulmer]{Martin Ulmer}
\address{Department of Mathematics, 
 Brown University, RI, US}
\email{martin\_ulmer@brown.edu}

\begin{abstract}
\noindent
We study an elliptic operator \(L:=-\mathrm{div}(A\nabla \cdot)\) on the upper half space. It is known that solvability of the Regularity problem in \(\dot{W}^{1,p}\) implies solvability of the Dirichlet problem in \(L^{p'}\) for the adjoint operator. Previously, Shen (\cite{shen_relationship_2007}, 2007) established a partial reverse result. In our work, we show that if we assume a mixed \(L^1-L^\infty\) condition on only \(|\partial_t A|\), as introduced in \cite{ulmer_solvability_2026}, the full reverse direction holds. As a result, we can extend the class of operators for which the Regularity problem is solvable by operators satisfying this mixed \(L^1-L^\infty\) condition.
\end{abstract}

\maketitle

\tableofcontents

\section{Introduction}

In this work, let \(\Omega:=\mathbb{R}^{n+1}_+:=\mathbb{R}^n\times (0,\infty), n\geq 1\) be the upper half space and set \(L:=-\Div(A\nabla\cdot)\) as a uniformly elliptic operator with bounded measurable coefficients. More specifically, \(A(x,t)\) is a real not necessarily symmetric \(n+1\) by \(n+1\) matrix and there exists \(\lambda_0>0\) such that
\begin{align}
    \lambda_0 |\xi|^2\leq \xi^T A(x,t) \xi \leq \lambda_0^{-1}|\xi|^2 \qquad \textrm{ for all }\xi\in \mathbb{R}^{n+1},\label{eq:DefinitionOfUniformElliptic}
\end{align}
and for a.e. \((x,t)=(x_1,...,x_n,t)\in \mathbb{R}^{n+1}_+\). We are interested in weak solutions $u\in W^{1,2}_{loc}(\Omega)$ that obtain prescribed Dirichlet data $f$, i.e. solutions $u$ to
\[\begin{cases} Lu=-\Div(A\nabla u)=0 &\textrm{in }\Omega, \\ u=f &\textrm{on }\partial\Omega.\end{cases}\]
The \(L^p\) Dirichlet boundary value problem (see \refdef{def:L^pDirichletProblem}) asks for existence of solutions that attain boundary data $f\in L^p(\partial\Omega)$. If the boundary data \(f\) has more regularity and lies in the homogeneous Sobolev space \(\dot{W}^{1,p}(\partial\Omega)\), the question of existence of a solution is called the Regularity boundary value problem (see \refdef{def:L^pregularityProblem}). In contrast to the Dirichlet boundary value problem, the Regularity problem focuses on convergence of the tangential derivative of the solution \(u\) to the tangential derivative of $f$ on the boundary.
\medskip

The question we would like to address is well established and goes back to Dahlberg (cf. \cite{dahlberg_estimates_1977} and \cite{dahlbert_poisson_1979}). He showed the existence of solutions to the Dirichlet boundary value problem with boundary data in \(L^2\) for the Laplacian on Lipschitz graph domains. Notably, he made the following observation: 
Let \(\Omega=\{(x,t):t>\psi(x)\}\) be a Lipschitz graph domain, i.e. the domain over the graph of the Lipschitz function $\psi$, and $f\in L^2(\partial\Omega)$. Solving
\[\begin{cases} \Delta u=0 &\textrm{in }\Omega,\\ u=f &\textrm{on } \partial\Omega,\end{cases}\qquad \textrm{is equivalent to solving } \qquad \begin{cases} Lu=0 &\textrm{in }\mathbb{R}^n\times(0,\infty),\\  u= \tilde{f} &\textrm{on } \mathbb{R}^n,\end{cases}\]
for an elliptic operator \(L\) depending on the Lipschitz function \(\psi\) and \(\tilde{f}(x)=f(x,\psi(x))\). This suggests that solving a simpler elliptic PDE on a more complicated Lipschitz graph domain is equivalent to solving a more complicated elliptic PDE on the simpler domain of the upper half space.
Depending on which flattening one uses in the argument, the operator \(L\) has some additional properties. These sets of additional properties can be summarized as either the Dahlberg-Kenig-Pipher (DKP) condition or the \(t\)-independence condition. The literature regarding the DKP condition is extensive, addressing not only the Dirichlet boundary value problem but also a range of different elliptic or parabolic boundary value problems (cf. the survey article \cite{dindos_boundary_2023} and references therein). Similarly, the \(t\)-independent condition yields solvability of various other boundary value problems (cf. \cite{kenig_new_2000}, \cite{kenig_regularity_2008}, \cite{hofmann_regularity_2015}, \cite{hofmann_regularity_2015}, \cite{castro_boundedness_2016}, \cite{nystrom_l2_2017}, \cite{auscher_dirichlet_2018}, \cite{hofmann_dirichlet_2022}). Broadly speaking, the DKP condition controls the Lipschitz constant of the matrix \(A\) close to the boundary and requires it to grow less than the function \(1/t\) close to the boundary. The \(t\)-independence condition on the other hand asserts \(A(x,t)=A(x)\), i.e. that \(A\) is independent in transversal direction with potentially almost arbitrarily bad behavior in any direction parallel to the boundary.

\medskip
Both of these structurally different conditions, the DKP and \(t\)-independent condition, are sufficient for solvability of the Dirichlet and Regularity boundary value problem. This raises the question of whether there are other sufficient conditions or improvements of these two that still allow us to solve the Dirichlet or Regularity boundary value problem. Finding such improvements would reduce the number of operators for which we do not know whether any boundary value problem is solvable. It is also noteworthy that \cite{modica_construction_1980} and \cite{caffarelli_completely_1981} provide examples of linear uniformly elliptic operators for which the Dirichlet boundary value problem is not solvable. Furthermore, the article \cite{kenig_new_2000} demonstrates that if the matrix \(A\) is nonsymmetric, we cannot expect to obtain \(L^2\) solvability for the Dirichlet problem with \(t\)-independent coefficients. In these cases, the best we can hope for is solvability with \(L^p\) data for potentially large \(p\). Given that the solvability range of the regularity problem is dual to that of the Dirichlet problem, we can only expect solvability for the Regularity problem for potentially small \(p>1\) in our context.
\medskip

Although the \(t\)-independent condition for the Dirichlet problem in the symmetric case was well established by Jerison and Kenig via a ``Rellich" identity in \cite{jerison_dirichlet_1981}, the extension from symmetric to nonsymmetric matrices in the Dirichlet and Regularity problem necessitated new tools and took some time. It was only after the Kato conjecture was resolved (cf. \cite{Auscher_Kato}) that solvability of the Dirichlet boundary value problem (cf. \cite{hofmann_square_2015}) and the Regularity boundary value problem (cf. \cite{hofmann_regularity_2015}) could be proved for nonsymmetric matrices under the \(t\)-independence condition and by use of these Kato tools.
\medskip

In addition to the two previously mentioned conditions, there is also the Dini-condition from \cite{fabes_necessary_1984}, where the authors showed that \(t\)-independence can be relaxed if we have continuous coefficients. More precisely, they assume a symmetric \(A\) with continuous coefficients, a bounded \(C^1\)-domain \(\Omega\), and that the modulus of continuity 
\[\eta(s)=\sup_{P\in \partial\Omega, 0<r<s}|A_{ij}(P-rV(P))-A_{ij}(P)|\]
with outer normal vector field \(V\) satisfies the Dini-type condition
\begin{align}\int_0\frac{\eta(s)^2}{s}ds<\infty.\label{DiniTypeCond}\end{align}
Under these assumptions, they show \(\omega\in B_2(\sigma)\subset A_\infty(\sigma)\), i.e. solvability of the \(L^2\) Dirichlet problem. This Dini-type condition has also been shown to be somewhat necessary in the following sense: For a given function \(\beta\) with \(\int_0 \frac{\beta(s)^2}{s}ds=+\infty\) the article \cite{caffarelli_completely_1981} constructs symmetric continuous matrices with modulus \(\eta(s)\leq \beta(s)\) which admit elliptic measures \(\omega\) that are completely singular with respect to the surface measure. Later, \cite{dahlberg_absolute_1986} was able to drop the assumption of continuity by examining perturbations from continuous matrices. Lastly, \cite{alfonseca_analyticity_2011} demonstrates that a sufficiently small \(L^\infty\) perturbation of a symmetric \(t\)-independent matrix still allows the resulting operator to solve the Dirichlet problem. This condition is also a generalization of the \(t\)-independence condition and even applies to operators with complex coefficients, despite still being limited to symmetric matrices. All of these conditions were studied in the context of the Dirichlet but not the Regularity problem.

\medskip

Another condition that was studied in \cite{ulmer_solvability_2026} is the mixed \(L^1-L^\infty\) condition which states that $A$ is weakly differentiable in $t$ and
\begin{align}
\int_0^\infty \Vert \partial_t A(\cdot, t)\Vert_{L^\infty}dt<\infty. \label{cond:mixedL1LInftyCond}
\end{align}
From \cite{ulmer_solvability_2026} we know that if \(|\partial_t A|\leq C/t\) and \eqref{cond:mixedL1LInftyCond} is satisfied, then the Dirichlet problem is solvable for some \(1<p<\infty\). It is clear that this condition generalizes the \(t\)-independence condition, but it differs from both the Dini-condition and the condition in \cite{alfonseca_analyticity_2011}. Even if the Dini-condition (or the condition in \cite{alfonseca_analyticity_2011}) serves as sufficient condition for unbounded domains, for nonsymmetric matrices and for the Regularity problem---which is not established yet--- there are examples of matrices in \cite{ulmer_solvability_2026} that satisfy either condition but not the other. Interestingly, for \(n=1\), the same methods also lead to an improvement. In the upper half plane, we only need to assume an \(L^1\) Carleson condition on \(|\partial_t A|\) instead of \eqref{cond:mixedL1LInftyCond} (cf. \cite{ulmer_solvability_2026}).
\medskip

Since these new conditions ensure solvability for the Dirichlet problem, it is natural to ask whether they also enable us to solve the Regularity problem. We will derive this result as a consequence of our main theorem.
\medskip

To determine the solvability of either boundary value problem, it can be useful to analyze their relationships. In particular, \cite{kenig_neumann_1993} established that on Lipschitz domains solvability of the Regularity problem \((R)^L_p\) with data in \(\dot{W}^{1,p}\) implies solvability of the Dirichlet problem \((D)^{L^*}_{p'}\) of the adjoint operator with boundary data in \(L^{p'}\) for the dual exponent (see also \cite{dindos_regularity_2012} for the endpoint \(p=1\)). While their result is only formulated for symmetric operators, it is well established that this proof also applies for nonsymmetric operators. In that sense, we can say that the Dirichlet problem is generally easier to solve than the Regularity problem. Interestingly, there is a partial reverse result from \cite{shen_relationship_2007} which is also stated for symmetric operators only but also holds for nonsymmetric ones. It states that if the Dirichlet problem \((D)^{L^*}_{p'}\) is solvable then either the Regularity problem \((R)^L_p\) is solvable or \((R)^L_q\) is not solvable for any \(1<q<\infty\). Which of the two cases is present though depends on the given PDE. For instance, we know that under the DKP condition or the \(t\)-independence condition solvability the Dirichlet problem \((D)^{L^*}_{p'}\) is equivalent to solvability of the Regularity problem \((R)^L_{p}\). On the contrary, \cite{gallegos_one-sided_2025} gives an example on a domain with purely unrectifiable boundary, where this is not the case. Whether an equivalence of solvability of \((D)^{L^*}_{p'}\) and \((R)^L_{p}\) holds in general (on Lipschitz domains for instance), continues to be an open question.

The main result of this work is that this reverse implication also holds if we assume \eqref{cond:mixedL1LInftyCond}, the mixed \(L^1-L^\infty\) condition on \(|\partial_t A|\).

\begin{thm}\label{MAINTHM}
    Assume \(L:=-\mathrm{div}(A\nabla \cdot)\) is a uniformly elliptic operator with bounded, merely measurable coefficients and let \(\Omega=\mathbb{R}^{n+1}_+\). Let \(p>1\) and assume that \(|\partial_t A(x,t)|\leq C/t\) holds for some $C>0$. If \((\operatorname{D})_{p'}^{L^*}\) is solvable and \eqref{cond:mixedL1LInftyCond} holds,
    then \((\operatorname{R})_p^L\) is solvable.
\end{thm}

Since \cite{ulmer_solvability_2026} established solvability of the Dirichlet problem for some \(1<p<\infty\) under the \(L^1-L^\infty\) condition \eqref{cond:mixedL1LInftyCond} for all dimensions \(n\geq 1\) and additionally under the \(L^1\)-Carleson condition on \(|\partial_t A|\) for \(n=1\) (see \eqref{cond:ExplanationForCarlesonCondition}), we can conclude solvability of the Regularity problem as presented in the following corollary.

\begin{cor}\label{cor:MainCor1}
     Assume \(L:=-\mathrm{div}(A\nabla \cdot)\) is a uniformly elliptic operator with bounded, merely measurable coefficients and let \(\Omega=\mathbb{R}^{n+1}_+\). If there exists \(C>0\) such that \(A\) satisfies \(|\partial_t A|\leq C/t\) and \eqref{cond:mixedL1LInftyCond}, then there exists \(p>1\) such that \((\operatorname{R})_p^L\) is solvable.
\end{cor}

Since the mixed \(L^1-L^\infty\) condition \eqref{cond:mixedL1LInftyCond} generalizes the \(t\)-independent condition, \refcor{cor:MainCor1} is also an alternative proof of solvability of the Regularity problem for \(t\)-independent operators. This result has been proved in \cite{hofmann_regularity_2015} and our proof only relies on integration by parts techniques and all the tools the resolved Kato conjecture provides including the semi group theory, but does not need single layer potentials like \cite{hofmann_regularity_2015}.

\subsectionstar{Overview of the proof}

The proof of the main result \refthm{MAINTHM} is provided in the last section, Section~\ref{section:MainProof}. A central object in the proof will be the approximation operator \(\mathcal{P}_tf:=e^{-t^2L^t_{||}}f\), of which we are going to use several area function bounds. If the matrix is \(t\)-independent, this approximation operator simplifies to a elliptically homogenized semigroup operator \(\mathcal{P}_tf:=e^{-t^2L_{||}}f\), for which the necessary area function bounds are well established in the literature. In that case, the proof of \refthm{MAINTHM} in Section~\ref{section:MainProof} gives an alternative proof to the \(t\)-independent case in \cite{hofmann_regularity_2015} without layer potentials.

However, assuming \eqref{cond:mixedL1LInftyCond}, the mixed $L^1-L^\infty$ condition, requires establishing corresponding $L^p$ area function bounds for the more involved approximation operator \(e^{-t^2L^t_{||}}f\) and partial derivatives of $\mathcal{P}_tf$. For that, we are going to split \(\partial_t\mathcal{P}_t\) into the sum of two different operators \(W_1f\) and \(W_2f\) and examine the area function expressions separately (see Section \ref{section:rho}). Our goal is to apply the real interpolation method to get \(L^p\) bounds for the area function expressions. This reduces matters to proving Hardy-Sobolev space to \(L^1\) boundedness of the area functions and boundedness of the corresponding Carleson function in \(L^\infty\). The real interpolation method and its application in our case are presented in Section \ref{section:RealInterpolation}. The remaining two sections, Section \ref{section:W2} and Section \ref{section:W1}, now deal with proving all the remaining area functions bounds on \(W_1f\) and \(W_2f\) respectively. Notably, first, we have to establish the necessary area function bounds for \(W_2f\) in Section \ref{section:W2}, before we can turn to the area function bounds of \(W_1f\) in Section \ref{section:W1}, since the bounds for \(W_1f\) rely on already established bounds for \(W_2f\).

\subsectionstar{Acknowledgements}
The author wishes to express sincere gratitude to Jill Pipher and
Martin Dindo\v{s} for their valuable and
insightful discussions on this subject,
which significantly contributed to the
refinement of the applied methods.
\section{Preliminaries}

\subsection{Notation and Setting}
We will work on the upper half space \(\Omega=\mathbb{R}^{n+1}_+=\mathbb{R}^n\times (0,\infty)\), where we call the last component the \(t\)-direction, or \(t\)-component, and use the following definitions:
\begin{itemize}
    \item A ball with center \((x,t)\in\mathbb{R}^n\times\mathbb{R}\) and radius $r>0$ is denoted as $B(x,t,r)$. A boundary cube centered at \(x\in\partial\Omega\) with radius \(r\) is denoted by \(\Delta(x,r):=B(x,0,r)\cap\partial\Omega\);
    \item A Carleson region over a boundary ball is written as \(T(\Delta(x,r)):=B(x,0,r)\cap\Omega\);
    \item The surface measure is the \(n\)-dimensional Hausdorff measure restricted to the boundary of the domain and denoted by \(\sigma:=\mathcal{H}^{n}|_{\partial\Omega}\), and we also use the notation $|E|:=\sigma(E)$ for $E\subset\mathbb{R}^n$ or $|E|:=\mathcal{H}^{n+1}(E)$ for $E\subset\mathbb{R}^{n+1}$;
    \item The matrix \(A(x,t)\in\mathbb{R}^{(n+1)\times(n+1)}\) is uniformly elliptic, i.e. \eqref{eq:DefinitionOfUniformElliptic} holds. Its coefficients are real, merely bounded and measurable. The matrix \(A\) is potentially nonsymmetric, and we set the different components as
    \[A(x,t)=\begin{pmatrix} \AP(x,t) & b(x,t)\\c(x,t) & d(x,t)\end{pmatrix},\]
    where \(\AP(x,t)\in\mathbb{R}^{n\times n}, c(x,t)\in\mathbb{R}^{1\times n}, b(x,t)\in\mathbb{R}^{ n\times 1}, d(x,t)\in\mathbb{R}\). Then we set \(L:=-\mathrm{div}(A\nabla \cdot)\) as an elliptic operator and \(L^*:=-\mathrm{div}(A^*\nabla \cdot)\) is the corresponding adjoint operator, where \(A^*\) is the transpose of \(A\).
    Furthermore, we assume that each component of $A$ is weakly differentiable in \(t\) almost everywhere. 
    \item For every $k\in\mathbb{Z}$, the set \(\mathcal{D}_k(\mathbb{R}^n):=\{Q^k_l\}_{l\in\mathbb{Z}}\) is a dyadic decomposition of $\mathbb{R}^n$ of scale $2^{-k}$, i.e. \(Q^k_l\subset\mathbb{R}^n\) is a boundary cube with diameter equal to \(2^{-k}\). Furthermore, the family \(\{Q_l\}_l\) and the family of its enlargements \(\{3Q_l\}_l\) have finite overlap, i.e. \(1\leq \sum_{l}\chi_{Q^k_l}\leq \sum_{l}\chi_{3Q^k_l}\leq N\) for some fixed \(N\) independent of the scale \(k\). If $E\subset\mathbb{R}^n$ is a set, we denote by \(\mathcal{D}_k(E):=\{Q^k_l; Q^k_l\cap E\neq \emptyset\}\) the family of cubes of scale $2^{-k}$ that cover $E$;
    \item The classical Hardy-Littlewood maximal operator on $\mathbb{R}^n$ is defined by
    \[\mathcal{M}[F](x):=\sup_{r>0}\frac{1}{\sigma(\Delta(x,r))}\int_{\Delta(x,r)}|F(y)| d\sigma(y);\]
    \item We define the family of elliptic operators \((L^t_{||})_{t>0}=(-\mathrm{div}_{||}(\AP(\cdot,t)\nabla_{||} \cdot))_{t>0}\), where the subscript \(||\) means that we take the gradient or divergence only with respect to the first \(n\) components or the \(x-\)components only. If clear from context, we might also write \(L^t=L_{||}^t\) dropping the subscript. Let us also note here that \(L_{||}^t\) is an operator on functions \(v:\mathbb{R}^n\to\mathbb{R}\);
    \item A nontangential cone with aperture \(\alpha>0\) is given by \[\Gamma_\alpha(x):=\{(y,t)\in\Omega; |x-y|<\alpha t\},\]
    and \(\Gamma_\alpha^\tau(x):=\Gamma_\alpha(x)\cap\{(y,t)\in\Omega; t\leq \tau\}\) denotes the cone truncated at height \(\tau>0\);
    \item The mean-valued nontangential maximal function is defined as \[\tilde{\mathcal{N}}_{p,\alpha}(F)(x):=\sup_{(y,s)\in\Gamma_\alpha(x)}\Big(\fint_{B(y,s,s/2)}|F(z,\tau)|^pdzd\tau\Big)^{1/p}, \]
    with \(p\geq 1\). If the aperture $\alpha$ has no influence on the argument, we drop the subscript $\alpha$. If $p=\infty$, we recover the classical nontangential maximal function
    \[\mathcal{N}_{\alpha}(F)(x):=\sup_{(y,s)\in\Gamma_\alpha(x)}|F(y,s)|.\]
    Furthermore, we set the at height \(\tau>0\) truncated version as 
    \[\tilde{\mathcal{N}}_\alpha^\tau(F)(x):=\sup_{(y,s)\in\Gamma^\tau_\alpha(x)}\fint_{B(y,s,s/2)}|F(z,t)|dzdt, \]
    and the away truncated version for \(\tau>0\) as
    \begin{align}\tilde{\mathcal{N}}_{\alpha,\tau}(F)(x):=\sup_{(y,s)\in\Gamma_\alpha(x)\cap\{(y,s)\in \Omega; s\geq \tau \}}\fint_{B(y,s,s/2)}|F(z,\tau)|dzd\tau; \label{def:NontangentialMaximalFucntionTruncatedAway}\end{align}
    \item The area function of \(F\in L^2_{loc}\) is set as
    \begin{align}\mathcal{A}_\alpha(F)(x):=\Big(\int_{\Gamma_\alpha(x)}\frac{|F(y,t)|^2}{t^{n+1}}dydt\Big)^{1/2},\label{recall:areafunctiondef}\end{align}
    and the square function is set as \(\mathcal{S}(F)(x):=\mathcal{A}(tF)(x)\);
    \item The Carleson function is defined by
    \begin{align}\mathcal{C}(f)(x):=\Big(\sup_{r>0}\frac{1}{\sigma(\Delta(x,r))}\int_{T(\Delta(x,r))}\frac{|f(y,t)|^2}{t}dydt\Big)^{1/2}\label{recall:carelsonfunctiondef}\end{align}
    for \(x\in \partial\Omega=\mathbb{R}^n\);
    \item We say that a nonnegative measure \(\mu\) is a \textit{Carleson measure} if there exists a \(C>0\) such that for every boundary ball \(\Delta\)
    \[\mu(T(\Delta))\leq C\sigma(\Delta).\]
    The smallest such constant \(C\) is also called the Carleson norm of \(\mu\) and denoted by \(\Vert \mu\Vert_{\mathcal{C}}\). Moreover, we refer to the \(L^1\)-Carleson condition  on \(\partial_t A\) by the condition
    \begin{align}
        \sup_{\Delta\subset\partial\Omega\textrm{ boundary ball}}\frac{1}{\sigma(\Delta)}\int_{T(\Delta)}\sup_{(y,s)\in B(x,t,t/2)}|\partial_s A(y,s)| dxdt\leq C<\infty.\label{cond:ExplanationForCarlesonCondition}
    \end{align}
\end{itemize}
    An important property of the Carleson measures is the following duality property to nontangential maximal functions.
    \begin{prop}[Carleson's Theorem]\label{Carleson's Theorem}
    For every \(F:\mathbb{R}^{n+1}_+\to\mathbb{R}\), with \(\tilde{\mathcal{N}}_1(F)\in L^1(d\sigma)\) and \(G:\mathbb{R}^{n+1}_+\to\mathbb{R}\) with \(\Vert \sup_{B(x,t,t/2)}|G|dxdt \Vert_{\mathcal{C}}<\infty\) we have
    \begin{align}
        \int_{\mathbb{R}^{n+1}_+}|F(x,t)G(x,t)|dxdt\leq C\Big\Vert \sup_{B(x,t,t/2)}|G| dxdt\Big\Vert_\mathcal{C}\int_{\mathbb{R}^{n}}\tilde{\mathcal{N}}_1(F)(x)dx.\label{eq:DualityCarelsonNormNontangential}
    \end{align}
\end{prop}
    This property can be formulated also more generally and in an \(L^p\)-version and can be found for the nontangential maximal function in \cite[Proposition 3]{coifman_new_1985}, or in \cite[Proposition 3.11 or Corollary 3.12]{milakis_harmonic_2013}, and for the mean-valued nontangential maximal function in \cite[Lemmas A.18 and A.20]{mourgoglou_solvability_2025}.

Let us define the Dirichlet boundary value problem with boundary data in \(L^p\).

\begin{defin}[\((\operatorname{D})_p^L\)]\label{def:L^pDirichletProblem}
    We say the \textit{\(L^p\) Dirichlet boundary value problem} is solvable for \(L\) if for all boundary data for \(f\in C^\infty_c(\partial\Omega)\cap L^p(\partial\Omega)\) the unique solution \(u\in W_{loc}^{1,2}(\Omega)\) of
    \[\begin{cases}
        Lu=0&\Omega,
        \\
        u=f&\partial\Omega,
    \end{cases}\]
    satisfies
    \begin{align}\Vert \mathcal{N}(u)\Vert_{L^p(\partial\Omega)}\lesssim\Vert f\Vert_{L^p(\partial\Omega)},\label{eq:DPinequalityElliptic}\end{align}
    where the implied constants are independent of \(u\) and \(f\). In this case we also write that \((\operatorname{D})^L_p\) holds for \(L\).
\end{defin}

Furthermore, we call a Borel function \(g:\partial\Omega\to\mathbb{R}\) a \textit{Haj\l{}asz upper gradient} of \(f:\partial\Omega\to\mathbb{R}\) if
\[|f(X)-f(Y)|\leq |X-Y|(g(X)+g(Y))\qquad\textrm{for a.e. }X,Y\in \partial\Omega.\]
We denote the collection of all Haj\l{}asz upper gradients of \(f\) as \(\mathcal{D}(f)\) and define \(\dot{L}_1^p(\partial\Omega)\) by all \(f\) with
\[\Vert f\Vert_{\dot{L}^p_1(\partial\Omega)}:= \inf_{g\in\mathcal{D}(f)}\Vert g\Vert_{L^p(\partial\Omega)}<\infty.\]
This space is also called \textit{homogeneous Haj\l{}asz Sobolev space}. In the present case of a flat boundary we have that \(\dot{L}_1^p(\mathbb{R}^n)=\dot{W}^{1,p}(\mathbb{R}^n)\) with comparable norms (see \cite{mourgoglou_regularity_2023}). Hence, we will continue with the boundary space \(\dot{W}^{1,p}\) for the regularity problem. 

\begin{defin}[\((\operatorname{R})_p^L\)]\label{def:L^pregularityProblem}
    We say the \textit{\(L^p\) Regularity boundary value problem} is solvable for \(L\) if for all boundary data for \(f\in C^\infty_c(\partial\Omega)\cap \dot{W}^{1,p}(\partial\Omega)\) the unique solution \(u\in W_{loc}^{1,2}(\Omega)\) of
    \[\begin{cases}
        Lu=0&\Omega,
        \\
        u=f&\partial\Omega,
    \end{cases}\]
    satisfies
    \begin{align}\Vert \tilde{\mathcal{N}}_1(\nabla u)\Vert_{L^p(\partial\Omega)}\lesssim\Vert f\Vert_{\dot{L}_{1}^p(\partial\Omega)},\label{eq:RPinequalityElliptic}\end{align}
    where the implied constants are independent of \(u\) and \(f\). In this case we also write that \((\operatorname{R})^L_p\) holds for \(L\).
\end{defin}
We would like to point out, that the regularity problem is often defined with $\tilde{\mathcal{N}}_2$ instead of $\tilde{\mathcal{N}}_1$ on the left hand side of \eqref{eq:RPinequalityElliptic}. However, it is well known that entirely real variable arguments (as in \cite{Shen_extrapolation_2019}) allow the conclusion 
\[\big(\fint_B|\nabla u|^2dxdt\big)^{1/2}\lesssim \fint_{\frac{3}{2}B}|\nabla u|dxdt\] 
for solutions $u$, and hence under increasing the aperture $\tilde{\mathcal{N}}_{2,\alpha}(\nabla u)\lesssim \tilde{\mathcal{N}}_{1,2\alpha}(\nabla u)$ holds pointwise. Thus, this definition is equivalent to defining the regularity problem with $\tilde{\mathcal{N}}_2$.

\subsection{The approximation operator \(\mathcal{P}_t\)}\label{section:rho}
Recall that  \(L^t:=L_{||}^t:=-\mathrm{div}_{||}(\AP(x,t)\nabla_{||}\cdot)\) is an elliptic operator on the right upper block \(\AP\) for each fixed \(t>0\). For this family of operators we can define the approximation operator \[\mathcal{P}_t:=e^{-t^2L_{||}^t}.\]
Without the dependence of the operator \(L_{||}^t\) on \(t\), this \(\mathcal{P}_t\) is the ellipticized heat semigroup, and solves the heat equation on the upper half space if \(t^2\) is replaced by \(t\) . Since this operator does not satisfy any PDE directly, we can decouple the dependencies in \(t\) and define
\[Wf(x,t,s):=e^{-tL_{||}^s}f(x)\]
as the solution to the ("t-independent") heat equation
\[\begin{cases} \partial_t Wf(x,t,s)+L_{||}^s Wf(x,t,s)=0, &(x,t)\in \Omega, \\ Wf(x,0,s)=f(x), &x\in\partial\Omega\end{cases}\]
for fixed \(s>0\). Taking the partial derivative of \(Wf\) in \(s\) yields
\[\begin{cases} \partial_t \partial_sWf(x,t,s)-L_{||}^s \partial_s Wf(x,t,s)=\mathrm{div}_{||}(\partial_s \AP(x,s)\nabla_{||}Wf(x,t,s)), &(x,t)\in \Omega, \\ Wf(x,0,s)=0, &x\in\partial\Omega.\end{cases}\]
By Duhamel's principle, we can obtain an explicit formula for \(\partial_s Wf(x,t,s)\). Hence we obtain
\begin{align*}
\partial_t\mathcal{P}_tf(x)&=[\partial_tWf(x,t^2,s) + \partial_sWf(x,t^2,s)]|_{s=t}
\\
&=-2tL_{||}^te^{-t^2L_{||}^t}f(x) + \int_0^t2\tau e^{-(t^2-\tau^2)L_{||}^t}\mathrm{div}(\partial_t\AP(x,t)\nabla_{||}e^{-\tau^2L_{||}^t}f(x))d\tau
\\
&=:W_1f(x,t) + W_2f(x,t).
\end{align*}
This argument is similar to the argument presented in \cite{ulmer_solvability_2026}. It is noteworthy that if the operator \(L_{||}^t\) is independent of \(t\), then \(W_2f\equiv 0\). This corresponds to assuming $t$-independence of the operator \(L\) and recovers the approximation operators defined in \cite{hofmann_regularity_2015}.
\medskip

When clear from context and when there is no chance of confusing the present derivative with the full one is, we will drop the subscript \(||\).

\subsection{Properties of the heat semi-group and \(\mathcal{P}_t\)}

First, we would like to define \(L^p-L^q\) boundedness and off-diagonal estimates.

\begin{defin}\label{def:L2boundednessOfSemigroupOperators}
    We say $\mathcal{T}_t$ is $L^p-L^q$ bounded for some $q,p\in [1,\infty]$ with $p\leq q$, if for some constant $C>0$, and all $t>0$ and $f\in L^2(\mathbb{R}^n)\cap L^p(\mathbb{R}^n)$
    \[\Vert\mathcal{T}_tf\Vert_{L^q}\leq Ct^{-\frac{n}{2p}+\frac{n}{2q}}\Vert f\Vert_{L^p}.\]
    We say $\mathcal{T}_t$ satisfies $L^p-L^q$ off-diagonal estimates for some $q,p\in [1,\infty]$ with $p\leq q$, if for some constants $C,c>0$, for all closed sets $E,F\subset\mathbb{R}^n$, all $f\in L^2(\mathbb{R}^n)\cap L^p(\mathbb{R}^n)$ with support in $E$, and all $t>0$ we have
    \[\Vert\mathcal{T}_tf\Vert_{L^q}\leq Ct^{-\frac{n}{2p}+\frac{n}{2q}}e^{-c\frac{\mathrm{dist}(E,F)^2}{t}}\Vert f\Vert_{L^p}.\]
\end{defin}

Concretely, for the semigroup operators that will be relevant here, we have the following statements. This proposition does not list all the boundedness and off-diagonal estimates that hold for these semigroup operators and we would like to refer the reader to \cite{AuscherTchamitchian1998} (or also \cite{hofmann_lp_2022}) for a more complete discussion of the $L^p$ semigroup theory.

\begin{prop}\label{prop:L2boundednessOfSemigroupOperators}
    Let \(L_{||}:=\mathrm{div}_{||}(\AP(x)\nabla_{||} \cdot)\) be a \(t\)-independent operator. Then 
    \begin{enumerate}
        \item \(e^{-tL_{||}}, t\partial_te^{-tL_{||}}\) are \(L^p-L^p\) bounded and satisfy \(L^2-L^2\) off-diagonal estimates for all \(1<p<\infty\);
        \item \(\sqrt{t}\nabla_{||} e^{-tL_{||}}\) is \(L^p-L^2\) bounded and satisfies \(L^p-L^2\) off-diagonal estimates for all \(1<p\leq 2\);
        \item \(\sqrt{t}e^{-tL_{||}}\mathrm{div}\) is \(L^2-L^2\) bounded and satisfies \(L^2-L^2\) off-diagonal estimates.
    \end{enumerate}
    Furthermore, we also have for \(f\in W^{1,2}(\mathbb{R}^n)\)
    \[\Vert \nabla_{||} e^{-tL_{||}}f\Vert_{L^2}\leq C\Vert \nabla_{||} f\Vert_{L^2}.\]
\end{prop}
The boundedness and the off-diagonal estimates for \(\sqrt{t}e^{-tL_{||}}\mathrm{div}\) are not explicitly stated in \cite{AuscherTchamitchian1998} or \cite{hofmann_lp_2022}, but follow immediately from the \(L^2-L^2\) boundedness and off-diagonal estimates of its adjoint operator $\sqrt{t}\nabla_{||}e^{-tL_{||}^*}$.

Moreover, we have the following Gaussian bounds of the kernel \(K_t(x,y)\) of the semigroup \(e^{tL_{||}}\).

\begin{prop}[{\cite[Prop 4.3]{hofmann_lp_2022} or \cite[Theorem 6.17]{ouhabaz_analysis_2004}}]\label{Prop:KernelBounds}
    For any \(l\in\mathbb{N}\) there exists \(C=C(n,\lambda,l), \beta=\beta(n,\lambda)>0\) such that
    \begin{align}
        |\partial_t^l K_t(x,y)|\leq C_l t^{-\frac{n}{2}-l}e^{-\beta\frac{|x-y|^2}{t}}
    \end{align}
    for all \(x,y\in\mathbb{R}^n\).
\end{prop}

A consequence of the kernel are the following local bounds.
\begin{prop}[{\cite[Prop. 11]{hofmann_dirichlet_2022} or \cite[Cor. 5.8]{ulmer_solvability_2026}}]\label{prop:PROP11}
    Let \(x\in \partial\Omega=\mathbb{R}^n\) and \((y,t)\in\Gamma_\alpha(x)\), then there exists \(C=C(n,\lambda,\alpha)>0\) such that
    \begin{enumerate}[(i)]
        \item \begin{align}
        W_1f(y,t)=tL_{||}e^{-t^2L_{||}^t}f(y,t)\leq C \mathcal{M}[|\nabla_{||}f|](x);
    \end{align}
    \item \begin{align}
        e^{-t^2L_{||}^t}(f-(f)_{\Delta_{\alpha t}(x)})(y,t)\leq C \mathcal{M}[|\nabla_{||}f|](x),
    \end{align}
    \end{enumerate}
    where \((f)_{\Delta_{\alpha t}(x)}\) denotes the average \((f)_{\Delta_{\alpha t}(x)}:=\fint_{\Delta_{\alpha t}(x)} f(y) dy\).
\end{prop}

The previous proposition gives rise to local bounds for operators with the spatial gradient $\nabla_{||}$. Part of this lemma can also be found in \cite[Lemma 6.4]{ulmer_solvability_2026}, but we provide the proof for completeness.
\begin{lemma}\label{lemma:NablaSemigroupBoundedbyM[nablaf]}
    Let $Q\subset\partial\Omega$ be a boundary cube and  $0<c_1<c_2$. Then, for all $f\in W^{1,2}(\mathbb{R}^n)$ and all $t\in (c_1l(Q),c_2l(Q))$ 
    \[\fint_{2Q} |\nabla_{||} e^{-t^2L^s}f|^2dx+l(Q)^4\fint_{2Q} |\nabla_{||} L^s e^{-t^2L^s}f|^2dx\leq C\inf_{x\in Q}\mathcal{M}[|\nabla_{||}f|]^2(x),\]
    where the constant $C$ depends on $c_1$ and $c_2$.
\end{lemma}

\begin{proof}
First, let us define a smooth cut-off function $\psi\in C^\infty(\mathbb{R}^n)$ with $\psi\equiv 1$ on $2Q$, $\psi\equiv 0$ on $\mathbb{R}^n\setminus 3Q$ and $|\nabla \psi|\lesssim \frac{1}{l(Q)}$. Then for any function $v\in L^2_{loc}(\mathbb{R}^n)$ and any constant $c_0\in\mathbb{R}$, we have
\begin{align*}\fint_{3Q} |\nabla_{||} v|^2\psi^2 dx &\lesssim \fint_{3Q} A_{||} \nabla_{||} v \cdot \nabla_{||} (v-c_0)  \psi^2 , dx \\
&= \fint_{3Q} A_{||} \nabla_{||} v \cdot \nabla_{||}((v-c_0) \psi^2) , dx -  2\fint_{3Q} A_{||} \nabla_{||} v \cdot \nabla_{||}\psi ((v-c_0) \psi) dx \\
&\lesssim \fint_{3Q} |(L^s v) (v-c_0) \psi^2| dx + \theta\fint_{3Q} |\nabla_{||} v|^2\psi^2dx+ C_\theta\frac{1}{l(Q)^2}\fint_{3Q} |v-c_0|^2 dx,\end{align*}
where we used Young's inequality for some $\theta>0$. For sufficiently small $\theta>0$, the second term can be absorbed into the left hand side, and we get in total
\begin{align}
    \fint_{3Q} |\nabla_{||} v|^2\psi^2 dx\lesssim \fint_{3Q} |(L^s v) (v-c_0) \psi^2| dx + \frac{1}{l(Q)^2}\fint_{3Q} |v-c_0|^2 dx.\label{eq:CaccForvinProof}
\end{align}
By the same arguments as in the proof of \cite[Prop. 11]{hofmann_dirichlet_2022}, we have that for every $y\in 3Q$
\begin{align}|t^3L^sL^se^{-t^2L^s}f(y)|+|tL^se^{-t^2L^s}f(y)|+t^{-1}|e^{-t^2L^s}f(y)-(f)_{4Q}|\lesssim \inf_{x\in Q}M[|\nabla_{||} f|](x).\label{eq:ptwiseboundsinproof}\end{align}

Plugging $v=e^{-t^2L^s}f$ and $c_0=(f)_{4Q}$ into \eqref{eq:CaccForvinProof} gives
\begin{align*}
    \fint_{3Q} |\nabla_{||} e^{-t^2L^s}|^2\psi^2 dx\lesssim \fint_{3Q} |(L^s e^{-t^2L^s}f) (e^{-t^2L^s}f-(f)_{4Q}) \psi^2| dx + \frac{1}{l(Q)^2}\fint_{3Q} |e^{-t^2L^s}f-(f)_{4Q}|^2 dx.
\end{align*}
Thus, by \eqref{eq:ptwiseboundsinproof} and $t\approx l(Q)$
\begin{align*}
    \fint_{3Q} |\nabla_{||} e^{-t^2L^s}|^2\psi^2 dx\lesssim  \inf_{x\in Q}M[|\nabla_{||} f|]^2(x).
\end{align*}
Similarly, plugging $v=L^se^{-t^2L^s}f$ and $c_0=0$ into \eqref{eq:CaccForvinProof} and using \eqref{eq:ptwiseboundsinproof} and $t\approx l(Q)$ gives
\begin{align*}
    l(Q)^4\fint_{3Q} |\nabla_{||} L^s e^{-t^2L^s}|^2\psi^2 dx&\lesssim l(Q)^4\fint_{3Q} |(L^s L^s e^{-t^2L^s}f) (L^se^{-t^2L^s}f) \psi^2| dx
    \\
    &\qquad+ l(Q)^2\fint_{3Q} |L^se^{-t^2L^s}f|^2 dx
    \\
    &\lesssim  \inf_{x\in Q}M[|\nabla_{||} f|]^2(x).
\end{align*}
\end{proof}

A direct consequence of \refprop{prop:PROP11} is the following:
\begin{cor}\label{cor:NontangentialBoundsOnW_1andW_2}
    For \(f\in W^{1,p}(\partial \Omega), p>1\), it holds that
    \begin{enumerate}[(i)]
        \item \(\Vert \tilde{\mathcal{N}}_2(\nabla_{||}\mathcal{P}_sf)\Vert_{L^p(\partial\Omega)} + \Vert \tilde{\mathcal{N}}_1(\nabla_{||}\mathcal{P}_sf)\Vert_{L^p(\partial\Omega)}\lesssim\Vert \nabla_{||} f\Vert_{L^p(\partial\Omega)};\textrm{ and}\)
        \item \(\Vert \tilde{\mathcal{N}}_1(W_1f)\Vert_{L^p(\partial\Omega)}\lesssim\Vert \nabla_{||} f\Vert_{L^p(\partial\Omega)}.\)
    \end{enumerate}
\end{cor}

The following two results are the much celebrated solution to the Kato conjecture which was fully resolved in \cite{Auscher_Kato} for \(p=2\). The \(L^p\) theory for other \(p\) was first fully established in \cite{auscher_necessary_2007} with partial \(L^p\) theory results appearing earlier (please refer to the introduction of \cite{auscher_necessary_2007} for the full history).

\begin{prop}\label{prop:KatoConjecture}
    For a function \(f\in \dot W^{1,p}(\mathbb{R}^n)\) we have
    \[\Vert L_{||}^{1/2} f\Vert_{L^p(\mathbb{R}^n)}\leq C(n,\lambda_0,\Lambda_0)\Vert \nabla f\Vert_{L^p(\mathbb{R}^n)}\]
    for all \(1<p<\infty\). Furthermore, there exists \(\varepsilon_1=\varepsilon_1(n,\lambda_0,\Lambda_0)>0\) such that if \(1<p<2+\varepsilon_1\) then
    \[\Vert \nabla f\Vert_{L^p(\mathbb{R}^n)}\leq C(n,\lambda_0,\Lambda_0)\Vert L_{||}^{1/2} f\Vert_{L^p(\mathbb{R}^n)}.\]
\end{prop}

\subsection{Properties of $W_1f$ and $W_2f$}

First, we note that since \(W_1f\) and \(W_2f\) satisfy certain PDEs (see Section~\ref{section:rho}), we have Cacciopolli-type inequalities for \(W_1f\) and \(W_2f\) respectively.
\begin{prop}[Lemma 4.5 in \cite{ulmer_solvability_2026}]\label{prop:CacciopolliTypeInequality}
    Let \(B(x,t,4r)\subset \Omega\), then 
    \begin{align}
        \int_{B(x,t,r)} |\nabla_{||} W_1f(y,s)|^2 dyds\lesssim \frac{1}{r^2}\int_{B(x,t,2r)} |W_1f(y,s)|^2 dyds.\label{lemma:w_tSqFctBound1}
    \end{align}
\end{prop}

For the Cacciopolli-type inequality for \(W_2f\), let us define
\[v_2(x,t;s)=\int_0^t2\tau e^{-(t^2-\tau^2)L^s}\mathrm{div}(\partial_s\AP(x,s)\nabla e^{-\tau^2L^s}f)(x)d\tau.\]
Recall that \(W_2f(x,s)=v_2(x,s,s)\).

\begin{lemma}[{\cite[Lemma 7.4]{ulmer_solvability_2026}}]\label{lemma:SqFctBoundsForw_s^2}
Let $s>0$, and let \(\Delta(x,s/32)\subset \partial\Omega\) be a boundary cube. Then
\begin{align}
    \int_{\Delta(x,s/32)} |\nabla_{||} W_2f(x,s)|^2dx&\lesssim \frac{1}{s^2}\fint_{s}^{5s/4}\int_{\Delta(x,s/16)}|v_2(x,t;s)|^2dxdt\nonumber
    \\
    &\qquad+ \Vert\partial_s A(\cdot,s)\Vert_{L^\infty(\Delta(x,s/16))}^2\fint_{s}^{5s/4}\int_{\Delta(x,s/16)} |\nabla e^{-t^2L^s}f|^2dxdt \nonumber
    \\
    &\qquad+ \Vert\partial_s A(\cdot,s)\Vert_{L^\infty(\Delta(x,s/16))}^2 s^2 \fint_{s}^{5s/4}\int_{\Delta(x,s/16)}|\nabla\partial_te^{-t^2L^s}f|^2dxdt.\label{lemma:localEstimateNablaw_s^2}
\end{align}
\end{lemma}

Furthermore, the operator $tW_1$ and $t^2\nabla W_1$ satisfy $L^2-L^2$ boundedness as stated in \refprop{prop:L2boundednessOfSemigroupOperators}. The following proposition states that $W_2$ and $\nabla W_2$ also satisfy $L^2-L^2$ boundedness. The main proof ideas stem from \cite[Lemma 7.3]{ulmer_solvability_2026}, but we present the full proof with all adaptations below.
\begin{prop}\label{prop:W2l2l2boundedness}
    With implicit constants only depending on $n,\lambda_0,$ and $\Lambda_0$, it holds that
    \begin{enumerate}[(i)]
        \item \label{lemma:w_s^2spatialL^2Bound} \(\Vert W_2f\Vert_{L^2(\mathbb{R}^n)}\lesssim \Vert \partial_s A(\cdot,s)\Vert_{L^\infty(\mathbb{R}^n)}\Vert f\Vert_{L^2(\mathbb{R}^n)},\) and hence
        \item \(\Vert \nabla_{||} W_2f\Vert_{L^2(\mathbb{R}^n)}\lesssim \frac{1}{s}\Vert \partial_s A(\cdot,s)\Vert_{L^\infty(\mathbb{R}^n)}\Vert f\Vert_{L^2(\mathbb{R}^n)}.\)\label{lemma:w_s^2SqFctBound1}
    \end{enumerate}
\end{prop}

\begin{proof}
    We begin by proving \eqref{lemma:w_s^2spatialL^2Bound}. By Minkowski's inequality, we obtain
    \begin{align*}
        \Vert W_2f\Vert_{L^2(\mathbb{R}^n)}&=\Big\Vert\int_0^{s} \eta^2\tau e^{-(s^2-\tau^2)L^s}\Div(\partial_s A(\cdot,s)\nabla e^{-\tau^2L^s}f)d\tau\Big\Vert_{L^2(\mathbb{R}^n)}
        \\
        &\qquad =\int_0^{s} \Vert\tau e^{-(s^2-\tau^2)L^s}\Div(\partial_s A(\cdot,s)\nabla e^{-\tau^2L^s}f)\Vert_{L^2(\mathbb{R}^n)} d\tau.
    \end{align*}
    By the Kato estimates \refprop{prop:KatoConjecture} and \refprop{prop:L2boundednessOfSemigroupOperators} we can observe that 
    \begin{align*}
        &\Vert\tau e^{-(s^2-\tau^2)L^s}\Div(\partial_s A(\cdot,s)\nabla e^{-\tau^2L^s}f)\Vert_{L^2(\mathbb{R}^n)}
        \\
        &\lesssim\Vert\tau \nabla e^{-(s^2-\tau^2)L^s}(L^s)^{-1/2}\Div(\partial_s A(\cdot,s)\nabla e^{-\tau^2L^s}f)\Vert_{L^2(\mathbb{R}^n)}
        \\
        &\lesssim\frac{\tau}{\sqrt{s^2-\tau^2}}\Vert (L^s)^{-1/2}\Div(\partial_s A(\cdot,s)\nabla e^{-\tau^2L^s}f)\Vert_{L^2(\mathbb{R}^n)}.
    \end{align*}
    Continuing with \refprop{prop:KatoConjecture} and \refprop{prop:L2boundednessOfSemigroupOperators} we get
    \begin{align*}
        &\lesssim\frac{\tau}{\sqrt{s^2-\tau^2}}\Vert\partial_s A(\cdot,s)\nabla e^{-\tau^2L^s}f\Vert_{L^2(\mathbb{R}^n)}
        \\
        &\lesssim\frac{1}{\sqrt{s^2-\tau^2}}\Vert\partial_s \AP(\cdot,s)\Vert_{L^\infty(\mathbb{R}^n)} \Vert f\Vert_{L^2(\mathbb{R}^n)}
    \end{align*}
    Integrating \(\frac{\tau}{\sqrt{s^2-\tau^2}}\) from \(0\) to \(s\) in \(\tau\) gives a constant, and hence \eqref{lemma:w_s^2spatialL^2Bound} follows.

    The second bound \eqref{lemma:w_s^2SqFctBound1} follows from Caccioppoli-type inequality \reflemma{lemma:SqFctBoundsForw_s^2}, which yields
    \begin{align*}
        \Vert\nabla_{||} W_2f\Vert_{L^2(\mathbb{R}^n)}&\lesssim \Big(\frac{1}{s^2}\fint_{s^{5s/4}}\Vert v_2(\cdot,t;s)\Vert_{L^2(\mathbb{R}^n)}^2+\Vert\partial_sA(\cdot,s)\Vert_{L^\infty(\mathbb{R}^n)}\Vert \nabla e^{-t^2L^s}f\Vert_{L^2(\mathbb{R}^n)}^2
        \\
        &\hspace{3cm} +\Vert\partial_sA(\cdot,s)\Vert_{L^\infty(\mathbb{R}^n)}s^2\Vert \nabla \partial_t e^{-t^2L^s}f\Vert_{L^2(\mathbb{R}^n)}^2dt\Big)^{1/2}.
    \end{align*}
    Fro the term with \(v_2\) we apply completely analogous arguments to \eqref{lemma:w_s^2spatialL^2Bound}, and the bound for the other two terms follows from \refprop{prop:L2boundednessOfSemigroupOperators}.

\end{proof}

\section{Tent spaces and real interpolation}\label{section:RealInterpolation}

Let \(A_0,A_1\) be two normed vector spaces of functions \(a:\mathbb{R}^n\to\mathbb{R}\). For each \(a\in A_0 + A_1\) we define the \(K-\)functional of real interpolation by
\[K(a,t,A_0,A_1)=\inf_{a=a_0+a_1}\Vert a_0\Vert_{A_0} + t\Vert a_1\Vert_{A_1}.\]
For \(0<\theta<1,1\leq q\leq \infty\), we denote by \((A_0,A_1)_{\theta,q}\) the real interpolation space between \(A_0\) and \(A_1\) defined as
\[(A_0,A_1)_{\theta,q}=\Big\{a\in A_1+A_0:\Vert a\Vert_{\theta,q}=\Big(\int_0^\infty (t^{-\theta}K(a,t,A_0,A_1))^q\frac{\mathrm{d}t}{t}\Big)^{1/q}<\infty \Big\}.\]
According to \cite[Thm. 3.1.2]{bergh_interpolation_1976}, \(K\) can be seen as an exact interpolation functor, which means that if an operator \(T\) is bounded from \(A_0\to B_0\) and from \(A_1\to B_1\) for linear normed vector spaces of functions, then
\(T:(A_0,A_1)_{\theta,q}\to (B_0,B_1)_{\theta,q}\) is a bounded linear operator with \(\Vert T\Vert_{(A_0,A_1)_{\theta,q} \to (B_0,B_1)_{\theta,q}}\leq C\Vert T\Vert^{1-\theta}_{A_0\to B_0}\Vert T\Vert^{\theta}_{A_1\to B_1}\).  
\medskip

Let us introduce the Hardy-Sobolev space like in \cite{dindos_regularity_2012} or \cite{badr_atomic_2010} and \cite{badr_abstract_2010}.
\begin{defin}\label{def:HardySpace}
    Let \(1<\beta\leq\infty\). We call a function \(a:\mathbb{R}^n\to\mathbb{R}\) a \textit{homogeneous Hardy-Sobolev \(\beta\)-atom} associated to a boundary ball \(\Delta\subset\mathbb{R}^n\) if
    \begin{enumerate}[(i)]
        \item \(\mathrm{supp}(a)\subset \Delta\);
        \item \(\Vert \nabla a\Vert_{L^\beta(\mathbb{R}^n)}\leq |\Delta|^{-\frac{1}{\beta'}} \); and
        \item \(\Vert a\Vert_{L^1(\mathbb{R}^n)}\leq l(\Delta). \)
    \end{enumerate}
    If \(f\) can be written as
    \begin{align}f=\sum_{j=1}^\infty\lambda_ja_j\label{eq:AtomicHardyDecomposition}\end{align}
    for \(\beta\)-atoms \(a_j\) and coefficients \(\lambda_j\in\mathbb{R}\) with \(\sum_{j=1}^\infty|\lambda_j|<\infty\), we say that \(f\in \dot{HS}^{1,\beta}_{atom}\), where \(\Vert f\Vert_{\dot{HS}^{1,\beta}_{atom}}:=\inf\sum_{j=1}^\infty|\lambda_j|\) with an infimum that is taken over all choices of decompositions \eqref{eq:AtomicHardyDecomposition}. 
\end{defin}

Now, we have the following real interpolation result.
\begin{prop}[{\cite[Thm. 0.4]{badr_abstract_2010}}]
    For every \(\beta\in(1,\infty]\) and \(1<p<\infty\) the real interpolation space is
    \[(\dot{HS}^{1,\beta}_{atom}, \dot{W}^{1,\infty})_{1-1/p,p}=\dot{W}^{1,p}.\]
\end{prop}

On the other hand, recall the definition of the area function \eqref{recall:areafunctiondef} and of the Carleson function \eqref{recall:carelsonfunctiondef} and define the tent spaces over \(\mathbb{R}^n\) with parameter \(1\leq p<\infty\) as
\begin{align*}
T^{p,2}(\Omega):=\{F\in L^2_{loc}(\Omega);\Vert F\Vert_{T^{p,2}}:=\Vert \mathcal{A}(F)\Vert_{L^p(\partial\Omega)}<\infty\}
\end{align*}
and
\begin{align*}
T^{\infty,2}(\Omega):=\{F\in L^2_{loc}(\Omega);\Vert F\Vert_{T^{\infty,2}}:=\Vert \mathcal{C}(F)\Vert_{L^\infty(\partial\Omega)}<\infty\}.
\end{align*}

By \cite{coifman_new_1985} we have
\begin{prop}[{\cite[Thm. 4']{coifman_new_1985}}]
    For every \(1<p<\infty\) the real interpolation space is
    \[(T^{1,2},T^{\infty,2})_{1-1/p,p}=T^{p,2}.\]
\end{prop}

\subsection{\(L^p\) Area function bounds}
For the proof of the following \reflemma{lemma:AreaFctBoundsInLp}, we are going to need different $L^p$ area function bounds. In this section we demonstrate how the real interpolation method is used to obtain these bounds.

The goal is to show the following lemma.
\begin{lemma}\label{lemma:AreaFctBoundsInLp}
    Assume \eqref{cond:mixedL1LInftyCond}. Let \(1<p<\infty\). Then there exists a constant \(C=C(n,\lambda_0,\Lambda_0)>0\) such that for all \(f\in \dot{W}^{1,p}(\partial\Omega)\)
    \[\Vert \mathcal{A}(\partial_t\mathcal{P}_tf)\Vert_{L^p(\partial\Omega)} + \Vert \mathcal{A}(t\nabla_{||}\partial_t\mathcal{P}_tf)\Vert_{L^p(\partial\Omega)}\leq C \Vert \nabla_{||}f\Vert_{L^p(\partial\Omega)}.\]
\end{lemma}

To prove this, first, we break the operator \(\partial_t\mathcal{P}_tf\) up into the sum of \(W_1f\) and \(W_2f\) as introduced in Section \ref{section:rho} and discuss the necessary bounds separately.

Specifically, we establish the following two results, which immediately imply \reflemma{lemma:AreaFctBoundsInLp}.
\begin{cor}\label{cor:AreaFctBoundsInLpW2}
    Assume \eqref{cond:ExplanationForCarlesonCondition}. For every \(1<p<\infty\) and \(f\in\dot{W}^{1,p}(\mathbb{R}^n)\) we have
    \[\Vert \mathcal{A}(W_2f)\Vert_{L^p(\partial\Omega)}+\Vert \mathcal{A}(t\nabla_{||}W_2f)\Vert_{L^p(\partial\Omega)}\leq C\Vert \nabla_{||} f\Vert_{L^p(\partial\Omega)}.\]
\end{cor}
\begin{cor}\label{cor:AreaFctBoundsInLpW1}
    Assume \eqref{cond:mixedL1LInftyCond}. For every \(f\in \dot{W}^{1,p}(\mathbb{R}^n)\) we have
    \[\Vert \mathcal{A}(W_1f)\Vert_{L^p(\partial\Omega)}+\Vert \mathcal{A}(t\nabla_{||}W_1f)\Vert_{L^p(\partial\Omega)}\leq C\Vert \nabla_{||} f\Vert_{L^p(\partial\Omega)}.\]
\end{cor}
It remains to prove these corollaries. By the method of real interpolation it suffices to establish the
\begin{itemize}
    \item \(L^\infty\) to Carleson function bound \(\Vert \mathcal{C}(\mathcal{T}_tf)\Vert_{L^\infty}\lesssim \Vert \nabla_{||} f\Vert_{L^\infty}\), and
    \item the Hardy-Sobolev space to \(L^1\) bound \(\Vert \mathcal{A}(\mathcal{T}_tf)\Vert_{L^1}\lesssim \Vert \nabla_{||} f\Vert_{\dot{HS}^{1,\beta}_{atom}}\)
\end{itemize}
for each of the four operators \(\mathcal{T}_t\in\{W_1,W_2,t\nabla_{||} W_1,t\nabla_{||} W_2\}\). 
This will be done by establishing the following four lemmas, which then together imply \refcor{cor:AreaFctBoundsInLpW2} and \refcor{cor:AreaFctBoundsInLpW1}:

\begin{lemma}\label{lemma:CarlesonFunctionBoundW_2}
    Assume \eqref{cond:ExplanationForCarlesonCondition}. Let \(f\in \dot{W}^{1,\infty}(\mathbb{R}^n)\). Then
    \begin{enumerate}[(i)]
        \item \(\Vert \mathcal{C}(W_2f)\Vert_{L^\infty}=\sup_{\Delta\subset\partial\Omega}\frac{1}{\sigma(\Delta)}\int_{T(\Delta)}\frac{|W_2f(x,t)|^2}{t}dxdt\lesssim \Vert\nabla_{||}f\Vert_{L^\infty}\), and\label{item:LocalSquareBoundW2}
        \item \(\Vert \mathcal{C}(t\nabla_{||} W_2f)\Vert_{L^\infty}=\sup_{\Delta\subset\partial\Omega}\frac{1}{\sigma(\Delta)}\int_{T(\Delta)}\frac{|t\nabla_{||} W_2f(x,t)|^2}{t}dxdt\lesssim \Vert\nabla_{||}f\Vert_{L^\infty}\).\label{item:LocalSquareBoundNablaW2}
    \end{enumerate}
\end{lemma}

\begin{lemma}\label{lemma:HardyNormBoundsW_2}
    Assume \eqref{cond:ExplanationForCarlesonCondition}. Let \(f\in \dot{HS}_{atom}^{1,\beta}(\mathbb{R}^n)\) for \(2<\beta<\infty\). Then we have
    \begin{enumerate}[(i)]
        \item \(\Vert\mathcal{A}(W_2f)\Vert_{L^1(\mathbb{R}^n)}\lesssim \Vert f\Vert_{\dot{HS}_{atom}^{1,\beta}(\mathbb{R}^n)},\) and \label{item:HardySquareBound}
        \item \(\Vert\mathcal{A}(t\nabla_{||} W_2f)\Vert_{L^1(\mathbb{R}^n)}\lesssim \Vert f\Vert_{\dot{HS}_{atom}^{1,\beta}(\mathbb{R}^n)}\).\label{item:HardySquareBoundNabla}
    \end{enumerate}
\end{lemma}

\begin{lemma}\label{lemma:CarlesonFunctionBoundW_1}
    Assume \eqref{cond:mixedL1LInftyCond}. Let \(f\in \dot{W}^{1,\infty}(\mathbb{R}^n)\), then
    \begin{enumerate}[(i)]
        \item \(\Vert \mathcal{C}(W_1f)\Vert_{L^\infty}=\sup_{\Delta\subset\partial\Omega}\frac{1}{\sigma(\Delta)}\int_{T(\Delta)}\frac{|W_1f(x,t)|^2}{t}dxdt\lesssim \Vert\nabla_{||}f\Vert_{L^\infty}\), and\label{item:LocalSquareBoundW1}
        \item \(\Vert \mathcal{C}(t\nabla_{||} W_1f)\Vert_{L^\infty}=\sup_{\Delta\subset\partial\Omega}\frac{1}{\sigma(\Delta)}\int_{T(\Delta)}\frac{|t\nabla_{||} W_1f(x,t)|^2}{t}dxdt\lesssim \Vert\nabla_{||}f\Vert_{L^\infty}\).\label{item:LocalSquareBoundNablaW1}
    \end{enumerate}
\end{lemma}

\begin{lemma}\label{lemma:HardyNormBoundsW_1}
    Assume \eqref{cond:mixedL1LInftyCond}. Let \(f\in \dot{HS}_{atom}^{1,2}(\mathbb{R}^n)\). Then
    \begin{enumerate}[(i)]
        \item \(\Vert\mathcal{A}(W_1f)\Vert_{L^1(\mathbb{R}^n)}\lesssim \Vert f\Vert_{\dot{HS}_{atom}^{1,2}(\mathbb{R}^n)}\)\label{item:HardySquareBoundW1}
        \item \(\Vert\mathcal{A}(t\nabla_{||} W_1f)\Vert_{L^1(\mathbb{R}^n)}\lesssim \Vert f\Vert_{\dot{HS}_{atom}^{1,2}(\mathbb{R}^n)}\)\label{item:HardySquareBoundNablaW1}
    \end{enumerate}
\end{lemma}

These lemmas are proved in the following two sections, Section~\ref{section:W2} and Section~\ref{section:W1}. First, we will deal with the bounds involving \(W_2\), i.e. \reflemma{lemma:CarlesonFunctionBoundW_2} and \reflemma{lemma:HardyNormBoundsW_2}, before we are able to turn to the bounds involving \(W_1\), namely \reflemma{lemma:CarlesonFunctionBoundW_1} and \reflemma{lemma:HardyNormBoundsW_1}.

\section{Area function bounds of \(W_2\)}\label{section:W2}

\subsection{Carleson function bounds and proof of \reflemma{lemma:CarlesonFunctionBoundW_2}}

We begin with obtaining pointwise Carleson function bounds for \(W_2f\).
\begin{lemma}\label{lemma:PointwiseCarlesonFunctionBoundW_2}
    Assume \eqref{cond:ExplanationForCarlesonCondition}. Let \(f\in W^{1,1}_{loc}(\mathbb{R}^n)\). Then
    \begin{enumerate}[(i)]
        \item \(\mathcal{C}(W_2f)^2(z):=\sup_{\Delta(z)}\frac{1}{\sigma(\Delta)}\int_{T(\Delta)}\frac{|W_2f(x,t)|^2}{t}dxdt\lesssim \mathcal{M}[\mathcal{M}[|\nabla_{||}f|^2]](z)\), and\label{item:PointwiseLocalSquareBoundW2}
        \item \(\mathcal{C}(t\nabla_{||} W_2f)^2(z):=\sup_{\Delta(z)}\frac{1}{\sigma(\Delta)}\int_{T(\Delta)}\frac{|t\nabla_{||} W_2f(x,t)|^2}{t}dxdt\lesssim \mathcal{M}[\mathcal{M}[|\nabla_{||}f|^2]](z)\),\label{item:PointwiseLocalSquareBoundNablaW2}
    \end{enumerate}
    where the suprema are taken over all boundary balls centered at \(z\).
\end{lemma}

\begin{proof}
    For \eqref{item:PointwiseLocalSquareBoundW2} we let \(\Delta\) be a boundary ball with center \(z\in \partial\Omega\). Then Minkowski's inequality yields

    \begin{align*}
        &\int_{T(\Delta)}\frac{|W_2f(x,s)|^2}{s}dxds
        \\
        &\qquad = \int_0^{\infty}\frac{1}{s}\Big\Vert\int_0^{s} \tau e^{-(s^2-\tau^2)L_{||}^s}\mathrm{div}(\partial_s \AP(x,s)\nabla e^{-\tau^2L_{||}^s}f)d\tau\Big\Vert_{L^2(\Delta)}^2 ds
        \\
        &\qquad\leq \int_0^{l(\Delta)}\frac{1}{s}\Big(\int_0^{s} \Vert\tau e^{-(s^2-\tau^2)L_{||}^s}\mathrm{div}(\partial_s \AP(x,s)\nabla e^{-\tau^2L_{||}^s}f)\Vert_{L^2(\Delta)} d\tau\Big)^2  ds.
    \end{align*}

    For the inner \(L^2\) norm, we use off-diagonal estimates \refprop{prop:L2boundednessOfSemigroupOperators} and cut-off functions \(\eta_k:=\chi_{2^k\Delta\setminus2^{k-1}\Delta}\) such that \(\sum_k\chi_{2^k\Delta\setminus 2^{k-1}\Delta} + \chi_{2\Delta}=1\) to obtain
    \begin{align}
        &\Vert\tau e^{-(s^2-\tau^2)L_{||}^s}\mathrm{div}(\partial_s \AP(x,s)\nabla e^{-\tau^2L_{||}^s}f)\Vert_{L^2(\Delta)}\nonumber
        \\
        &\leq \sum_{k\geq 2}\Vert\tau e^{-(s^2-\tau^2)L_{||}^s}\mathrm{div}(\eta_k\partial_s \AP(x,s)\nabla e^{-\tau^2L_{||}^s}f)\Vert_{L^2(\Delta)}\nonumber
        \\
        &\qquad + \Vert\tau e^{-(s^2-\tau^2)L_{||}^s}\mathrm{div}(\chi_{2\Delta}\partial_s \AP(x,s)\nabla e^{-\tau^2L_{||}^s}f)\Vert_{L^2(\Delta)}\nonumber
        \\
        &\leq \sum_{k\geq 2}\frac{\tau}{\sqrt{s^2-\tau^2}}e^{-c\frac{2^{2k}l(\Delta)^2}{s^2-\tau^2}}\Vert\partial_s \AP(x,s)\nabla e^{-\tau^2L_{||}^s}f\Vert_{L^2(2^k\Delta\setminus 2^{k-1}\Delta)}\nonumber
        \\
        &\qquad + \frac{\tau}{\sqrt{s^2-\tau^2}}\Vert\partial_s \AP(x,s)\nabla e^{-\tau^2L_{||}^s}f\Vert_{L^2(2\Delta)}=I+II.\nonumber
    \end{align}
    We continue for the first term with the same ideas, in particular the off-diagonal estimates, and the pointwise bound \(|\partial_sA|\leq \frac{1}{s}\), and get
    \begin{align}
        I&\leq \sum_{k\geq 2}\sum_{m\geq 2}\frac{\tau}{\sqrt{s^2-\tau^2}}e^{-c\frac{2^{2k}l(\Delta)^2}{s^2-\tau^2}}\Vert\partial_s \AP(x,s)\nabla e^{-\tau^2L_{||}^s}(\eta_{k+m}(f-(f)_{2^k\Delta}))\Vert_{L^2(2^k\Delta\setminus 2^{k-1}\Delta)}\nonumber
        \\
        & \qquad + \sum_{k\geq 2}\frac{\tau}{\sqrt{s^2-\tau^2}}e^{-c\frac{2^{2k}l(\Delta)^2}{s^2-\tau^2}}\Vert\partial_s \AP(x,s)\nabla e^{-\tau^2L_{||}^s}(\chi_{2^{k+1}\Delta}(f-(f)_{2^k\Delta}))\Vert_{L^2(2^k\Delta\setminus 2^{k-1}\Delta)}\nonumber
        \\
        &\leq \sum_{k\geq 2}\sum_{m\geq 2}\frac{1}{\sqrt{s^2-\tau^2}}e^{-c\frac{2^{2k}l(\Delta)^2}{s^2-\tau^2} - c\frac{2^{2(m+k-1)}l(\Delta)^2}{\tau^2}}\Vert\partial_s \AP(\cdot,s)\Vert_\infty\Vert(f-(f)_{2^k\Delta})\Vert_{L^2(2^{k+m}\Delta)}\nonumber
         \\
        & \qquad + \sum_{k\geq 2}\frac{1}{\sqrt{s^2-\tau^2}}e^{-c\frac{2^{2k}l(\Delta)^2}{s^2-\tau^2}}\Vert\partial_s \AP(\cdot,s)\Vert_{L^\infty}\Vert f-(f)_{2^k\Delta}\Vert_{L^2(2^{k+1}\Delta)}.\nonumber
        \\
        &\leq |\Delta|^{1/2}\Big(\sum_{k\geq 2}\sum_{m\geq 2}\frac{2^{(k+m)(\frac{n}{2}+1)}l(\Delta)}{s\sqrt{s^2-\tau^2}}e^{-c\frac{2^{2k}l(\Delta)^2}{s^2-\tau^2} - c\frac{2^{2(m+k-1)}l(\Delta)^2}{\tau^2}}\mathcal{M}[|\nabla_{||} f|^2]^{1/2}\label{eq:sumtointegral1}
         \\
        & \qquad + \sum_{k\geq 2}\frac{\tau 2^{(k+1)\frac{n}{2}}l(\Delta)}{s\sqrt{s^2-\tau^2}}e^{-c\frac{2^{2k}l(\Delta)^2}{s^2-\tau^2}}\mathcal{M}[|\nabla_{||} f|^2]^{1/2}\Big).\label{eq:sumtointegral2}
    \end{align}
    Here we used Poincar\'{e} in the last line, noting that for the first term we can observe that
    \begin{align}
        &\Vert f-(f)_{2^k\Delta}\Vert_{L^2(2^{k+m}\Delta)}\label{observationf-MVonLargerDomain}
        \\
        &\qquad\leq \Vert f-(f)_{2^{k+m}\Delta}\Vert_{L^2(2^{k+m}\Delta)} + \sum_{l=1}^m \Vert(f)_{2^{k+l}\Delta}-(f)_{2^{k+l-1}\Delta}\Vert_{L^2(2^{k+m}\Delta)}\nonumber
        \\
        &\qquad\lesssim 2^{(k+m)}l(\Delta)\Vert \nabla_{||}f\Vert_{L^2(2^{k+m}\Delta)} + \sum_{l=1}^m 2^{k+l}l(\Delta)\Vert\nabla_{||}f\Vert_{L^2(2^{k+m}\Delta)}\nonumber
        \\
        &\qquad\lesssim 2^{(k+m)}l(\Delta)\Vert \nabla_{||}f\Vert_{L^2}\lesssim  2^{(k+m)(\frac{n}{2}+1)}l(\Delta)|\Delta|^{1/2}\mathcal{M}[|\nabla_{||}f|^2]^{1/2}(z),\nonumber
    \end{align}
    We consider the sums in \eqref{eq:sumtointegral1} and \eqref{eq:sumtointegral2} as Riemann sums of integral. For instance for \eqref{eq:sumtointegral2}, the sum can be bounded by the integrals
    \[\int_0^\infty\frac{\tau}{s\sqrt{s^2-\tau^2}}x^{\frac{n}{2}-1}e^{-cx^2\frac{l(\Delta)^2}{s^2-\tau^2}}dx,\]
    which after the change of variables \(y=\frac{l(\Delta)^2}{s^2-\tau^2}x\) gives a convergent integral
    \[\frac{\tau\sqrt{s^2-\tau^2}}{sl(\Delta)\sqrt{s^2-\tau^2}}\int_0^\infty y^{\frac{n}{2}-1}e^{-y^2}dy\lesssim \frac{C}{l(\Delta)}.\]
    Similar arguments for the other sum in \eqref{eq:sumtointegral1} lead to the conclusion
    \[I\lesssim \mathcal{M}[|\nabla_{||} f|^2]^{1/2}(z)|\Delta|^{1/2}\Big(\frac{1}{l(\Delta)}+ \frac{\tau}{l(\Delta)\sqrt{s^2-\tau^2}}\Big).\]
    
    Hence we obtain
    \begin{align*}
         &\int_0^{l(\Delta)}\frac{1}{s}\Big(\int_0^{s} I d\tau\Big)^2  ds
        \\
        &\lesssim\int_0^{l(\Delta)}\frac{1}{s}\Big(\int_0^{s} \mathcal{M}[|\nabla_{||} f|^2]^{1/2}(z)|\Delta|^{1/2}\Big(\frac{1}{l(\Delta)}+ \frac{\tau}{l(\Delta)\sqrt{s^2-\tau^2}}\Big) d\tau\Big)^2 ds
        \\
        &\lesssim |\Delta| \mathcal{M}[|\nabla_{||} f|^2](z).
    \end{align*}

    For \(II\) we have for a scale \(k\) such that \(2^{-k}\leq \tau \leq 2^{-k+1}\) with \refprop{prop:PROP11}
    \begin{align*}
    II&= \frac{\tau}{\sqrt{s^2-\tau^2}}\Vert\partial_s \AP(x,s)\nabla e^{-\tau^2L_{||}^s}f\Vert_{L^2(2\Delta)}
    \\
    &= \frac{\tau}{\sqrt{s^2-\tau^2}}\Big(\sum_{Q\in \mathcal{D}_k(3\Delta)}\Vert\partial_s \AP(x,s)\nabla e^{-\tau^2L_{||}^s}f\Vert_{L^2(Q)}^2\Big)^{1/2}
    \\
    &= \frac{\tau}{\sqrt{s^2-\tau^2}}\Big(\sum_{Q\in \mathcal{D}_k(3\Delta)}\sup_{(x,s)\in Q}|\partial_s \AP(x,s)|^2\inf_{(x,s)\in Q} |\mathcal{M}[|\nabla_{||} f|](x,s)|^2\Big)^{1/2}
    \\
    &= \frac{\tau}{\sqrt{s^2-\tau^2}}\Big(\int_{3\Delta}\sup_{(y,t)\in B(x,s,s/2)}|\partial_t \AP(y,t)|^2 |\mathcal{M}[|\nabla_{||} f|]|^2 dx\Big)^{1/2}.
    \end{align*}

    Hence we obtain
    \begin{align*}
        &\int_0^{l(\Delta)}\frac{1}{s}\Big(\int_0^{s} II d\tau\Big)^2 ds
        \\
        &\leq\int_0^{l(\Delta)}\frac{1}{s}\Big(\int_0^{s} \frac{\tau}{\sqrt{s^2-\tau^2}}\Vert\partial_s \AP(\cdot,s)\nabla e^{-\tau^2L_{||}^s}f\Vert_{L^2(2\Delta)} d\tau\Big)^2 ds
        \\
        &\leq\int_0^{l(\Delta)}\frac{1}{s}\Big(\int_0^{s} \frac{\tau}{\sqrt{s^2-\tau^2}}\Big(\int_{3\Delta}\sup_{(y,t)\in B(x,s,s/2)}|\partial_t \AP|^2 \mathcal{M}[|\nabla_{||} f|]^2 dx\Big)^{1/2} d\tau\Big)^2 ds
        \\
        &\leq\int_0^{l(\Delta)}\int_{3\Delta}\sup_{(y,t)\in B(x,s,s/2)}|\partial_t \AP|^2s \mathcal{M}[|\nabla_{||} f|]^2 dx ds.
    \end{align*}
    Making use of Carleson's Theorem \eqref{eq:DualityCarelsonNormNontangential} we continue with
    \begin{align*}
        &\lesssim \Big\Vert\sup_{(y,t)\in B(x,s,s/2)}|\partial_t \AP(y,t)|^2s \Big\Vert_{\mathcal{C}} \int_\Delta \tilde{\mathcal{N}}_1(\mathcal{M}[|\nabla_{||} f|]^2) dx
        \\
        &\lesssim \Big\Vert\sup_{(y,t)\in B(x,s,s/2)}|\partial_t \AP(y,t)|^2s \Big\Vert_{\mathcal{C}} |\Delta| \mathcal{M}[\mathcal{M}[|\nabla_{||} f|]^2](z).
    \end{align*}
    
    In total we get for \eqref{item:PointwiseLocalSquareBoundW2}
    \begin{align*}
        \mathcal{C}(W_2f)(z)&=\sup_{\Delta=\Delta(z)}\frac{1}{|\Delta|}\int_{T(\Delta)}\frac{|W_2f(x,s)|^2}{s}dxds
        \\
        &\lesssim \frac{1}{|\Delta|}\Big(\int_0^{l(\Delta)}\frac{1}{s}\Big(\int_0^{s} I d\tau\Big)^2  ds + \int_0^{l(\Delta)}\frac{1}{s}\Big(\int_0^{s} II d\tau\Big)^2  ds\Big)
        \\
        &\lesssim \mathcal{M}[\mathcal{M}[|\nabla_{||} f|]^2](z).
    \end{align*}

     For \eqref{item:PointwiseLocalSquareBoundNablaW2} we have by Cacciopolli type inequality \refprop{prop:CacciopolliTypeInequality}
    \begin{align*}
        &\int_{T(\Delta)}\frac{|t\nabla_{||} W_2f(x,t)|^2}{t}dxdt
        \lesssim \int_{T(2\Delta)}\frac{|W_2f(x,t)|^2}{t}dxdt
        \\
        &\qquad + \int_{2\Delta}\int_0^{2l(\Delta)}\sup_{(y,s)\in B(x,t,t/2)}|\partial_s \AP(y,s)|^2t \mathcal{M}[|\nabla_{||}f|^2] dxdt
        \\
        &\lesssim \Big\Vert\sup_{(y,s)\in B(x,t,t/2)}|\partial_s \AP(y,s)|^2 t \Big\Vert_{\mathcal{C}} \int_\Delta \tilde{\mathcal{N}}_1(\mathcal{M}[|\nabla_{||} f|]^2) dx
        \\
        &\lesssim \sigma(\Delta) \mathcal{M}[\mathcal{M}[|\nabla_{||} f|^2]](z).
    \end{align*}

\end{proof}

As a corollary we obtain the Carleson measure bound that we need in the real interpolation argument for \refcor{cor:AreaFctBoundsInLpW2}.

\begin{proof}[Proof of \reflemma{lemma:CarlesonFunctionBoundW_2}]
    The only observation needed is that if \(f\in \dot{W}^{1,\infty}(\mathbb{R}^n)\), then \(\mathcal{M}[\mathcal{M}[|\nabla_{||} f|^2]]\leq \Vert \nabla_{||} f\Vert_{L^\infty}^2\). Then the statement follows from \reflemma{lemma:PointwiseCarlesonFunctionBoundW_2}.
\end{proof}

Another corollary is the area function bound with \(\beta>2\) which will be needed for the proof of \reflemma{lemma:HardyNormBoundsW_2}.
\begin{cor}\label{lemma:HardyNormBoundsW_2beta}
    Assume \eqref{cond:ExplanationForCarlesonCondition}. Let \(f\in \dot{W}^{1,\beta}(\partial\Omega)\) for \(2<\beta<\infty\). Then we have
    \begin{enumerate}[(i)]
        \item \(\Vert\mathcal{A}(W_2f)\Vert_{L^\beta(\mathbb{R}^n)}\lesssim \Vert \nabla_{||}f\Vert_{L^\beta(\mathbb{R}^n)},\) and
        \item \(\Vert\mathcal{A}(t\nabla_{||} W_2f)\Vert_{L^\beta(\mathbb{R}^n)}\lesssim \Vert \nabla_{||}f\Vert_{L^\beta(\mathbb{R}^n)}\).
    \end{enumerate}
\end{cor}

\begin{proof}
    By \cite{coifman_new_1985} we know \(\Vert \mathcal{A}(g)\Vert_{L^\beta(\mathbb{R}^n)}\lesssim_\beta \Vert \mathcal{C}(g)\Vert_{L^\beta(\mathbb{R}^n)}\), if \(2<\beta<\infty\). Combining this with \reflemma{lemma:PointwiseCarlesonFunctionBoundW_2} and \(L^p\) boundedness of the Hardy-Littlewood maximal function, we obtain
    \[\Vert \mathcal{A}(W_2f)\Vert_{L^\beta}\lesssim_\beta \Vert \mathcal{C}(W_2f)\Vert_{L^\beta}\lesssim \big(\int_{\partial\Omega} \mathcal{M}[\mathcal{M}[|\nabla_{||} f|^2]]^{\beta/2}dx\big)^{1/\beta}\lesssim \Vert \nabla_{||} f\Vert_{L^\beta}.\]
    A completely analogous argument works for \(t\nabla_{||} W_2\).
\end{proof}

\subsection{Hardy-Sobolev bound (Proof of \reflemma{lemma:HardyNormBoundsW_2})}

\begin{proof}[Proof of \reflemma{lemma:HardyNormBoundsW_2}]
    First, we note that it is enough to show
    \[\Vert\mathcal{A}(W_2f)\Vert_{L^1(\mathbb{R}^n)}, \Vert\mathcal{A}(t\nabla_{||} W_2f)\Vert_{L^1(\mathbb{R}^n)}\leq C\]
    for all homogeneous Hardy-Sobolev \(\beta\)-atoms \(f\) associated with \(\Delta\), whence we assume that \(f\) is such an atom going forward. Furthermore, since both $W_2$ and $\nabla W_2$ kill constants, i.e. \(W_21=0\) and \(\nabla W_21=0\), we can without loss of generality assume that $f$ has mean value zero.
    
    We begin with showing \eqref{item:HardySquareBound}. We split the integral into a local and a far away part 
    \begin{align*}
        \Vert\mathcal{A}(W_2f)\Vert_{L^1(\mathbb{R}^n)}=\Vert\mathcal{A}(W_2f)\Vert_{L^1(\mathbb{R}^n\setminus 5\Delta)} + \Vert\mathcal{A}(W_2f)\Vert_{L^1(5\Delta)}.
    \end{align*}
    For the local part we have by Hölder's inequality, \refcor{lemma:HardyNormBoundsW_2beta} and the properties of the Hardy-Sobolev space (cf. \refdef{def:HardySpace}) that
    \begin{align*}
        \Vert\mathcal{A}(W_2f)\Vert_{L^1(5\Delta)}\lesssim \Vert\mathcal{A}(W_2f)\Vert_{L^\beta(5\Delta)}|\Delta|^{1/\beta'}\lesssim \Vert\nabla_{||} f\Vert_{L^\beta(\mathbb{R}^n)}|\Delta|^{1/\beta'}\lesssim 1.
    \end{align*}

    First, we note that the kernel bounds in \refprop{Prop:KernelBounds} imply
    \begin{align}
        &\sup_{(x,t)\in\Delta_{t/2}(y)} |tL^te^{-t^2L^t}f(x)|^2\nonumber
        \\
        &\qquad\lesssim \begin{cases} \frac{1}{t^{2n+2}}e^{-c\frac{\mathrm{dist}(\Delta_{t/2}(y),\Delta)^2}{t^2}}\Vert f\Vert_{L^1(\Delta)}^2 &\textrm{if }\mathrm{dist}(\Delta_{t/2}(y),\Delta)>0, \\ \frac{1}{t^{2n+2}}\Vert f\Vert_{L^1(\Delta)}^2& \mathrm{else}.\end{cases}\label{eq:CasesForL^1BoundofTermInProof}
    \end{align}

    For the away part we are going to distinguish several cases. We are going to split the integral over \(\mathbb{R}^n\) into integrals over annuli of the form \(2^{j+1}\Delta\setminus 2^{j}\Delta \) for \(j\geq 2\) and bound all cones in each annulus uniformly. However, each cone with tip in \(2^{j+1}\Delta\setminus 2^{j}\Delta \) is itself split into a a close and a far away part \(\Gamma=\Gamma^{2^{j-1}l(\Delta)}\cup \Gamma\setminus\{t\leq 2^{j-1}l(\Delta)\}\).

    To start with, we fix \(y\in\partial\Omega\) with \(y\in 2^{j+1}\Delta\setminus 2^j\Delta\). Let us also fix \(t\) and we look at the term 
    \[\int_{\Delta_{t/2}(y)}\Big|\int_0^t2\tau e^{-(t^2-\tau^2)L^t}\mathrm{div}(\partial_t\AP\nabla e^{-\tau^2L^t}f)d\tau\Big|^2dx.\]

    First assume that \(0\leq t\leq 2^{j-1}l(\Delta)\). Then \(\Delta_{t/2}(y)\subset \Delta_{2^{j-2}l(\Delta)}(y):=\tilde{\Delta}\), and we have for a dualising function \(g\in L^2(\tilde{\Delta})\) 
    \begin{align}
        &\Vert\tau e^{-(t^2-\tau^2)L_{||}^t}\mathrm{div}(\partial_t \AP(x,t)\nabla e^{-\tau^2L_{||}^t}f)\Vert_{L^2(\tilde{\Delta})}\nonumber
        \\
        &=\tau\int_{\mathbb{R}}\partial_t \AP(x,t)\nabla e^{-\tau^2L_{||}^t}f\cdot \nabla e^{-(t^2-\tau^2)L_{||}^t} g dx\nonumber
        \\
        &\lesssim
            \tau\Vert\partial_t\AP\Vert_\infty\big(\Vert \nabla e^{-(t^2-\tau^2)L_{||}^t} g\Vert_{L^2(2\tilde{\Delta})} \Vert \nabla e^{-\tau^2L_{||}^t}f\Vert_{L^2(2\tilde{\Delta})}\nonumber
        \\
        &\qquad+ \Vert \nabla e^{-(t^2-\tau^2)L_{||}^t} g\Vert_{L^2(2\Delta)} \Vert \nabla e^{-\tau^2L_{||}^t}f\Vert_{L^2(2\Delta)}\nonumber
        \\
        &\qquad+ \Vert \nabla e^{-(t^2-\tau^2)L_{||}^t} g\Vert_{L^2(\mathbb{R}^n\setminus 2\tilde{\Delta})} \Vert \nabla e^{-\tau^2L_{||}^t}f\Vert_{L^2(\mathbb{R}^n\setminus 2\Delta)}\big).\label{eq:ProofUsefulEquation1}
    \end{align}
    Let us note that by using a smooth cut-off function \(\eta\) with \(\mathrm{supp}(\eta)\subset\frac{3}{2}\tilde{\Delta}\) and \(\eta\equiv 1\) on \(\tilde{\Delta}\) we have
    \[\int_{\tilde{\Delta}}|\nabla e^{-\tau^2L_{||}^t}f|^2dx\lesssim \int_{2\tilde{\Delta}} W_1f \cdot \mathcal{P}_tf\eta^2 dx + \int_{2\tilde{\Delta}} \nabla e^{-\tau^2L_{||}^t}f\cdot\nabla \eta \mathcal{P}_tf\eta dx\]
    which implies \(\int_{\tilde{\Delta}}|\nabla e^{-\tau^2L_{||}^t}f|^2dx\lesssim \int_{\tilde{\Delta}}|W_1f|^2dx + \frac{1}{t^2}\int_{\tilde{\Delta}}|\mathcal{P}_tf|^2dx\). This observation holds not only for the set \(\tilde{\Delta}\), but also for integrals over other sets and enlargements thereof.
    For each of those terms we can use a similar observation to \eqref{eq:CasesForL^1BoundofTermInProof} and combine this with \refprop{prop:L2boundednessOfSemigroupOperators} to bound \(\eqref{eq:ProofUsefulEquation1}\) by
    \begin{align*}
        &\Vert\partial_t\AP\Vert_\infty\Big[\Big(\int_{2\tilde{\Delta}}\big(\frac{\tau}{\sqrt{t^2-\tau^2}\tau^{n+1}}e^{-c\frac{2^{2(j-1)} l(\Delta)^2}{\tau^2}}\Vert f\Vert_{L^1(\Delta)}\big)^2dx\Big)^{1/2}\Vert g\Vert_{L^2}
        \\
        &\qquad +  \Big(\int_{2\Delta}\big(\frac{1}{\sqrt{t^2-\tau^2}^{n+1}}e^{-c\frac{2^{2(j-1)} l(\Delta)^2}{t^2-\tau^2}}\Vert g\Vert_{L^1(\tilde{\Delta})}\big)^2dx\Big)^{1/2}\Vert f\Vert_{L^2}
        \\
        &\qquad +  \tau\Big(\int_{\mathbb{R}^n\setminus 2\tilde{\Delta}}\big(\frac{1}{\sqrt{t^2-\tau^2}^{n+1}}e^{-c\frac{\mathrm{dist}(x,\tilde{\Delta})^2}{t^2-\tau^2}}\Vert g\Vert_{L^1(\tilde{\Delta})}\big)^2dx\Big)^{1/2} 
        \\
        &\qquad\qquad \cdot\Big(\int_{\mathbb{R}^n\setminus 2\Delta}\big(\frac{1}{\tau^{n+1}}e^{-c\frac{\mathrm{dist}(x,\Delta)^2}{\tau}}\Vert f\Vert_{L^1(\Delta)}\big)^2dx\Big)^{1/2} \Big].
    \end{align*}
    Further, we can use Hölder's inequality and Poincar\'{e}'s inequality and continue with
    \begin{align*}
        &\Vert\partial_t\AP\Vert_\infty\Big[\frac{\tau 2^{jn/2}l(\Delta)^n}{\sqrt{t^2-\tau^2}\tau^{n+1}}e^{-\frac{2^{2(j-1)} l(\Delta)^2}{\tau^2}}\Vert f\Vert_{L^2(\Delta)}\Vert g\Vert_{L^2}
        \\
        &\qquad +  \frac{2^{jn/2}l(\Delta)^n}{\sqrt{t^2-\tau^2}^{n+1}}e^{-\frac{2^{2(j-1)} l(\Delta)^2}{t^2-\tau^2}}\Vert g\Vert_{L^2(\tilde{\Delta})}\Vert f\Vert_{L^2}
        \\
        &\qquad + \frac{2^{jn/2}l(\Delta)^{n/2}}{\sqrt{t^2-\tau^2}^{n+1}}\Vert g\Vert_{L^2(\tilde{\Delta})}\Big(\int_{2^{j-1}l(\Delta)}^\infty r^{n-1}e^{-\frac{r^2}{t^2-\tau^2}}dr\Big)^{1/2}
        \\
        &\qquad\qquad\cdot\frac{l(\Delta)^{n/2}}{\tau^{n}}\Vert f\Vert_{L^2(\Delta)}\Big(\int_{l(\Delta)}^\infty r^{n-1}e^{-\frac{r^2}{\tau^2}}dr\Big)^{1/2} \Big]
        \\
        &\lesssim \Vert\partial_t\AP\Vert_\infty\Vert \nabla_{||}f\Vert_{L^2(\Delta)}\Vert g\Vert_{L^2}\Big[\frac{\tau 2^{jn/2}l(\Delta)^{n+1}}{\sqrt{t^2-\tau^2}\tau^{n+1}}e^{-\frac{2^{2(j-1)} l(\Delta)^2}{\tau^2}}  +  \frac{2^{jn/2}l(\Delta)^{n+1}}{\sqrt{t^2-\tau^2}^{n+1}}e^{-\frac{2^{2(j-1)} l(\Delta)^2}{t^2-\tau^2}}
        \\
        &\qquad + \frac{2^{jn/2}l(\Delta)^{n+1}}{\sqrt{t^2-\tau^2}^{n/2+1}\tau^{n/2}}\Big(\int_{2^{j-1}l(\Delta)/\sqrt{t^2-\tau^2}}^\infty r^{n-1}e^{-r^2}dr\Big)^{1/2}\Big(\int_{l(\Delta)/\tau}^\infty r^{n-1}e^{-r^2}dr\Big)^{1/2}\Big].
    \end{align*}
    The integral expressions are a little bit delicate. If \(n\) is even, than \(\int_{a}^\infty r^{n-1}e^{-r^2}dr=P(a)e^{-a^2}\) where \(P\) is a polynomial of degree \(n-2\). If \(n\) is odd, we can bound the integral above by the same expression with \(n+1\) instead, since we only integrate over numbers greater than \(1\). Hence, we can bound the first integral by \(\frac{2^{(j-1)(n-2)}l(\Delta)^{n-2}}{\sqrt{t^2-\tau^2}^{n-2}}e^{-\frac{2^{2(j-1)}l(\Delta)}{t^2-\tau^2}}\) or \(\frac{2^{(j-1)(n-1)}l(\Delta)^{n-1}}{\sqrt{t^2-\tau^2}^{n-1}}e^{-\frac{2^{2(j-1)}l(\Delta)}{t^2-\tau^2}}\). Since \(\frac{2^{j-1}l(\Delta)}{\sqrt{t^2-\tau^2}}\geq 1\), we can bound both expressions by \(\frac{2^{n(j-1)}l(\Delta)^{n}}{\sqrt{t^2-\tau^2}^{n}}e^{-\frac{2^{2(j-1)}l(\Delta)}{t^2-\tau^2}}\). Similarly, we obtain for the second integral the bound \(\frac{l(\Delta)^n}{\tau^n}e^{-\frac{l(\Delta)^2}{\tau^2}}\).
    Plugging in these bounds yields
    \begin{align*}
        &\lesssim \Vert\partial_t\AP\Vert_\infty\Vert\nabla_{||}f\Vert_{L^2}\Big[\frac{\tau 2^{jn/2}l(\Delta)^{n+1}}{\sqrt{t^2-\tau^2}\tau^{n+1}}e^{-\frac{2^{2(j-1)} l(\Delta)^2}{\tau^2}} +  \frac{2^{jn/2}l(\Delta)^{n+1}}{\sqrt{t^2-\tau^2}^{n+1}}e^{-\frac{2^{2(j-1)} l(\Delta)^2}{t^2-\tau^2}}
        \\
        &\qquad + \frac{2^{jn}l(\Delta)^{2n+1}}{\sqrt{t^2-\tau^2}^{n+1}\tau^{n}}e^{-\frac{2^{2(j-1)} l(\Delta)^2}{t^2-\tau^2}}e^{-\frac{l(\Delta)^2}{\tau^2}} \Big]
        \\
        &\lesssim \frac{1}{t}\Big[\frac{l(\Delta)^{n/2+1}}{\sqrt{t^2-\tau^2}\tau^{n}}e^{-\frac{2^{2(j-1)} l(\Delta)^2}{\tau^2}} +\frac{ l(\Delta)^{n/2+1}}{\sqrt{t^2-\tau^2}^{n+1}}e^{-\frac{2^{2(j-1)} l(\Delta)^2}{t^2-\tau^2}}
        \\
        &\qquad + \frac{2^{jn}l(\Delta)^{3n/2+1}}{\sqrt{t^2-\tau^2}^{n+1}\tau^{n}}e^{-\frac{2^{2(j-1)} l(\Delta)^2}{t^2-\tau^2}}e^{-\frac{l(\Delta)^2}{\tau^2}}\Big].
    \end{align*}
    In the last step we used \(|\partial_tA|\leq \frac{C}{t}\) and that \(f\) is a Hardy-Sobolev atom which yields \(\Vert \nabla_{||} f\Vert_{L^2(\mathbb{R}^n)}\leq |\Delta|^{1/2-1/\beta} \Vert \nabla_{||} f\Vert_{L^\beta(\mathbb{R}^n)}\leq |\Delta|^{-1/2}\) by Hölder inequality.
    \medskip
    
    Now, we can estimate the integral over a truncated cone for a point \(y\in\partial\Omega\) with \(y\in 2^{j+1}\Delta\setminus2^j\Delta\). We have
    \begin{align}
        &\Big(\int_{\Gamma^{2^jl(\Delta)}(y)}\frac{|W_2f|^2}{t^{n+1}}dxdt\Big)^{1/2}\nonumber
        \\
        &=\Big(\int_0^{2^{j-1}l(\Delta)} \frac{1}{t^{n+1}}\int_{\Delta_{t/2}(y)}\Big|\int_0^t2\tau e^{-(t^2-\tau^2)L^t}\mathrm{div}(\partial_t\AP\nabla e^{-\tau^2L^t}f)d\tau\Big|^2dx dt\Big)^{1/2}\nonumber
        \\
        &\lesssim \Big(\int_0^{2^{j-1}l(\Delta)} \frac{1}{t^{n+3}}\Big(\int_0^t\frac{2^{jn/2}l(\Delta)^{n/2+1}}{\sqrt{t^2-\tau^2}\tau^{n}}e^{-\frac{2^{2(j-1)} l(\Delta)^2}{\tau^2}}\nonumber
        \\
        &\qquad +\frac{2^{jn/2}l(\Delta)^{n/2+1}}{\sqrt{t^2-\tau^2}^{n+1}}e^{-\frac{2^{2(j-1)} l(\Delta)^2}{t^2-\tau^2}} + \frac{2^{jn}l(\Delta)^{3n/2+1}}{\sqrt{t^2-\tau^2}^{n+1}\tau^{n}}e^{-\frac{2^{2(j-1)} l(\Delta)^2}{t^2-\tau^2}}e^{-\frac{l(\Delta)^2}{\tau^2}}d\tau\Big)^2 dt\Big)^{1/2}.\label{eq:InArgumentEquation1}
    \end{align}
    Here we use the observation that a function of type \((0,t]\to\mathbb{R}, \rho\mapsto \frac{1}{\rho^n}e^{-\frac{c}{\rho^2}}\) is maximized in \(\rho=t\) if \(t\leq c\) and in \(\rho=c\) if \(t\geq c\). Applying this observation to each of the expressions in \eqref{eq:InArgumentEquation1} gives
    \begin{align}
        \eqref{eq:InArgumentEquation1}&\lesssim \Big(\int_0^{2^{j-1}l(\Delta)} \Big(\frac{2^{jn}l(\Delta)^{n+2}}{t^{3n+3}} + \frac{2^{2jn}l(\Delta)^{3n+2}}{t^{5n+3}}\Big)\nonumber
        \\
        &\qquad\qquad \cdot e^{-2\frac{2^{2(j-1)} l(\Delta)^2}{t^2}}\Big(\int_0^t\frac{1}{\sqrt{t^2-\tau^2}}d\tau\Big)^2 dt\Big)^{1/2}\nonumber
        \\
        &\lesssim \Big(\int_0^{2^{j-1}l(\Delta)}\Big(\frac{1}{2^{j(2n+3)}l(\Delta)^{2n+1}} + \frac{1}{2^{j(3n+3)}l(\Delta)^{2n+1}}\Big) dt\Big)^{1/2}\nonumber
        \\
        &\lesssim \Big(\frac{1}{2^{j(n+1)}}+\frac{1}{2^{j(n+n/2+1)}}\Big)\frac{1}{l(\Delta)^{n}}
        \lesssim \frac{1}{2^{j(n+1)}}\frac{1}{l(\Delta)^{n}}.\label{eq:Proofshortconeheight}
    \end{align}

    On the other hand we have for the away part of the cone by Minkowski inequality
    
    \begin{align}
        &\Big(\int_{\Gamma(y)\setminus\{t\leq 2^{j-1}l(\Delta)\}}\frac{|W_2f|^2}{t^{n+1}}dxdt\Big)^{1/2}\nonumber
        \\
        &=\Big(\int_{2^{j-1}l(\Delta)}^\infty \frac{1}{t^{n+1}}\int_{\Delta_{t/2}(y)}\Big|\int_0^t2\tau e^{-(t^2-\tau^2)L^t}\mathrm{div}(\partial_t\AP\nabla e^{-\tau^2L^t}f)d\tau\Big|^2dx dt\Big)^{1/2}\nonumber
        \\
        &\lesssim\Big(\int_{2^{j-1}l(\Delta)}^\infty \frac{1}{t^{n+1}}\Big(\int_0^{t/2}\tau \Vert e^{-(t^2-\tau^2)L^t}\mathrm{div}(\partial_t\AP\nabla e^{-\tau^2L^t}f)\Vert_{L^2(\Delta_{t/2}(y))}d\tau\Big)^2 dt\Big)^{1/2}\nonumber
        \\
        &\qquad+ \Big(\int_{2^{j-1}l(\Delta)}^\infty \frac{1}{t^{n+1}}\Big(\int_{t/2}^t\tau \Vert e^{-(t^2-\tau^2)L^t}\mathrm{div}(\partial_t\AP\nabla e^{-\tau^2L^t}f)\Vert_{L^2(\Delta_{t/2}(y))}d\tau\Big)^2 dt\Big)^{1/2}.\label{eq:NonlocalConeEstimateW2}
    \end{align}
    We split the inner integral into two integrals, one over small \(\tau\) and one over large \(\tau\).

    We now look at the inner \(L^2\) norm in just the \(x\) components. For \(t/2\leq \tau\leq t\) we have by \refprop{prop:L2boundednessOfSemigroupOperators} and Hölder's inequality
    \begin{align*}
        &\tau\Vert e^{-(t^2-\tau^2)L^t}\mathrm{div}(\partial_t\AP\nabla e^{-\tau^2L^t}f)\Vert_{L^2(\Delta_{t/2}(y))}
        \\
        &\qquad\lesssim \frac{\tau}{\sqrt{t^2-\tau^2}}\Vert\partial_t A\Vert_{\infty}\Vert\nabla e^{-\tau^2L^t}f\Vert_{L^2(\mathbb{R}^n)}
        \lesssim \frac{\tau}{t\sqrt{t^2-\tau^2}\tau^{n+1}}\Vert f\Vert_{L^1(\mathbb{R}^n)}
        \\
        &\qquad\lesssim  \frac{1}{t\sqrt{t^2-\tau^2}t^{n}}\Vert f\Vert_{L^1(\mathbb{R}^n)}
        \lesssim  \frac{l(\Delta)^{n/2+1}}{t\sqrt{t^2-\tau^2}t^{n}}\Vert \nabla_{||}f\Vert_{L^2(\mathbb{R}^n)}
        \lesssim  \frac{l(\Delta)}{t\sqrt{t^2-\tau^2}t^{n/2}}.
    \end{align*}

    For \(0\leq \tau\leq t/2\) we proceed with a dualising function \(g\in L^2(\Delta_{t/2}(y))\) and the same ideas as in the local cone part to get
    \begin{align*}
        &\tau\Vert e^{-(t^2-\tau^2)L^t}\mathrm{div}(\partial_t\AP\nabla e^{-\tau^2L^t}f)\Vert_{L^2(\Delta_{t/2}(y))}
        \\
        &\lesssim \tau\Vert\partial_t A\Vert_{\infty}\Big(\Vert \nabla e^{-(t^2-\tau^2)L^t}g \Vert_{L^2(2\Delta)}\Vert\nabla e^{-\tau^2L^t}f\Vert_{L^2(2\Delta)}
        \\
        &\qquad\qquad + \Vert \nabla e^{-(t^2-\tau^2)L^t}g \Vert_{L^2(\mathbb{R}^n)}\Vert\nabla e^{-\tau^2L^t}f\Vert_{L^2(\mathbb{R}^n\setminus 2\Delta)}\Big)
        \\
        &\lesssim\frac{\tau}{t} \Big(\frac{l(\Delta)^{n/2}}{\sqrt{t^2-\tau^2}^{n+1}}\Vert g\Vert_{L^1(\Delta_{t/2}(y))}\Vert \nabla_{||} f\Vert_{L^2(\mathbb{R}^n)}
        \\
        &\qquad\qquad+ \frac{1}{\sqrt{t^2-\tau^2}\tau^{n+1}}\Vert g \Vert_{L^2}\Vert f\Vert_{L^1(\Delta)}\big(\int_{l(\Delta)}^\infty r^{n-1}e^{-\frac{r^2}{\tau^2}}dr\big)^{1/2}\Big)
        \\
        &\lesssim \frac{\tau}{t}\Big(\frac{l(\Delta)^{n/2}}{\sqrt{t^2-\tau^2}^{n+1}}\Vert g\Vert_{L^1(\Delta_{t/2}(y))}\Vert \nabla_{||} f\Vert_{L^2(\mathbb{R}^n)}
        \\
        &\qquad\qquad+ \frac{l(\Delta)^{n/2}}{\sqrt{t^2-\tau^2}\tau^{n+1}}\Vert g \Vert_{L^2}\Vert f\Vert_{L^1(\Delta)}e^{-c\frac{l(\Delta)^2}{\tau^2}}\Big)
        \\
        &\lesssim \frac{\tau}{t}\Big(\frac{l(\Delta)^{n/2}t^{n/2}}{\sqrt{t^2-\tau^2}^{n+1}}\Vert g\Vert_{L^2}\Vert \nabla_{||} f\Vert_{L^2(\mathbb{R}^n)} + \frac{l(\Delta)^{n+1}}{\sqrt{t^2-\tau^2}\tau^{n+1}}\Vert g \Vert_{L^2}\Vert \nabla_{||}f\Vert_{L^2(\Delta)}e^{-c\frac{l(\Delta)^2}{\tau^2}}\Big)
        \\
        &\lesssim \frac{\tau}{t}\Big(\frac{l(\Delta)^{n/2}t^{n/2}}{\sqrt{t^2-\tau^2}^{n+1}}\Vert \nabla_{||} f\Vert_{L^2(\mathbb{R}^n)} + \frac{l(\Delta)^{n+1}}{\sqrt{t^2-\tau^2}t^{n+1}}\Vert \nabla_{||}f\Vert_{L^2(\Delta)}\Big)
        \\
        &\lesssim  \frac{l(\Delta)^{n/2+1}}{\sqrt{t^2-\tau^2} t^{n/2+1}}\Vert\nabla_{||} f\Vert_{L^2(\mathbb{R}^n)}\lesssim \frac{l(\Delta)}{\sqrt{t^2-\tau^2} t^{n/2+1}}.
    \end{align*}
    Hence we obtained the same bound for both small and large \(\tau\), and we can continue \eqref{eq:NonlocalConeEstimateW2} with
    \begin{align*}
        &\Big(\int_{2^{j-1}l(\Delta)}^\infty \frac{1}{t^{n+1}}\Big(\int_0^{t/2}\tau \Vert e^{-(t^2-\tau^2)L^t}\mathrm{div}(\partial_t\AP\nabla e^{-\tau^2L^t}f)\Vert_{L^2(\Delta_{t/2}(y))}d\tau\Big)^2 dt\Big)^{1/2}
        \\
        &\qquad+ \Big(\int_{2^{j-1}l(\Delta)}^\infty \frac{1}{t^{n+1}}\Big(\int_{t/2}^t\tau \Vert e^{-(t^2-\tau^2)L^t}\mathrm{div}(\partial_t\AP\nabla e^{-\tau^2L^t}f)\Vert_{L^2(\Delta_{t/2}(y))}d\tau\Big)^2 dt\Big)^{1/2}
        \\
        &\lesssim\Big(\int_{2^{j-1}l(\Delta)}^\infty \frac{l(\Delta)^{2}}{t^{2n+3}}\Big(\int_0^{t}\frac{1}{\sqrt{t^2-\tau^2}} d\tau\Big)^2 dt\Big)^{1/2}
        \lesssim\Big(\int_{2^{j-1}l(\Delta)}^\infty \frac{l(\Delta)^{2}}{t^{2n+3}} dt\Big)^{1/2}
        \\
        &\lesssim \frac{1}{2^{j(n+1)}l(\Delta)^{n}}.
    \end{align*}
    Together with \eqref{eq:Proofshortconeheight} this yields
    \begin{align*}
        \Vert \mathcal{A}(W_2f)\Vert_{L^1(\mathbb{R}^n\setminus 5\Delta)}&\lesssim \sum_{j\geq 2}\int_{2^{j+1}\Delta\setminus 2^j\Delta} \Big(\int_{\Gamma^{2^{j}l(\Delta)}(y)}\frac{|W_2(f)|^2}{t^{n+1}}dxdt\Big)^{1/2}
        \\
        &\qquad\qquad + \Big(\int_{\Gamma(y)\setminus\{t\leq 2^{j-1}l(\Delta)\}}\frac{|W_2(f)|^2}{t^{n+1}}dxdt\Big)^{1/2}dy
        \\
        &\lesssim \sum_{j\geq 2} \big(2^{j+1}l(\Delta)\big)^n\frac{1}{2^{j(n+1/2)}l(\Delta)^n}\leq C.
    \end{align*}

    Lastly, \eqref{item:HardySquareBoundNabla} follows from \eqref{item:HardySquareBound} and Cacciopolli type inequality \refprop{prop:CacciopolliTypeInequality} like previously in \eqref{item:PointwiseLocalSquareBoundNablaW2} of \reflemma{lemma:PointwiseCarlesonFunctionBoundW_2}.
\end{proof}

\section{Area function bounds of \(W_1f\)}\label{section:W1}

As we saw in the proof of \reflemma{lemma:HardyNormBoundsW_2}, we would like to split off the local part when proving the Hardy-Sobolev to \(L^1\) bound for \(W_1\). Therefore, we used \(L^\beta-L^\beta\) boundedness of the area function of \(W_2\) (see \reflemma{lemma:HardyNormBoundsW_2beta}). For \(W_1\) we would like to proceed similarly. Therefore, the following proposition establishes \(L^2-L^2\) boundedness of the area functions of \(W_1\). In \cite[Lemma 7.9]{ulmer_solvability_2026}, this bound is proved in a slightly more difficult setting. There, the operator \(W_1\) is applied on a family of functions $\phi^t$ that also depend on $t$, i.e. the $L^2$ area function bound applies to $tL^te^{-t^2L^t}\phi^t$. The $\phi^t$ are asked to have certain uniform bounds. In our case, $W_1f=tL^te^{-t^2L^t}f$, where $f$ is independent of $t$. We note that the same proof for \(\phi^t\) also applies when $\phi^t$ is replaced by \(f\). For convenience, we provide the proof of \cite[Lemma 7.9]{ulmer_solvability_2026} in a slightly more general setting, that also applies here, in the appendix as \reflemma{lemma:L2AreaBoundGeneralCase}. Hence we have the following area function bounds:

\begin{prop}\label{prop:L^2SquareFctBoundW_1}
Assume \eqref{cond:mixedL1LInftyCond}. For \(f\in W^{1,2}(\mathbb{R}^n)\) it holds that
\begin{align}
    \Vert \mathcal{A}(W_1f)\Vert_{L^2(\mathbb{R}^n)}\lesssim \Vert \nabla f\Vert_{L^2(\mathbb{R}^n)}\label{eq:L2BoundAreaW1}
\end{align}
and
\begin{align}
    \Vert \mathcal{A}(t\nabla_{||} W_1f)\Vert_{L^2(\mathbb{R}^n)}\lesssim \Vert \nabla f\Vert_{L^2(\mathbb{R}^n)}.\label{eq:L2BoundAreaNablaW1}
\end{align}
\end{prop}

Let us note that this area function bound is proven assuming the mixed $L^1-L^\infty$ condition, \eqref{cond:mixedL1LInftyCond}, instead of the Carleson condition on $\partial_t A$, \eqref{cond:ExplanationForCarlesonCondition}. Since this $L^2$ area function bound will be used in both the proof of the Hardy-Sobolev to $L^1$ bound (\reflemma{lemma:HardyNormBoundsW_1}) and the $L^\infty$ Carleson function bound (\reflemma{lemma:CarlesonFunctionBoundW_1}), these lemmas can only be proven assuming the mixed $L^1-L^\infty$ condition \eqref{cond:mixedL1LInftyCond}.

\subsection{Carleson function bound (Proof of \reflemma{lemma:CarlesonFunctionBoundW_1})}

    The proof of \eqref{item:LocalSquareBoundW1} essentially only uses off-diagonal estimates, \refprop{prop:L2boundednessOfSemigroupOperators}, and the \(L^2-L^2\) boundedness of the area functions, \refprop{prop:L^2SquareFctBoundW_1}.
    
    Let us fix a boundary ball $\Delta\subset\partial\Omega$.
    We introduce a smooth cut-off function \(\eta\in C_0^\infty(3\Delta)\) with \(\eta\equiv 1\) on \(2\Delta\) and \(|\nabla_{||} \eta|\lesssim \frac{1}{l(\Delta)}\) and split the data \(f\) into a local and far-away part. We obtain
    \begin{align*}
        &\int_{T(\Delta)}\frac{|W_1f(x,t)|^2}{t}dxdt=\int_{T(\Delta)}\frac{|W_1(f-(f)_{3\Delta})(x,t)|^2}{t}dxdt
        \\
        &\qquad \leq \int_{T(\Delta)}\frac{|W_1((f-(f)_{3\Delta})\eta)(x,t)|^2}{t}dxdt
        \\
        &\qquad \qquad + \int_{T(\Delta)}\frac{|W_1((f-(f)_{3\Delta})(1-\eta))(x,t)|^2}{t}dxdt.
    \end{align*}
    By \(L^2-L^2\) boundedness of the area functions \refprop{prop:L^2SquareFctBoundW_1} and Poincar\'{e}'s inequality, the first integral is bounded by 
    \[\Vert\nabla_{||}( (f-(f)_{3\Delta})\eta)\Vert_{L^2(\mathbb{R}^n)}\lesssim \frac{1}{l(\Delta)}\Vert (f-(f)_{3\Delta}) \Vert_{L^2(3\Delta)} + \Vert \nabla_{||} f\Vert_{L^2(3\Delta)}\lesssim \Vert \nabla_{||} f\Vert_{L^2(3\Delta)}.\]
    Thus we only need to deal with the second integral, the away part.

    \smallskip
    We have for a fixed \(t\in (0,l(\Delta))\) with off-diagonal estimates \refprop{prop:L2boundednessOfSemigroupOperators} for \(W_1g(x,t)=tL^t e^{-t^2 L^t}g(x)\)
    \begin{align*}
        &\Vert W_1 ((f-(f)_{3\Delta})(1-\eta))\Vert_{L^2(\Delta)}
        \\
        &\lesssim \Vert W_1 (\chi_{4\Delta}(f-(f)_{3\Delta})(1-\eta))\Vert_{L^2(\Delta)}+\sum_{k\geq 2}\Vert W_1 (\chi_{2^{k+1}\Delta\setminus 2^{k}\Delta}(f-(f)_{3\Delta}))\Vert_{L^2(\Delta)}
        \\
        &\lesssim \frac{1}{t}e^{-c\frac{l(\Delta)^2}{t^2}}\Vert f-(f)_{3\Delta} \Vert_{L^2(4\Delta\setminus 2\Delta)} +  \sum_{k\geq 2}\frac{1}{t}e^{-c\frac{2^{2k}l(\Delta)^2}{t^2}}\Vert f-(f)_{3\Delta}\Vert_{L^2(2^{k+1}\Delta\setminus 2^k\Delta)}.
    \end{align*}
    As in \eqref{observationf-MVonLargerDomain}, Poincar\'{e}'s inequality yields
    \begin{align*}\Vert f-(f)_{3\Delta}\Vert_{L^2(2^{k+1}\Delta\setminus 2^k\Delta)}&\lesssim (k+1)2^{k+1}l(\Delta)\Vert \nabla_{||}f\Vert_{L^2(\mathbb{R}^n)}
    \\
    &\lesssim 2^{(k+1)(\frac{n}{2}+2)}l(\Delta)^{\frac{n}{2}+1}\Vert \nabla_{||} f\Vert_{L^\infty(\mathbb{R}^n)}.
    \end{align*}

    Hence, we obtain
    \begin{align*}
        &\frac{1}{t}e^{-c\frac{l(\Delta)^2}{t^2}}\Vert f-(f)_{3\Delta} \Vert_{L^2(4\Delta\setminus 2\Delta)} +  \sum_{k\geq 2}\frac{1}{t}e^{-c\frac{2^{2k}l(\Delta)^2}{t^2}}\Vert f-(f)_{3\Delta}\Vert_{L^2(2^{k+1}\Delta\setminus 2^k\Delta)}
        \\
        &\lesssim \frac{l(\Delta)^{\frac{n}{2}+1}}{t}\Big(e^{-c\frac{l(\Delta)^2}{t^2}} + \sum_{k\geq 2} (k+1)2^{(k+1)(\frac{n}{2}+1)}e^{-c\frac{2^{2k}l(\Delta)^2}{t^2}}\Big)\Vert \nabla_{||} f\Vert_{\infty}
        \\
        &\lesssim \frac{l(\Delta)^{\frac{n}{2}+1}}{t}\Big(\sum_{k\geq 1} (k+1)2^{(k+1)(\frac{n}{2}+1)}e^{-c\frac{2^{2k}l(\Delta)^2}{t^2}}\Big)\Vert \nabla_{||} f\Vert_{\infty}
        \\
        &\lesssim \frac{l(\Delta)^{\frac{n}{2}+1}}{t}\Big(\sum_{k\geq 1} 2^{(k+1)(\frac{n}{2}+2)}e^{-c\frac{2^{2k}l(\Delta)^2}{t^2}}\Big)\Vert \nabla_{||} f\Vert_{\infty}.
    \end{align*}

    With this in hand, we return now to 
    \begin{align*}
    &\int_{T(\Delta)}\frac{|W_1((f-(f)_{3\Delta})(1-\eta))(x,t)|^2}{t}dxdt
    \\
    &\qquad\lesssim \int_0^{l(\Delta)}\frac{l(\Delta)^{n+2}}{t^3}\Big(\sum_{k\geq 1} 2^{(k+1)(\frac{n}{2}+2)}e^{-c\frac{2^{2k}l(\Delta)^2}{t^2}}\Big)^2\Vert \nabla_{||} f\Vert_{\infty}^2dt
    \\
    &\qquad \lesssim l(\Delta)^{n+2}\Vert \nabla_{||} f\Vert_{\infty}^2\int_0^{l(\Delta)}\Big(\sum_{k\geq 1} \frac{1}{t^{3/2}}2^{(k+1)(\frac{n}{2}+2)}e^{-c\frac{2^{2k}l(\Delta)^2}{t^2}}\Big)^2dt.
    \end{align*}

    Similar to previous arguments in \reflemma{lemma:HardyNormBoundsW_2}, we can study the function \(t\mapsto \frac{1}{t^{3/2}}e^{-c\frac{2^{2k}l(\Delta)^2}{t^2}}\) for \(t\in (0,l(\Delta)]\) and observe that it attains its maximum if \(t=l(\Delta)\), whence
    \[\sum_{k\geq 1} \frac{1}{t^{3/2}}2^{(k+1)(\frac{n}{2}+2)}e^{-c\frac{2^{2k}l(\Delta)^2}{t^2}}\lesssim \frac{1}{l(\Delta)^{3/2}}\sum_{k\geq 1} 2^{(k+1)(\frac{n}{2}+2)}e^{-c2^{2k}}.\]
    We consider the sum as Riemann sums of the corresponding integral
    \[\int_0^\infty x^{\frac{n}{2}+1}e^{-cx^2}dx,\]
    which converges. Hence we obtain in total
    
    \begin{align*}
    &\int_{T(\Delta)}\frac{|W_1((f-(f)_{3\Delta})(1-\eta))(x,t)|^2}{t}dxdt
    \\
    &\qquad \lesssim l(\Delta)^{n+2}\Vert \nabla_{||} f\Vert_{\infty}^2\int_0^{l(\Delta)}\frac{1}{l(\Delta)^3}dt=|\Delta| \Vert \nabla_{||} f\Vert_{\infty}^2.
    \end{align*}

\subsection{Hardy-Sobolev bounds (Proof of \reflemma{lemma:HardyNormBoundsW_1})}

    First, we note that it is enough to show
    \[\Vert\mathcal{A}(W_1f)\Vert_{L^1(\mathbb{R}^n)}+\Vert\mathcal{A}(t\nabla_{||} W_1f)\Vert_{L^1(\mathbb{R}^n)}\leq C\]
    for all homogeneous Hardy-Sobolev \(1/2\)-atoms \(f\) associated with \(\Delta\), whence we assume that \(f\) is such an atom going forward.
    
    We begin with showing \eqref{item:HardySquareBoundW1}. We split the integral into a local and a far away part 
    \begin{align*}
        \Vert\mathcal{A}(W_1f)\Vert_{L^1(\mathbb{R}^n)}=\Vert\mathcal{A}(W_1f)\Vert_{L^1(\mathbb{R}^n\setminus 5\Delta)} + \Vert\mathcal{A}(W_1f)\Vert_{L^1(5\Delta)}.
    \end{align*}
    For the local part we have by Hölder's inequality, \refprop{prop:L^2SquareFctBoundW_1}, and the properties of the Hardy-Sobolev space that
    \begin{align*}
        \Vert\mathcal{A}(W_1f)\Vert_{L^1(5\Delta)}\lesssim \Vert\mathcal{A}(W_1f)\Vert_{L^2(5\Delta)}|\Delta|^{1/2}\lesssim \Vert\nabla_{||} f\Vert_{L^2(\mathbb{R}^n)}|\Delta|^{1/2}\lesssim 1.
    \end{align*}
    For the away part, 
    we make use of observation \eqref{eq:CasesForL^1BoundofTermInProof}. We divide \(\mathbb{R}^n\) into annuli \(2^{j+1}\Delta\setminus 2^j\Delta\) for \(j\geq 2\). If we have \(y\in 2^{j+1}\Delta\setminus 2^j\Delta\) then
    \begin{align*}
        &\int_{\Gamma(y)}\frac{|W_1f|^2}{t^{n+1}}dxdt=\int_0^\infty \frac{1}{t^{n+1}}\int_{\Delta_{t/2}(y)}|tL^te^{-t^2L^t}f(x,t)|^2dx dt
        \\
        &\qquad=\int_0^{2^{j-1}l(\Delta)} \frac{e^{-2\frac{|2^{j}l(\Delta)-\frac{t}{2}|^2}{t^2}}}{t^{2n+3}}\Vert f\Vert_{L^1(\mathbb{R}^n)}^2dx dt + \int_{2^{j-1}l(\Delta)}^\infty \frac{1}{t^{2n+3}}\Vert f\Vert_{L^1(\mathbb{R}^n)}^2 dt
        \\
        &\qquad\lesssim\int_0^{2^{j-1}l(\Delta)} \frac{l(\Delta)^{2}e^{-\frac{2^{2(j-1)}l(\Delta)^2}{t^2}}}{t^{2n+3}}dx dt
         + \int_{2^{j-1}l(\Delta)}^\infty \frac{l(\Delta)^{2}}{t^{2n+3}} dt
        \\
        &\qquad \lesssim \frac{2^{-j(2n+2)}}{l(\Delta)^{2n}}.
    \end{align*}
    Hence 
    \begin{align*}
        \Vert\mathcal{A}(W_1f)\Vert_{L^1(\mathbb{R}^n\setminus 5\Delta)}&\lesssim \sum_{j\geq 2}\int_{2^{j+1}\Delta\setminus 2^j\Delta}\Big(\int_{\Gamma(y)}\frac{|W_1f|^2}{t^{n+1}}dxdt\Big)^{1/2}dy
        \\
        &\lesssim \sum_{j\geq 2}\frac{2^{jn}l(\Delta)^n}{2^{j(n+1)}l(\Delta)^n}\leq C,
    \end{align*}
    and \eqref{item:HardySquareBoundW1} follows. Next, \eqref{item:HardySquareBoundNablaW1} follows from \eqref{item:HardySquareBoundW1} and the Caccioppoli inequality \refprop{prop:CacciopolliTypeInequality}. 
    \smallskip

\section{The regularity problem under the $L^1-L^\infty$ condition on $\partial_t A$}\label{section:MainProof}

Here, We present the proof of \refthm{MAINTHM} using the $L^p$ area function bounds in \reflemma{lemma:AreaFctBoundsInLp}. Furthermore, we are going to need two more auxiliary lemmas, \reflemma{lemma:CarelsonFctBoundPs} and \reflemma{lemma:W2vcBoundebyN(v)}, which are proved in Section~\ref{sec:Aux}, before we turn to the proof of \refthm{MAINTHM} in Section~\ref{sec:subMainProof}. Since both of them rely on the $L^2$ area function bound of $sL^se^{-s^2L^s}$, they can only be proven assuming the mixed $L^1-L^\infty$ condition \eqref{cond:mixedL1LInftyCond}. At time of writing, it is unclear whether these lemmas can also be proven assuming merely \eqref{cond:ExplanationForCarlesonCondition}, the Carleson condition on $\partial_tA$ .

\subsection{Auxiliary lemmas}\label{sec:Aux}

The following lemma stems from \cite[Lemma 4.1]{hofmann_regularity_2015}.
\begin{lemma}\label{lemma:R_tBoundedbynablav}
    Assume that a family of operators $\{R_t\}_{t>0}$ satisfies $L^2-L^2$ off-diagonal estimates and that $R_t1=0$ (in the sense of $L^2_{loc}(\mathbb{R}^n)$). Then for every $1<p<\infty$ there exists $C>0$ such that
    \[\Vert\mathcal{A}(R_t F)\Vert_{L^p}\leq C\Vert\mathcal{A}(t|\nabla_{||} F|)\Vert_{L^p}\]
    for every $F\in L_{loc}^2(\mathbb{R}^{n+1}_+)$ such that the right hand side is finite, i.e. such that $t|\nabla_{||} F|\in T^{p,2}(\mathbb{R}^{n+1}_+)$.
\end{lemma}

The following result for operators $\mathcal{P}_t:=e^{-t^2L^t}$ defined for a $t$-independent $L^t$ is often a byproduct of the $T(b)$ argument in the proof of the Kato estimate (see for instance \cite{Auscher_Kato}). However, in our case, we need a new proof. 
\begin{lemma}\label{lemma:CarelsonFctBoundPs}
    Assume that \eqref{cond:mixedL1LInftyCond} holds. Then, for $g\in (L^2_{loc}(\mathbb{R}^n))^n$, we have
    \begin{align}\mathcal{C}(s\mathcal{P}^*_s\mathrm{div}(g))(y)= \sup_{\Delta=\Delta(y,r)\subset\partial\Omega}\frac{1}{\sigma(\Delta)}\int_{T(\Delta)}|\mathcal{P}_s^*\mathrm{div}(g)|^2s\,dxds\lesssim \mathcal{M}[|g|^2](y).\label{eq:ProofCarlesonBound1}\end{align}
    As a consequence, if $g\in (L^\infty(\mathbb{R}^n))^n$, then
    \begin{align}\Vert\mathcal{C}(s\mathcal{P}^*_s\mathrm{div}(g))\Vert_{L^\infty}= \sup_{\Delta\subset\partial\Omega}\frac{1}{\sigma(\Delta)}\int_{T(\Delta)}|\mathcal{P}_s^*\mathrm{div}(g)|^2s\,dxds\lesssim \Vert g\Vert_{L^\infty}.\label{eq:ProofCarlesonBound2}\end{align}
\end{lemma}

\begin{proof}
    Let us fix $y\in\mathbb{R}^n$ and $\Delta=\Delta(x,r)\subset\partial\Omega$.
    We split into a local and an away part, i.e.
\begin{align*}
    \int_{T(\Delta)}|\mathcal{P}^*_{s}\mathrm{div}(g)|^2s dxds&=\int_{T(\Delta)}|\mathcal{P}^*_{s}\mathrm{div}(g)|^2s dxds
    \\
    &=\int_{T(\Delta)}|\mathcal{P}_{s}\mathrm{div}(g\chi_{3\Delta})|^2s dxds
    \\
    &\qquad+ \int_{T(\Delta)}|\mathcal{P}_{s}\mathrm{div}(g\chi_{\mathbb{R}^n\setminus 3\Delta})|^2s dxds
    \\
    &=I_1+I_2.
\end{align*}

For the local part \(I_1\), we use a Hodge decomposition of $g$: There exists a family of functions $\phi^s\in W^{1,2}_0(3\Delta)$ such that
\[L^s\phi^s=\mathrm{div}_{||}(g\chi_{3\Delta}).\]
This equality can be understood as distributional equality, i.e. the left and right hand sides are equal as elements of $W^{-1,2}(\mathbb{R}^n)$. By arguments as in Section~4 of \cite{ulmer_solvability_2026}, we obtain weak differentiability of \(\phi^s\) and \(\nabla\phi^s\) in \(s\) and 
\begin{equation}\label{eq:Estiamtesonphis}
    \begin{split}
    \Vert\nabla\phi^s\Vert_{L^2(5\Delta)}&\lesssim C(n,\lambda_0,\Lambda_0)\Vert g\Vert_{L^2(3\Delta)}, \\
    \qquad\textrm{and}\qquad \Vert\nabla\partial_s\phi^s\Vert_{L^2(5\Delta)}&\lesssim C(n,\lambda_0,\Lambda_0)\Vert\partial_sA\Vert_{L^\infty(5\Delta)}\Vert g\Vert_{L^2(3\Delta)}.
    \end{split}
\end{equation}
Since the semigroup operator can be defined on $W^{-1,2}(\mathbb{R}^n)$, we have that \(\mathcal{P}_{s}^*\mathrm{div}(g)=\mathcal{P}_{s}^*(L^s)^*\phi^s\). Using this observation, we get
\begin{align*}
    I_1&\leq\int_0^{l(\Delta)}\int_{\mathbb{R}^n}|\mathcal{P}_{s}^*(L^s)^*\phi^s|^2 s\,dxds\leq \Vert\mathcal{A}(s(L^s)^*\mathcal{P}_{s}^*\phi^s)\Vert_{L^2(\mathbb{R}^n)}^2. 
\end{align*}

Since \(\phi^s\) satisfies \eqref{eq:Estiamtesonphis}, The \(L^2\) area function bound \reflemma{lemma:L2AreaBoundGeneralCase} applies and we obtain 
\[I_1\lesssim \Vert\mathcal{A}(s(L^s)^*\mathcal{P}_{s}^*\phi^s)\Vert_{L^2(\mathbb{R}^n)}^2\lesssim \Vert g\Vert_{L^2(3\Delta)}^2\lesssim \sigma(\Delta)\inf_{x\in \Delta}\mathcal{M}[|g|^2](x).\]

For the away part \(I_2\) we have by the off-diagonal estimates \refprop{prop:L2boundednessOfSemigroupOperators}, that
\begin{align*}
    I_2&=\int_0^{l(\Delta)}\Big(\sum_{k\geq 1}\Vert\mathcal{P}_{s}\mathrm{div}(g\chi_{2^{k+1}\Delta\setminus2^k\Delta})\Vert_{L^2(\Delta)}\Big)^2 sds
    \\
    &=\int_0^{l(\Delta)}\Big(\sum_{k\geq 1}\frac{1}{s}e^{-\frac{2^{2k}l(\Delta)^2}{s^2}}\Vert g\Vert_{L^2(2^{k+1}\Delta\setminus2^k\Delta)}\Big)^2 sds
    \\
    &=\inf_{x\in\Delta}\mathcal{M}[|g|^2](x)\int_0^{l(\Delta)}\Big(\sum_{k\geq 1}\frac{2^{nk/2}l(\Delta)^{n/2}}{s}e^{-\frac{2^{2k}l(\Delta)^2}{s^2}}\Big)^2 sds.
\end{align*}
Since the function \(s\mapsto\frac{1}{s}e^{-\frac{2^{2k}l(\Delta)^2}{s^2}}\) is monotonically increasing on \((0,l(\Delta))\), we obtain
\begin{align*}
    I_2&\lesssim\inf_{x\in\Delta}\mathcal{M}[|g|^2](x)\int_0^{l(\Delta)}\Big(\sum_{k\geq 1}\frac{2^{nk/2}l(\Delta)^{n/2}}{l(\Delta)}e^{-\frac{2^{2k}l(\Delta)^2}{l(\Delta)^2}}\Big)^2 sds
    \\
    &\lesssim\inf_{x\in\Delta}\mathcal{M}[|g|^2](x)\int_0^{l(\Delta)}l(\Delta)^{n-2}\Big(\sum_{k\geq 1}2^{kn/2}e^{-2^{2k}}\Big)^2 tdt
    \\
    &\lesssim \inf_{x\in\Delta}\mathcal{M}[|g|^2](x)l(\Delta)^n\int_0^{l(\Delta)}\frac{t}{l(\Delta)^2}dt\lesssim l(\Delta)^n\inf_{x\in\Delta}\mathcal{M}[|g|^2](x).
\end{align*}

Combining the estimates for $I_1$ and $I_2$ gives
\[\mathcal{C}(s\mathcal{P}_s^*\mathrm{div}(g))(y)\lesssim\mathcal{M}[|g|^2](y),\]
which shows \eqref{eq:ProofCarlesonBound1}. If $g\in L^\infty$, this immediately also implies \eqref{eq:ProofCarlesonBound2}.

\end{proof}

The next lemma is a fundamental part of the proof of the \refthm{MAINTHM}.

\begin{lemma}\label{lemma:W2vcBoundebyN(v)}
Let $f\in \dot{W}^{1,p}(\mathbb{R}^n)$, $c\in (L^\infty(\mathbb{R}^{n+1}_+))^n$, and $v\in L^2_{loc}(\mathbb{R}^{n+1}_+)$. Assume furthermore that \eqref{cond:mixedL1LInftyCond} holds. Then there exists $C>0$ such that for every $R>0$
\begin{align}
    \int_{\mathbb{R}^n\times (0,R)}\nabla_{||} W_2f\cdot cv dxds&\leq C\Vert \mathcal{N}(v)\Vert_{L^{p'}}\Vert\nabla_{||} f\Vert_{L^p}.
\end{align}
\end{lemma}

\begin{proof}
    First, we split this integral into two regimes
    \begin{align*}
        & \int_\Omega\nabla_{||} W_2f\cdot cv dxds
        \\
        &=\int_0^R\int_{\mathbb{R}^n}\Big(\int_0^s2\tau\nabla e^{(s^2-\tau^2)L^s}\mathrm{div}(\partial_s\AP(\cdot,s)\nabla e^{-\tau^2L^s}f)(x)d\tau\Big)c(x,s)v(x,s)\,dxds
        \\
        &=\int_0^R\Big(\int_{\mathbb{R}^n}\int_0^{s/4}2\tau\partial_s\AP(x,s)\nabla e^{-\tau^2L^s}f(x)\cdot\nabla e^{-(s^2-\tau^2)(L^s)^*}\mathrm{div}(cv)(x)\,d\tau dx\Big)ds
        \\
        &\quad+\int_0^R\Big(\int_{\mathbb{R}^n}\int_{s/4}^s2\tau e^{-\frac{(s^2-\tau^2)}{2}L^s}\mathrm{div}(\partial_s\AP(x,s)\nabla e^{-\tau^2L^s}f)(x)\, e^{-\frac{(s^2-\tau^2)}{2}(L^s)^*}\mathrm{div}(cv)(x)\,d\tau dx\Big)ds
        \\
        &:=I+J.
    \end{align*}
    Here we used the adjoint semigroup operator.

    \underline{\textbf{Step 1:}} First, we deal with \(I\). Let us fix $s\in (0,\infty)$ and $\tau\in (0,s/4)$. Let $1<p_0<p$. Furthermore, let $k$ be the smallest integer, such that $s/4<2^{-k}$. Then we have
    \begin{align}
        &\int_{\mathbb{R}^n}2\tau\partial_s\AP(x,s)\nabla e^{-\tau^2L^s}f(x)\cdot\nabla e^{-(s^2-\tau^2)(L^s)^*}\mathrm{div}(cv)(x)\, dx\nonumber
        \\
        &\lesssim \sum_{Q\in\mathcal{D}_k(\mathbb{R}^n)}\int_{Q}\tau\partial_s\AP(x,s)\nabla e^{-\tau^2L^s}f(x)\cdot\nabla e^{-(s^2-\tau^2)(L^s)^*}\mathrm{div}(cv)(x)\, dx\nonumber
        \\
        &\lesssim \sum_{Q\in\mathcal{D}_k(\mathbb{R}^n)}|Q|\Big(\fint_{Q}|\partial_s\AP(x,s)|^rdx\Big)^{\frac{1}{r}}\Big(\fint_{Q}|\nabla e^{-\tau^2L^s}f|^{p_0}dx\Big)^{\frac{1}{p_0}}\nonumber
        \\
        &\hspace{5cm}\cdot\Big(\fint_{Q}|\nabla e^{-(s^2-\tau^2)(L^s)^*}\mathrm{div}(cv)|^2dx\Big)^{\frac{1}{2}},\label{eq:ProofwsL1no-2}
    \end{align}
    where $r>2$ is chosen such that $\frac{1}{r}+\frac{1}{2}+\frac{1}{p_0}=1$. We now deal with the second and third integral separately. First, the operator $(s^2-\tau^2)\nabla e^{-(s^2-\tau^2)(L^s)^*}\mathrm{div}$ can be sees as the operator $\sqrt{s^2-\tau^2}\nabla e^{-\frac{(s^2-\tau^2)}{2}(L^s)^*}$ applied to the operator $\sqrt{s^2-\tau^2}e^{-\frac{(s^2-\tau^2)}{2}(L^s)^*}\mathrm{div}$. Since both of them satisfy off-diagonal estimates (see \refprop{prop:L2boundednessOfSemigroupOperators}), the operator $(s^2-\tau^2)\nabla e^{-(s^2-\tau^2)(L^s)^*}\mathrm{div}$ also satisfies \(L^2-L^2\) off-diagonal estimates. Hence,
    \begin{align}
        &\Vert\nabla e^{-(s^2-\tau^2)(L^s)^*}\mathrm{div}(cv)\Vert_{L^2(Q)}\nonumber
        \\
        &\leq \Vert\nabla e^{-(s^2-\tau^2)(L^s)^*}\mathrm{div}(cv\chi_{2Q})\Vert_{L^2(Q)}+\sum_{l\geq 1}\Vert\nabla e^{-(s^2-\tau^2)(L^s)^*}\mathrm{div}(cv\chi_{2^{l+1}Q\setminus2^{l}Q})\Vert_{L^2(Q)}\nonumber
        \\
        &\lesssim \frac{\Vert cv\Vert_{L^2(2Q)}}{s^2-\tau^2}+\sum_{l\geq 1}\frac{e^{-c\frac{2^{2l}l(Q)^2}{s^2-\tau^2}}}{s^2-\tau^2}\Vert cv\Vert_{L^2(2^lQ)}.\label{eq:ProofwsL1no1}
    \end{align}
    We use that 
    \[\Vert cv\Vert_{L^2(2^lQ)}\lesssim|Q|^{1/2}2^{ln/2}\big(\fint_{2^{l}Q}|v|^2 dx\big)^{1/2}\lesssim |Q|^{1/2}2^{ln/2}\inf_{x\in Q}\mathcal{M}[\mathcal{N}(v)^2]^{1/2}(x),\]
    and since we integrate over $0\leq\tau\leq s/4$ in $I$, we have $s^2-\tau^2\approx s^2\approx l(Q)^2$. Thus, we can continue bounding \eqref{eq:ProofwsL1no1} with
    \begin{align}
    \eqref{eq:ProofwsL1no1}&\lesssim \frac{|Q|^{1/2}}{s^2}\inf_{x\in Q}\mathcal{M}[\mathcal{N}(v)^2]^{1/2}(x) + \sum_{l\geq 1}\frac{|Q|^{1/2}2^{ln/2}e^{-c2^{2l}}}{s^2}\inf_{x\in Q}\mathcal{M}[\mathcal{N}(v)^2]^{1/2}(x)\nonumber
    \\
    &\lesssim \frac{|Q|^{1/2}}{s^2}\inf_{x\in Q}\mathcal{M}[\mathcal{N}(v)^2]^{1/2}(x),\label{eq:ProofwsL1no2}
    \end{align}
    which gives in total 
    \begin{align}
        \Big(\fint_{Q}|\nabla e^{-(s^2-\tau^2)(L^s)^*}\mathrm{div}(cv)|^2dx\Big)^{\frac{1}{2}}\lesssim \frac{1}{s^2}\inf_{x\in Q}\mathcal{M}[\mathcal{N}(v)^2]^{1/2}(x)\label{eq:ProofwsL1no2.5}
    \end{align}

    For the second integral expression in \eqref{eq:ProofwsL1no-2}, we split the integral using smooth cut-off functions. Let $\eta_Q\in C^\infty_0(Q)$ with $0\leq \eta_Q\leq 1$ such that $\eta_Q\equiv 1$ on $\frac{5}{8}Q$, and $\eta_Q\equiv 0$ on $\mathbb{R}^n\setminus\frac{7}{8}Q$ with $|\nabla \eta_Q|\lesssim \frac{1}{l(Q)}$. Furthermore, we set $\tilde{\eta}_{2^{l+1}Q\setminus2^lQ}:=\eta_{2^{l+1}Q}-\eta_{2^{l}Q}$. Then 
    \begin{align*}
        &\Big(\fint_{Q}|\nabla e^{-\tau^2L^s}f|^{p_0}dx\Big)^{\frac{1}{p_0}}=\Big(\fint_{Q}|\nabla e^{-\tau^2L^s}(f-(f)_{2Q})|^{p_0}dx\Big)^{\frac{1}{p_0}}
        \\
        &\lesssim \Big(\fint_{Q}|\nabla e^{-\tau^2L^s}((f-(f)_{2Q})\eta_{2Q})|^{p_0}dx\Big)^{\frac{1}{p_0}} + \sum_{l\geq 1}\Big(\fint_{Q}|\nabla e^{-\tau^2L^s}((f-(f)_{2Q})\tilde{\eta}_{2^{l+1}Q\setminus2^lQ}|^{p_0}dx\Big)^{\frac{1}{p_0}}
        \\
        &=:I_1+I_2.
    \end{align*}
    For $I_2$, we use the $L^{p_0}-L^2$ off-diagonal estimates \refprop{prop:L2boundednessOfSemigroupOperators} and obtain
    \begin{align}
        I_2&\lesssim \sum_{l\geq 1}\Big(\fint_{Q}|\nabla e^{-\tau^2L^s}((f-(f)_{2Q})\tilde{\eta}_{2^{l+1}Q\setminus2^lQ})|^{2}dx\Big)^{\frac{1}{2}}\nonumber
        \\
        &\lesssim \sum_{l\geq 1}\frac{2^{ln/p_0}}{\tau}e^{-c\frac{2^{2l}l(Q)^2}{\tau^2}}\Big(\fint_{2^{l+1}Q}|f-(f)_{2Q}|^{p_0}dx\Big)^{\frac{1}{p_0}}.\label{eq:ProofwsL1no3}
    \end{align}
    Since \(\tau\mapsto \frac{1}{\tau}e^{-c\frac{2^{2l}l(Q)}{\tau^2}}\) is monotonically increasing for $\tau\in (0,s)$, we get by Poincar\'{e}
    \begin{align*}
        \eqref{eq:ProofwsL1no3}&\lesssim \sum_{l\geq 1}\frac{2^{ln/p_0}}{s}e^{-c2^{2l}}\Big[\Big(\fint_{2^{l+1}Q}|f-(f)_{2^lQ}|^{p_0}dx\Big)^{\frac{1}{p_0}}+\sum_{m=1}^{l-1} \Big(\fint_{2^{l+1}Q}|(f)_{2^mQ}-(f)_{2^{m+1}Q}|^{p_0}dx\Big)^{\frac{1}{p_0}}\Big]
        \\
        &\lesssim  \sum_{l\geq 1}2^{ln/p_0}e^{-c2^{2l}} \Big(\fint_{2^{l+1}Q}|\nabla_{||} f|^{p_0}dx\Big)^{\frac{1}{p_0}}\lesssim \inf_{x\in Q}\mathcal{M}[|\nabla_{||} f|^{p_0}]^{1/p_0}(x).
    \end{align*}

    For \(I_1\), we can use the Kato estimate \refprop{prop:KatoConjecture} and \refprop{prop:L2boundednessOfSemigroupOperators}, and since $1<p_0<2$, we get
    \begin{align*}
        I_1&\lesssim|Q|^{-\frac{1}{p_0}}\Big(\int_{\mathbb{R}^n}|(L^s)^{1/2} e^{-\tau^2L^s}((f-(f)_{2Q})\eta_{2Q})|^{p_0}dx\Big)^{\frac{1}{p_0}}
        \\
        &\lesssim|Q|^{-\frac{1}{p_0}}\Big(\int_{\mathbb{R}^n}| e^{-\tau^2L^s}\big[(L^s)^{1/2}((f-(f)_{2Q})\eta_{2Q})\big]|^{p_0}dx\Big)^{\frac{1}{p_0}}
        \\
        &\lesssim|Q|^{-\frac{1}{p_0}}\Big(\int_{\mathbb{R}^n}|(L^s)^{1/2}((f-(f)_{2Q})\eta_{2Q})|^{p_0}dx\Big)^{\frac{1}{p_0}}
        \\
        &\lesssim|Q|^{-\frac{1}{p_0}}\Big(\int_{\mathbb{R}^n}|\nabla((f-(f)_{2Q})\eta_{2Q})|^{p_0}dx\Big)^{\frac{1}{p_0}}
        \\
        &\lesssim|Q|^{-\frac{1}{p_0}}\Big(\int_{\mathbb{R}^n}|\nabla_{||} f|^{p_0}dx\Big)^{\frac{1}{p_0}} + |Q|^{-\frac{1}{p_0}}\frac{1}{l(Q)}\Big(\int_{\mathbb{R}^n}|f-(f)_{2Q}|^{p_0}dx\Big)^{\frac{1}{p_0}}
    \end{align*}
    By Poincar\'{e}, we get
    \begin{align}
        I_1\lesssim \inf_{x\in Q} \mathcal{M}[|\nabla_{||} f|^{p_0}]^{1/p_0}(x).\label{eq:ProofwsL1no4}
    \end{align}

    Combining \eqref{eq:ProofwsL1no-2}, \eqref{eq:ProofwsL1no2.5}, \eqref{eq:ProofwsL1no3}, and \eqref{eq:ProofwsL1no4}, we get
    \begin{align*}
        I&\lesssim \int_0^R\Big(\int_0^{s/4}\frac{\tau}{s^2}\sum_{Q\in \mathcal{D}_k(\mathbb{R}^n)}\Vert\partial_sA(\cdot,s)\Vert_{L^\infty}|Q|\inf_{x\in Q}\mathcal{M}[|\nabla_{||} f|^{p_0}]^{1/p_0}(x)\inf_{x\in Q}\mathcal{M}[\mathcal{N}(v)^2]^{1/2}(x)\,d\tau\Big)ds
        \\
        &\lesssim \int_0^R \Vert\partial_sA(\cdot,s)\Vert_{L^\infty}\Big(\int_{\mathbb{R}^n}\mathcal{M}[|\nabla_{||} f|^{p_0}]^{1/p_0}\inf_{x\in Q}\mathcal{M}[\mathcal{N}(v)^2]^{1/2}dx\Big)ds
        \\
        &\lesssim \Big(\int_0^\infty \Vert\partial_sA(\cdot,s)\Vert_{L^\infty}ds\Big) \Vert \mathcal{M}[|\nabla_{||} f|^{p_0}]^{1/p_0}\Vert_{L^p}\Vert \mathcal{M}[\mathcal{N}(v)^2]^{1/2}\Vert_{L^{p'}}
        \\
        &\lesssim \Vert \nabla_{||} f\Vert_{L^p}\Vert \mathcal{N}(v)\Vert_{L^{p'}}.
    \end{align*}

    \underline{\textbf{Step 2:}} Next, let us bound $J$. Again, first, we fix $s$, and assume now that $k$ is the smallest integer such that $s<2^{-k}$. Then using the change of variables $\kappa=\sqrt{s^2-\tau^2}$ gives
    \begin{align}
        &\int_{\mathbb{R}^n}\int_{s/4}^s2\tau e^{\frac{-(s^2-\tau^2)}{2}L^s}\mathrm{div}(\partial_s\AP(x,s)\nabla e^{-\tau^2L^s}f)(x)\, e^{-\frac{(s^2-\tau^2)}{2}(L^s)^*}\mathrm{div}(cv)(x)\,d\tau dx\nonumber
        \\
        &=\int_{\mathbb{R}^n}\int_0^{\frac{\sqrt{15}}{4}s}2\kappa e^{-\frac{\kappa^2}{2}L^s}\mathrm{div}(\partial_s\AP(x,s)\nabla e^{-(s^2-\kappa^2)L^s}f)(x)\, e^{-\frac{\kappa^2}{2}(L^s)^*}\mathrm{div}(cv)(x)\,d\kappa dx\nonumber
        \\
        &\lesssim\sum_{Q\in\mathcal{D}_k(\mathbb{R}^n)}|Q|\Big(\fint_Q\int_{0}^{\frac{\sqrt{15}}{4}s}|e^{-\frac{\kappa^2}{2}(L^s)^*}\mathrm{div}(cv)|^2\kappa \,d\kappa dx\Big)^{1/2}\nonumber
        \\
        &\hspace{1.5cm}\cdot\Big(\fint_Q\int_{0}^{\frac{\sqrt{15}}{4}s}|e^{-\frac{\kappa^2}{2}L^s}\mathrm{div}(\partial_sA\nabla e^{-(s^2-\kappa^2)L^s}f)|^2\kappa \,d\kappa dx\Big)^{1/2}.\label{eq:ProofwsL1noJ1}
    \end{align}
    By \reflemma{lemma:CarelsonFctBoundPs}, we obtain for the first integral
    \begin{align}
        \Big(\fint_Q\int_{0}^{\frac{\sqrt{15}}{4}s}|e^{-\frac{\kappa^2}{2}(L^s)^*}\mathrm{div}(cv)|^2\kappa \,d\kappa dx\Big)^{1/2}\lesssim \inf_{x\in Q}\mathcal{M}[\mathcal{N}(cv)^2]^{1/2}(x).\label{eq:ProofwsL1noJ2}
    \end{align}
    Note that \reflemma{lemma:CarelsonFctBoundPs} is even more general, since the operator $L^s$ in the semigroup here is independent of the integration.

    Next, we split the second integral in \eqref{eq:ProofwsL1noJ1} into
    \begin{align*}
        &\Big(\fint_Q\int_{0}^{\frac{\sqrt{15}}{4}s}|e^{-\frac{\kappa^2}{2}L^s}\mathrm{div}(\partial_sA\nabla e^{-(s^2-\kappa^2)L^s}f)|^2\kappa \,d\kappa dx\Big)^{1/2}
        \\
        &\lesssim |Q|^{-1/2}\Big(\int_Q\int_{0}^{\frac{\sqrt{15}}{4}s}|e^{-\frac{\kappa^2}{2}L^s}\mathrm{div}(\chi_{2Q}\partial_sA\nabla e^{-(s^2-\kappa^2)L^s}f)|^2\kappa \,d\kappa dx\Big)^{1/2}
        \\
        &\qquad + |Q|^{-1/2}\Big(\int_Q\int_{0}^{\frac{\sqrt{15}}{4}s}|e^{-\frac{\kappa^2}{2}L^s}\mathrm{div}(\chi_{\mathbb{R}^n\setminus 2Q}\partial_sA\nabla e^{-(s^2-\kappa^2)L^s}f)|^2\kappa \,d\kappa dx\Big)^{1/2}
        \\
        &=:J_1+J_2
    \end{align*}
    For the away integral $J_2$ we use $L^2-L^2$ off-diagonal estimates \refprop{prop:L2boundednessOfSemigroupOperators} and obtain
    \begin{align*}
        J_2&\lesssim |Q|^{-1/2}\sum_{l\geq 1}\Big(\int_Q\int_{0}^{\frac{\sqrt{15}}{4}s}|e^{-\frac{\kappa^2}{2}L^s}\mathrm{div}(\chi_{2^{l+1}Q\setminus 2^lQ}\partial_sA\nabla e^{-(s^2-\kappa^2)L^s}f)|^2\kappa \,d\kappa dx\Big)^{1/2}
        \\
        &\lesssim |Q|^{-1/2}\sum_{l\geq 1}\Big(\int_{2^{l+1}Q}\int_{0}^{\frac{\sqrt{15}}{4}s}\frac{1}{\kappa}e^{-c\frac{2^{2l}l(Q)^2}{\kappa^2}}|\partial_sA\nabla e^{-(s^2-\kappa^2)L^s}f|^2 \,d\kappa dx\Big)^{1/2}
        \\
        &\lesssim\frac{\Vert\partial_sA(\cdot,s)\Vert_{L^\infty}}{|Q|^{1/2}}\Bigg[\sum_{l\geq 1}\Big(\int_{2^{l+1}Q}\int_{0}^{\frac{\sqrt{15}}{4}s}\frac{1}{\kappa}e^{-c\frac{2^{2l}l(Q)^2}{\kappa^2}}|\nabla e^{-(s^2-\kappa^2)L^s}(f\chi_{2^{l+1}Q})|^2 \,d\kappa dx\Big)^{1/2}
        \\
        &\qquad+\sum_{l,m\geq 1}\Big(\int_{2^{l+1}Q}\int_{0}^{\frac{\sqrt{15}}{4}s}\frac{1}{\kappa}e^{-c\frac{2^{2l}l(Q)^2}{\kappa^2}}|\nabla e^{-(s^2-\kappa^2)L^s}(f\chi_{2^{m+l+2}Q\setminus 2^{m+l+1}Q})|^2 \,d\kappa dx\Big)^{1/2}\Bigg].
    \end{align*}
    We let \(1<p_0<p\), and use $L^{p_0}-L^2$ boundedness and off-diagonal estimates \refprop{prop:L2boundednessOfSemigroupOperators} to get
    \begin{align*}
        J_2&\lesssim\frac{\Vert\partial_sA(\cdot,s)\Vert_{L^\infty}}{|Q|^{1/2}}\Bigg[\sum_{l\geq 1}\Big(\int_{0}^{\frac{\sqrt{15}}{4}s}\frac{1}{\kappa}e^{-c\frac{2^{2l}l(Q)^2}{\kappa^2}}\frac{1}{(s^2-\kappa^2)^{n(\frac{2}{p_0}-1)+2}}\Vert f\Vert^2_{L^{p_0}(2^{l+2}Q)} \,d\kappa\Big)^{1/2}
        \\
        &\qquad+\sum_{l,m\geq 1}\Big(\int_{0}^{\frac{\sqrt{15}}{4}s}\frac{1}{\kappa}e^{-c\frac{2^{2l}l(Q)^2}{\kappa^2}}\frac{1}{(s^2-\kappa^2)^{n(\frac{2}{p_0}-1)+2}}e^{-c\frac{2^{2(l+m)}l(Q)}{s^2-\kappa^2}}\Vert f\Vert^2_{L^{p_0}(2^{m+l+2}Q)} \,d\kappa\Big)^{1/2}\Bigg].
    \end{align*}
    Poincar\'{e} inequality implies 
    \[\Vert f\Vert^2_{L^{p_0}(2^{m+l+2}Q)}\lesssim 2^{(2n/p_0+2)(m+l)}l(Q)^{2n/p_0+2}\inf_{x\in Q}\mathcal{M}[|\nabla_{||} f|^{p_0}]^\frac{2}{p_0}(x),\]
    and hence
    \begin{align*}
        J_2&\lesssim\frac{\Vert\partial_sA(\cdot,s)\Vert_{L^\infty}}{|Q|^{1/2}}\Bigg[\sum_{l\geq 1}\Big(\int_{0}^{\frac{\sqrt{15}}{4}s}\frac{1}{\kappa}e^{-c\frac{2^{2l}l(Q)}{\kappa^2}}\frac{l(Q)^{2n/p_0+2}}{(s^2-\kappa^2)^{n(\frac{2}{p_0}-1)+2}}\inf_{x\in Q}\mathcal{M}[|\nabla_{||} f|^{p_0}]^\frac{2}{p_0}(x) \,d\kappa\Big)^{1/2}
        \\
        &\qquad+\sum_{l,m\geq 1}\Big(\int_{0}^{\frac{\sqrt{15}}{4}s}\frac{1}{\kappa}e^{-c\frac{2^{2l}l(Q)}{\kappa^2}}\frac{2^{(2n/p_0+2)(m+l)}l(Q)^{2n/p_0+2}}{(s^2-\kappa^2)^{n(\frac{2}{p_0}-1)+2}}e^{-c\frac{2^{2(l+m)}l(Q)}{s^2-\kappa^2}}\inf_{x\in Q}\mathcal{M}[|\nabla_{||} f|^{p_0}]^\frac{2}{p_0} \,d\kappa\Big)^{1/2}\Bigg]
    \end{align*}
    We can simplify, by using that \(s^2-\kappa^2\approx s^2\approx l(Q)^2\) and that $\kappa\mapsto\frac{1}{\kappa}e^{-c\frac{2^{2l}l(Q)}{\kappa^2}}$ is maximized for $\kappa\in (0,s)$ by $\kappa=s\approx l(Q)$ with the value \(\frac{1}{l(Q)}e^{-2^{2l}}\). This gives
    \begin{align}
        J_2&\lesssim\Vert\partial_sA(\cdot,s)\Vert_{L^\infty}\Bigg[\sum_{l\geq 1}\Big(\int_{0}^{\frac{\sqrt{15}}{4}s}\frac{1}{s}e^{-2^{2l}}\inf_{x\in Q}\mathcal{M}[|\nabla_{||} f|^{p_0}]^\frac{2}{p_0}(x) \,d\kappa\Big)^{1/2}\nonumber
        \\
        &\qquad+\sum_{l,m\geq 1}\Big(\int_{0}^{\frac{\sqrt{15}}{4}s}\frac{1}{s}e^{-2^{2l}}2^{(2n/p_0+2)(m+l)}e^{-c2^{2(l+m)}}\inf_{x\in Q}\mathcal{M}[|\nabla_{||} f|^{p_0}]^\frac{2}{p_0} \,d\kappa\Big)^{1/2}\Bigg]\nonumber
        \\
        &\lesssim\Vert\partial_sA(\cdot,s)\Vert_{L^\infty}\inf_{x\in Q}\mathcal{M}[|\nabla_{||} f|^{p_0}]^\frac{1}{p_0}(x)\Bigg[\sum_{l\geq 1}e^{-2^{l}} +\sum_{l\geq 1}e^{-2^{l}}\sum_{m\geq 1}\Big(2^{(2n/p_0+2)(m+l)}e^{-c2^{2(l+m)}}\Big)^{1/2}\bigg]\nonumber
        \\
        &\lesssim \Vert\partial_sA(\cdot,s)\Vert_{L^\infty}\inf_{x\in Q}\mathcal{M}[|\nabla_{||} f|^{p_0}]^\frac{1}{p_0}(x).\label{eq:ProofwsL1noJ1.5}
    \end{align}

    For $J_1$ we want to apply \reflemma{lemma:L2AreaBoundGeneralCase}. For that we notice that 
    \[x\mapsto\chi_{2Q}(x)\partial_sA(x,s)\nabla e^{-(s^2-\kappa^2)L^s}f(x)\in L^2(\mathbb{R}^n).\]
    Making use of a Hodge decomposition we obtain a family $\phi^\kappa\in W^{1,2}_0(2Q)$ with
    \begin{align}\mathrm{div}_{||}(A(x,s)\nabla\phi^\kappa(x))=\mathrm{div}(\chi_{2Q}\partial_sA\nabla e^{-(s^2-\kappa^2)L^s}f).\label{eq:ProofwsL1noJPDE}\end{align}
    By Appendix~\ref{appendix:HodgeDecomposition} or arguments as in \cite[Section 3]{ulmer_solvability_2026}, we have that $\phi^\kappa, \nabla\phi^\kappa$ are weakly differentiable in $\kappa$, and we obtain by testing above PDE by \(\phi^\kappa\)
    \begin{align*}     
    \Vert\nabla\phi^\kappa\Vert_{L^2}^2&\lesssim\int_{2Q}A(x,s)\nabla\phi^\kappa(x)\cdot\nabla\phi^\kappa(x) dx=\int_{2Q}\partial_sA(x,s)\nabla e^{-(s^2-\kappa^2)L^s}f(x)\cdot\nabla\phi^\kappa(x)dx
    \\
    &\lesssim \Vert\partial_sA(\cdot,s)\Vert_{L^\infty}\Vert\nabla e^{-(s^2-\kappa^2)L^s}f\Vert_{L^2}\Vert\nabla\phi^\kappa\Vert_{L^2}.
    \end{align*}
    Hence 
    \begin{align*}
        \Vert\nabla\phi^\kappa\Vert_{L^2}\lesssim \Vert\partial_sA(\cdot,s)\Vert_{L^\infty}\Vert\nabla e^{-(s^2-\kappa^2)L^s}f\Vert_{L^2(2Q)},
    \end{align*}
    and since $s^2-\kappa^2\approx s^2$ and $s\approx l(Q)$ by \reflemma{lemma:NablaSemigroupBoundedbyM[nablaf]}
    \begin{align}
        \Vert\nabla\phi^\kappa\Vert_{L^2}\lesssim \Vert\partial_sA(\cdot,s)\Vert_{L^\infty}|Q|^{1/2}\inf_{x\in Q}\mathcal{M}[|\nabla_{||} f|](x).\label{eq:ProofwsL1noJ2.5}
    \end{align}

    Similarly, since $\nabla\phi^\kappa$ is weakly differentiable in $\kappa$, we get by differentiating the PDE \eqref{eq:ProofwsL1noJPDE}
    \begin{align*}
        &\int_{2Q}A(x,s)\nabla\partial_\kappa\phi^\kappa(x)\cdot \nabla v(x) dx
        =\partial_\kappa\Big(\int_{2Q}A(x,s)\nabla\phi^\kappa(x)\cdot \nabla v(x) dx\Big)
        \\
        &=\int_{2Q}\partial_sA(x,s)\partial_\kappa\big(\nabla e^{-(s^2-\kappa^2)L^s}f(x)\big)\cdot\nabla v(x)dx
        =\int_{2Q}2\kappa\partial_sA\nabla L^se^{-(s^2-\kappa^2)L^s}f(x)\cdot\nabla v(x)dx.
    \end{align*}
    Thus,
    \begin{align*}
        \Vert\nabla\partial_\kappa\phi^\kappa\Vert_{L^2}&\lesssim \int_{2Q}A(x,s)\nabla\partial_\kappa\phi^\kappa(x)\cdot \nabla\partial_\kappa\phi^\kappa(x) dx
        \\
        &=\int_{2Q}2\kappa\partial_sA(x,s)\nabla L^se^{-(s^2-\kappa^2)L^s}f(x)\cdot\nabla\partial_\kappa\phi^\kappa(x)dx
        \\
        &\lesssim \kappa\Vert\partial_s A(\cdot,s)\Vert_{L^\infty}\Vert\nabla L^se^{-(s^2-\kappa^2)L^s}f\Vert_{L^2}\Vert\nabla\partial_\kappa\phi^\kappa\Vert_{L^2}.
    \end{align*}
    now, \reflemma{lemma:NablaSemigroupBoundedbyM[nablaf]} yields
    \begin{align}
        \Vert\nabla\partial_\kappa\phi^\kappa\Vert_{L^2}&\lesssim \frac{\kappa}{s^2-\kappa^2}\Vert\partial_s A(\cdot,s)\Vert_{L^\infty}|Q|^{1/2}\inf_{x\in Q}\mathcal{M}[|\nabla_{||} f|](x)\nonumber
        \\
        &\lesssim \frac{\kappa}{s^2}\Vert\partial_s A(\cdot,s)\Vert_{L^\infty}|Q|^{1/2}\inf_{x\in Q}\mathcal{M}[|\nabla_{||} f|](x)\label{eq:ProofwsL1noJ3}
    \end{align}
    where we used that $s^2-\kappa^2\approx s^2$ for the last inequality. 

    Now let us return to bounding $J_1$ with \reflemma{lemma:L2AreaBoundGeneralCase}. We note that
    \begin{align*}
        J_1&=|Q|^{-1/2}\Big(\int_Q\int_{0}^{\frac{\sqrt{15}}{4}s}|e^{-\frac{\kappa^2}{2}L^s}\mathrm{div}(\chi_{2Q}\partial_sA\nabla e^{-(s^2-\kappa^2)L^s}f)|^2\kappa \,d\kappa dx\Big)^{1/2}
        \\
        &= |Q|^{-1/2}\Big(\int_Q\int_{0}^{\frac{\sqrt{15}}{4}s}|e^{-\frac{\kappa^2}{2}L^s}\mathrm{div}(A(x,s)\nabla\phi^\kappa)|^2\kappa \,d\kappa dx\Big)^{1/2}
        \\
        &\lesssim |Q|^{-1/2}\Big\Vert\mathcal{A}\Big(\chi_{(0,\frac{\sqrt{15}}{4}s)}(\kappa)\kappa L^se^{-\frac{\kappa^2}{2}L^s}\phi^\kappa\Big)\Big\Vert_{L^2(\mathbb{R}^n)}.
    \end{align*}
    Since \eqref{eq:ProofwsL1noJ2.5}, and since by \eqref{eq:ProofwsL1noJ3}
    \begin{align*}
        \int_0^s\Vert\nabla \partial_\kappa\phi^\kappa\Vert_{L^2(\mathbb{R}^n)}d\kappa&\lesssim \Vert\partial_s A(\cdot,s)\Vert_{L^\infty}|Q|^{1/2}\inf_{x\in Q}\mathcal{M}[|\nabla_{||} f|](x)\int_0^\kappa\frac{\kappa}{s^2}d\kappa
        \\
        &\lesssim \Vert\partial_s A(\cdot,s)\Vert_{L^\infty}|Q|^{1/2}\inf_{x\in Q}\mathcal{M}[|\nabla_{||} f|](x),
    \end{align*}
    we have that \reflemma{lemma:L2AreaBoundGeneralCase} is applicable, and we obtain
    \begin{align}J_1\lesssim \Vert\partial_s A(\cdot,s)\Vert_{L^\infty}\inf_{x\in Q}\mathcal{M}[|\nabla_{||} f|](x).\label{eq:ProofwsL1noJ4}\end{align}
    Let us return to \(J\) now. Combining \eqref{eq:ProofwsL1noJ1} with the bounds for $J_1$ and $J_2$, \eqref{eq:ProofwsL1noJ4} and \eqref{eq:ProofwsL1noJ1.5}, we get
    \begin{align*}
        J&\lesssim\int_0^R \sum_{Q\in\mathcal{D}_K(\mathbb{R}^n)}|Q| \big(\inf_{x\in Q}\mathcal{M}[\mathcal{N}(cv)^2]^{1/2}(x)\big) \big(\Vert\partial_s A(\cdot,s)\Vert_{L^\infty}\inf_{y\in Q}\mathcal{M}[|\nabla_{||} f|](y)\big)
        \\
        &\lesssim \int_0^R\Vert\partial_sA(\cdot,s)\Vert_{L^\infty}\Big(\int_{\mathbb{R}^n}\mathcal{M}[\mathcal{N}(cv)^2]^{1/2} \mathcal{M}[|\nabla_{||} f|]dx\Big)ds
        \\
        &\lesssim \Vert \mathcal{M}[\mathcal{N}(cv)^2]^{1/2}\Vert_{L^{p'}}\Vert  \mathcal{M}[|\nabla_{||} f|]\Vert_{L^p}\int_0^R\Vert\partial_sA(\cdot,s)\Vert_{L^\infty}ds
        \\
        &\lesssim \Vert \mathcal{N}(v)\Vert_{L^{p'}}\Vert  \nabla_{||} f\Vert_{L^p}.
    \end{align*}

    This completes the proof.
\end{proof}

\subsection{Proof of \refthm{MAINTHM}}\label{sec:subMainProof}

To begin with, let \(h\) be a quasi dualising function given by the following lemma.
\begin{lemma}[Lemma 2.8 in \cite{kenig_neumann_1995}]\label{lemma:DefAndPropertiesOfH}
    There exists a function \(h:\mathbb{R}^{n+1}_+\to\mathbb{R}^{n+1}_+\) with compact support in \(\Omega=\mathbb{R}^{n+1}_+\) such that
    \[\Vert \tilde{\mathcal{N}}_1(\nabla u)\Vert_{L^p(\partial\Omega)}\lesssim \int_\Omega \nabla u \cdot h dxdt,\]
    and
    \[ \int_\Omega F \cdot h dxdt\lesssim \Vert \tilde{\mathcal{N}}_1(F)\Vert_{L^p(\partial\Omega)} \]
    for every vector valued function \(F:\mathbb{R}^{n+1}\to\mathbb{R}^{n+1}\).
\end{lemma}
Now let \(v\in W^{1,2}_{loc}(\Omega)\) be the solution to the Poisson-Dirichlet problem of the adjoint operator
\begin{align}\begin{cases} L^*v=-\mathrm{div}(h) &\textrm{in }\Omega,\\ v=0 &\textrm{on }\partial\Omega.\end{cases}\label{PDEForv}\end{align}
Since we assume solvability of the Dirichlet problem $(\operatorname{D})_{p'}^{L^*}$ for the adjoint \(L^*\), we know that the elliptic measure of the adjoint \(\omega^*\) lies in \(B_{p}(\sigma)\), where \(B_{p}(\sigma)\) is the reverse H\"{o}lder class. Without loss of generality, we can assume that $1<p<2$. Hence, we obtain the following bounds for \(v\):
\begin{prop}[\cite{kenig_neumann_1995}] 
    Let \(v\) be given by \eqref{PDEForv}. If the elliptic measure of the adjoint \(\omega^*\in B_{r}(\sigma)\) for some \(1<r<\infty\), then
    \begin{align}\Vert \mathcal{S}(v)\Vert_{L^r(\partial\Omega)}+\Vert \mathcal{N}(v)\Vert_{L^r(\partial\Omega)}+\Vert \tilde{\mathcal{N}}_2(\delta|\nabla v|)\Vert_{L^r(\partial\Omega)}+ \Vert \tilde{\mathcal{N}}_1(\delta|\nabla v|)\Vert_{L^r(\partial\Omega)}\lesssim C.\label{eq:BoundsOnV}\end{align}
\end{prop}

Under our assumptions, this proposition holds for \(r=p\) and for this choice of \(p\) we have by integration by parts and using that \(Lu=0\)
\begin{align}
    \Vert \tilde{\mathcal{N}}_1(\nabla u)\Vert_{L^p}&\lesssim \int_\Omega \nabla u \cdot h dxdt= \int_\Omega A\nabla u\cdot \nabla v dxdt - \int_{\partial\Omega}A^*\nabla v \cdot u\nu dx\nonumber
    \\
    &=-\int_{\partial \Omega}u b\cdot\nabla_{||}v + ud\partial_t v dx\label{eq:MainProof2}
    \\
    &=\int_{\partial \Omega}u(x,0) \Big(\int_0^\infty \partial_t\big(b(x,t)\cdot\nabla_{||}v(x,t) + d(x,t)\partial_t v(x,t)\big) dt\Big) dx.\label{eq:mainProof1}
\end{align}
Without loss of generality we can approximate all involved components of the matrix \(A\) and \(h\) by smooth functions so that the PDE \eqref{PDEForv} can be used pointwise. The following arguments are independent of this approximation and hence we can take the limit in the very end, but we omit this argument here. Since \(\int_0^\infty \partial_th(x,t)dt=0\),  we can continue with
\begin{align}
    \eqref{eq:mainProof1}&=-\int_{\partial \Omega}u \Big(\int_0^\infty \mathrm{div}_{||}\big(\AP^*(x,t)\cdot\nabla_{||}v(x,t)+c(x,t)\partial_tv(x,t)\big)dt\nonumber
    \\
    &\qquad - \int_0^\infty\mathrm{div}_{||}(h(x,t)) + \partial_t h(x,t) dt\Big) dx\nonumber
    \\
    &=\lim_{R\to\infty}\int_{\partial \Omega}\nabla_{||}u(x,0) \Big(\int_0^R\AP^*(x,t)\cdot\nabla_{||}v(x,t)+c(x,t)\partial_tv(x,t)dt\Big)dx\label{eq:MainProof3}
    \\
    &\qquad - \int_{\Omega}\nabla_{||}u(x,0)\cdot h(x,t)dxdt.\nonumber
\end{align}
The last term yields by \reflemma{lemma:DefAndPropertiesOfH}
\begin{align*}
    \Big|\int_{\Omega}\nabla_{||}u(x,0)\cdot h(x,t)dxdt\Big|=\Big|\int_{\Omega}\nabla_{||}f(x)\cdot h(x,t)dxdt\Big|\lesssim \Vert \mathcal{M}[\nabla_{||} f]\Vert_{L^p}\lesssim \Vert \nabla_{||} f\Vert_{L^p}. 
\end{align*}
Let us fix a large $R>0$. Set \(\Omega^R:=\mathbb{R}^n\times(0,R)\). The truncation of the integral in $t$ direction justifies the convergence of all the following integral terms, and will contribute some error terms that depend on $R$. We are going to use a smooth cut-off function $\eta_R\in C^\infty(\mathbb{R})$ with $0\leq \eta_R\leq 1$, $\eta_R\equiv 1$ on $(-\infty,R/2)$, and $\eta_0\equiv 0$ on $(R,\infty)$. For the term in \eqref{eq:MainProof2}, we can set 
\[V(x,s):=\int_s^R \AP^*(x,t)\cdot\nabla_{||}v(x,t) + c(x,t)\partial_t v(x,t)dt,\]
and get
\begin{align*}
    &\int_{\partial \Omega}\nabla_{||}u(x,0) \Big(\int_0^R\AP^*(x,t)\cdot\nabla_{||}v(x,t) + c(x,t)\partial_t v(x,t)dt\Big)dx
    \\
    &\qquad=\int_{\partial\Omega}\nabla_{||}f(x) V(x,0) dx =- \int_{\Omega^R}\partial_s(\nabla_{||}\mathcal{P}_sf(x) V(x,s)\eta_R(s))\,dxds
    \\
    &\qquad= -\int_{\Omega^R}\partial_s\nabla_{||}\mathcal{P}_sf(x)\cdot V(x,s)\eta_R(s) + \nabla_{||}\mathcal{P}_sf(x)\cdot \partial_s V(x,s)\eta_R(s) \,dxds
    \\
    &\qquad\qquad -\int_{\Omega^R}\nabla_{||}\mathcal{P}_sf(x)\cdot V(x,s)\partial_s\eta_R(s)\,dxs 
\end{align*}
We introduce \(\partial_s(s)=1\) and use integration by parts to obtain
\begin{align*}
    &= \int_{\Omega^R}\partial_{s}^2\nabla_{||}\mathcal{P}_sf(x)\cdot V(x,s)\eta_R(s)s\,dxds + 2\int_{\Omega^R}\nabla_{||}\partial_s\mathcal{P}_sf(x) \cdot\partial_s V(x,s)\eta_R(s)s\,dxds
    \\
    &\qquad + \int_{\Omega^R}\nabla_{||}\mathcal{P}_sf(x) \cdot\partial_{s}^2 V(x,s) \eta_R(s)s \,dxds 
    +\int_{\Omega^R}\nabla_{||}\partial_s\mathcal{P}_sf(x) \cdot V(x,s)\partial_s\eta_R(s)s\,dxds
    \\
    &\qquad+ \int_{\Omega^R}\nabla_{||}\mathcal{P}_sf(x) \cdot\partial_{s} V(x,s) \partial_s\eta_R(s)s \,dxds
    -\int_{\Omega^R}\nabla_{||}\mathcal{P}_sf(x)\cdot V(x,s)\partial_s\eta_R(s)\,dxds
    \\
    & =:I+II+III+E_1^R+E_2^R+E_3^R.
\end{align*}
Throughout the proof, we are going to collect error terms like $E_1^R,E_2^R$ and $E_3^R$ here, and deal with them in the end.
First for $II$, we observe that \(|\partial_s V(x,s)|\lesssim |\nabla v(x,s)|\) and hence
\begin{align*}
    |II|&\lesssim \Vert \mathcal{A}(s\nabla_{||}\partial_s\mathcal{P}_sf)\Vert_{L^p}\Vert \mathcal{S}(v)\Vert_{L^{p'}}\lesssim \Vert \mathcal{A}(s\partial_s\nabla_{||}\mathcal{P}_sf)\Vert_{L^p},
\end{align*}
where we used \eqref{eq:BoundsOnV}. By \reflemma{lemma:AreaFctBoundsInLp}, we obtain \(|II|\lesssim \Vert \nabla_{||}f\Vert_{L^p}\).

\medskip
Next for $III$, we split using integration by parts
\begin{align*}
    III&=\int_{\Omega^R}\nabla_{||}\mathcal{P}_sf(x) \cdot \partial_s(A_{||}^*(x,s)\nabla_{||}v(x,s) + c(x,s)\partial_sv(x,s))\eta_R(s)s\, dxds
    \\
    &\lesssim \int_{\Omega^R}\nabla_{||}\mathcal{P}_sf(x) \cdot \partial_sA_{||}^*(x,s)\nabla_{||}v(x,s) \eta_R(s)s\, dxds - \int_{\Omega^R}L^s\mathcal{P}_sf(x) \partial_s v(x,s) \eta_R(s)s\, dxds
    \\
    &\qquad - \int_{\Omega^R} \nabla_{||}\partial_s\mathcal{P}_sf(x)\cdot c(x,s)\partial_sv(x,s)\eta_R(s)s\, dxds -\int_{\Omega_R}\nabla_{||}\mathcal{P}_sf(x) \cdot c(x,s)\partial_sv(x,s)\eta_R(s)\, dxds
    \\
    &\qquad -\int_{\Omega_R}\nabla_{||}\mathcal{P}_sf(x) \cdot c(x,s)v(x,s)\partial_s\eta_R(s)\, dxds
    \\
    &=:III_1+III_2+III_3+III_4+E_4^R.
\end{align*}
The error term $E_4^R$ is postponed until the end. For $III_1$, a stopping time argument yields 
\[III_1\lesssim \Vert\sup_{B(x,s,s/2)}|\partial_sA|\Vert_{\mathcal{C}}\Vert\tilde{\mathcal{N}}(\nabla_{||}\mathcal{P}_s f)\Vert_{L^p}\Vert \mathcal{S}(v)\Vert_{L^{p'}}\lesssim \Vert\nabla_{||}f\Vert_{L^p},\]
where we used \eqref{eq:BoundsOnV}, \refcor{cor:NontangentialBoundsOnW_1andW_2} and \eqref{cond:ExplanationForCarlesonCondition}, the \(L^1\) Carleson condition on \(\partial_t A\). We omit this standard stopping time argument here, but would like to refer to the integral \(II_{31}\) in \cite{dindos_regularity_2023}, where the authors perform a stopping time argument on an expression of exactly this form.
Furthermore,
\[III_2\lesssim \Vert\mathcal{A}(sL^s\mathcal{P}_s f)\Vert_{L^p}\Vert \mathcal{S}(v)\Vert_{L^{p'}}\lesssim\Vert\nabla_{||}f\Vert_{L^p},\]
and 
\[III_3\lesssim\Vert\mathcal{A}(\nabla_{||}\partial_s\mathcal{P}_s f)\Vert_{L^p}\Vert \mathcal{S}(v)\Vert_{L^{p'}}\lesssim \Vert\nabla_{||}f\Vert_{L^p}\]
follow from \reflemma{lemma:AreaFctBoundsInLp} and \eqref{eq:BoundsOnV}.

For the last term $III_4$, we can integrate by parts ijn $s$ and obtain
\begin{align*}
    III_4&= -\int_{\Omega^R} \nabla_{||}\mathcal{P}_sf(x) \cdot \partial_s c(x,s)v(x,s)\eta_R(s)\,dxds -\int_{\Omega^R}\nabla_{||}\partial_s\mathcal{P}_sf(x) \cdot c(x,s)v(x,s)\eta_R(s)\, dxds
    \\
    &\qquad-\int_{\Omega^R}\nabla_{||}\mathcal{P}_sf(x) \cdot c(x,s)v(x,s) \partial_s\eta_R(s)\,dxds
    \\
    &= -\int_{\Omega^R} \nabla_{||}\mathcal{P}_sf(x) \cdot \partial_s c(x,s)v(x,s)\eta_R(s)\,dxds -\int_{\Omega^R}\nabla_{||}W_1f(x) \cdot c(x,s)v(x,s) \eta_R(s)\,dxds
    \\
    &\qquad + \int_{\Omega^R}\nabla_{||}W_2f(x) \cdot c(x,s)v(x,s) \eta_R(s)\,dxds-\int_{\Omega^R}\nabla_{||}\mathcal{P}_sf(x) \cdot c(x,s)v(x,s) \partial_s\eta_R(s)\,dxds
    \\
    &=:J_1+J_2+J_3+E_5^R.
\end{align*}
Bounding $E_5^R$ is postponed until the end, and Carleson's Theorem \refprop{Carleson's Theorem} yields
\[J_1\lesssim \Vert\sup_{B(x,s,s/2)}|\partial_sc|\Vert_{\mathcal{C}}\int_{\mathbb{R}^n}\tilde{\mathcal{N}}_1(|\nabla_{||}\mathcal{P}_sf| v)d\sigma\lesssim \Vert\partial_sc\Vert_{\mathcal{C}}\Vert\tilde{\mathcal{N}}_2(\nabla_{||}\mathcal{P}_sf)\Vert_{L^p}\Vert\tilde{\mathcal{N}}_2(v)\Vert_{L^{p'}},\]
which is bounded by $\Vert\nabla_{||}f\Vert_{L^p}$ due to \eqref{eq:BoundsOnV}, \refcor{cor:NontangentialBoundsOnW_1andW_2}, and \eqref{cond:ExplanationForCarlesonCondition}, the \(L^1\) Carleson condition on \(\partial_t A\).

The term $J_3$ is the object of \reflemma{lemma:W2vcBoundebyN(v)}, and hence with \eqref{eq:BoundsOnV} 
\[J_3\lesssim \Vert\nabla_{||} f\Vert_{L^p}\Vert \mathcal{N}(v)\Vert_{L^{p'}}\lesssim \Vert\nabla_{||} f\Vert_{L^p}.\]

For $J_2$ we obtain
\begin{align*}
    J_2&=-\int_{\Omega^R}\nabla_{||}L^s\mathcal{P}_sf(x) \cdot c(x,s)v(x,s)\eta_R(s)s\,dxds
    \\
    &=-\int_{\Omega^R} L^s\mathcal{P}_{\frac{\sqrt{3}}{2}s}f(x) \cdot \mathcal{P}_{s/2}^*\mathrm{div}(cv)(x,s) \eta_R(s)s\,dxds
    \\
    &=\int_{\Omega^R} L^s\mathcal{P}_{\frac{\sqrt{3}}{2}s}f(x) \cdot \big(\mathcal{P}_{s/2}^*\mathrm{div}(cv)(x,s)-\mathcal{P}_{s/2}^*\mathrm{div}(c)(x,s)v(x,s)\big) \eta_R(s)s\,dxds
    \\
    &\qquad +\int_{\Omega^R} L^s\mathcal{P}_{\frac{\sqrt{3}}{2}s}f(x) \cdot \mathcal{P}_{s/2}^*\mathrm{div}(c)(x,s)v(x,s) \eta_R(s)s\,dxds
    \\
    &\lesssim \Vert \mathcal{A}(sL^s\mathcal{P}_{\frac{\sqrt{3}}{2}s}f)\Vert_{L^p} \Vert \mathcal{A}\big(s(\mathcal{P}_{s/2}^*\mathrm{div}(cv)(x,s)-\mathcal{P}_{s/2}^*\mathrm{div}(c)(x,s)v(x,s))\big)\Vert_{L^p}
    \\
    &\qquad + \Vert \mathcal{A}(sL^s\mathcal{P}_{\frac{\sqrt{3}}{2}s}f)\Vert_{L^p} \Vert \mathcal{C}(s\mathcal{P}_{s/2}^*\mathrm{div}(c))\Vert_{L^\infty} \Vert \mathcal{N}(v)\Vert_{L^{p'}}.
\end{align*}
By a change of variables, we can observe that \(\Vert \mathcal{A}(sL^s\mathcal{P}_{\frac{\sqrt{3}}{2}s}f)\Vert_{L^p}\lesssim \Vert \mathcal{A}(sL^s\mathcal{P}_{s}f)\Vert_{L^p}\) and similarly \(\mathcal{C}(s\mathcal{P}_{s/2}^*\mathrm{div}(c))(x,s)\lesssim \mathcal{C}(s\mathcal{P}_{s}^*\mathrm{div}(c))(x,s)\). Hence, the area function bounds from \reflemma{lemma:AreaFctBoundsInLp} and the Carleson function bound from \reflemma{lemma:CarelsonFctBoundPs} apply. Together with \eqref{eq:BoundsOnV}, we get 
\[\Vert \mathcal{A}(sL^s\mathcal{P}_{\frac{\sqrt{3}}{2}s}f)\Vert_{L^p} \Vert \mathcal{C}(s\mathcal{P}_{s/2}^*\mathrm{div}(c))\Vert_{L^\infty} \Vert \mathcal{N}(v)\Vert_{L^{p'}}\lesssim\Vert\nabla_{||}f\Vert_{L^p}.\]
Moreover, the operator \(v\mapsto s\mathcal{P}_{s/2}^*\mathrm{div}(cv)-s\mathcal{P}_{s/2}^*\mathrm{div}(c)v\) kills constants and satisfies $L^2-L^2$ off-diagonal estimates, since if $v$ is supported in a closed set $E$ with $\mathrm{dist}(E,F)>0$, then
\[\Vert s\mathcal{P}_{s/2}^*\mathrm{div}(cv)-s\mathcal{P}_{s/2}^*\mathrm{div}(c)v\Vert_{L^2(F)}= \Vert s\mathcal{P}_{s/2}^*\mathrm{div}(cv)\Vert_{L^2(F)},\]
which satisfies $L^2-L^2$ off-diagonal estimates due to \refprop{prop:L2boundednessOfSemigroupOperators}. As a consequence, \reflemma{lemma:R_tBoundedbynablav} applies, and we get with \eqref{eq:BoundsOnV}
\begin{align*}
\Vert \mathcal{A}(sL^s\mathcal{P}_{\frac{\sqrt{3}}{2}s}f)\Vert_{L^p} \Vert \mathcal{A}\big(s(\mathcal{P}_{s/2}^*\mathrm{div}(cv)(x,s)-\mathcal{P}_{s/2}^*\mathrm{div}(c)(x,s)v(x,s))\big)\Vert_{L^p}
\\
\lesssim  \Vert \nabla_{||}f\Vert_{L^p} \Vert \mathcal{A}(s\nabla v)\Vert_{L^p}\lesssim \Vert\nabla_{||}f\Vert_{L^p}.
\end{align*}
Together, this implies
\[J_2\lesssim \Vert\nabla_{||} f\Vert_{L^p}.\]

\smallskip

Finally, for \(I\)
\begin{align*}
    |I|&=\Big|\int_{\Omega^R}s\nabla_{||}\partial_{s}^2\mathcal{P}_sf(x)\cdot\Big(\int_s^R \AP(x,t)\nabla_{||}v(x,t) + c(x,t)\partial_tv(x,t)dt\Big)\eta_R(s) dx\Big|
    \\
    &=\Big|\int_{\Omega^R}\nabla_{||}\Big(\int_0^ts\eta_R(s)\partial^2_{s}\mathcal{P}_sf(x)ds\Big)\cdot(\AP(x,t)\nabla_{||}v(x,t) + c(x,t)\partial_tv(x,t)) dxdt\Big|
\end{align*}
Since
\[\int_0^ts\eta_R(s)\partial^2_{s}\mathcal{P}_sf(x)ds=\big(t\partial_{t}\mathcal{P}_tf(x)-\mathcal{P}_tf(x)+f(x)\big)\eta_R(t) -\int_0^ts\partial_s\eta_R(s)\partial_s\mathcal{P}_sf(x)-\partial_s\eta_R(s)\mathcal{P}_sf(x)ds\]
we can continue with
\begin{align*}
    &=\Big|\int_{\Omega^R}\nabla_{||}\big(t\partial_{t}\mathcal{P}_tf(x)-\mathcal{P}_tf(x)+f(x)\big)\eta_R(t)\cdot\big(\AP(x,t)\nabla_{||}v(x,t) + c(x,t)\partial_tv(x,t)\big) dxdt\Big|
    \\
    &\qquad+\Big|\int_{\Omega^R}\nabla_{||}\big(\int_0^ts\partial_s\eta_R(s)\partial_s\mathcal{P}_sf(x)-\partial_s\eta_R(s)\mathcal{P}_sf(x)ds\big)\cdot\big(\AP(x,t)\nabla_{||}v(x,t) + c(x,t)\partial_tv(x,t)\big) dxdt\Big|
    \\
    &\lesssim \Vert\mathcal{A}(t\nabla_{||}\partial_t \mathcal{P}_tf)\Vert_{L^p}\Vert \mathcal{S}(v)\Vert_{L^{p'}}
    \\
    &\qquad + \Big|\int_{\Omega}\nabla_{||}\big(\mathcal{P}_tf(x)-f(x)\big)\eta_R(t)\cdot\big(\AP(x,t)\nabla_{||}v(x,t) + c(x,t)\partial_tv(x,t)\big) dxdt\Big|
    \\
    &\qquad+\Big|\int_{\Omega^R}\nabla_{||}\big(\int_0^ts\partial_s\eta_R(s)\partial_s\mathcal{P}_sf(x)-\partial_s\eta_R(s)\mathcal{P}_sf(x)ds\big)\cdot\big(\AP(x,t)\nabla_{||}v(x,t) + c(x,t)\partial_tv(x,t)\big) dxdt\Big|
    \\
    &=: \Vert\mathcal{A}(t\nabla_{||}\partial_t \mathcal{P}_tf)\Vert_{L^p}\Vert \mathcal{S}(v)\Vert_{L^{p'}} + I_1+E_6^R.
\end{align*}
By \eqref{eq:BoundsOnV} and \reflemma{lemma:AreaFctBoundsInLp}, we get $\Vert\mathcal{A}(t\nabla_{||}\partial_t \mathcal{P}_tf)\Vert_{L^p}\Vert \mathcal{S}(v)\Vert_{L^{p'}}\lesssim \Vert\nabla_{||}f\Vert_{L^p}$.
For the term $I_1$, we can note that $(\mathcal{P}_tf-f)\eta_R\in W^{1,2}_0(\Omega^R)$, and hence using the PDE \eqref{PDEForv} for \(v\) gives
\begin{align*}
    I_1\lesssim&\Big|\int_{\Omega^R}\nabla_{||}\big(\mathcal{P}_tf(x)-f(x)\big)\eta_R(t)\cdot\big(\AP^*(x,t)\nabla_{||}v(x,t) + c(x,t)\partial_tv(x,t)\big) dxdt\Big|
    \\
    &\qquad\leq \Big|\int_{\Omega^R}\partial_t\mathcal{P}_tf(x)\eta_R(t)\cdot\big(b(x,t)\nabla_{||}v(x,t) + d(x,t)\partial_tv(x,t)\big)dxdt\Big|
    \\
    &\qquad\qquad + \Big|\int_{\Omega^R}\partial_t\eta_R(t)\mathcal{P}_tf(x)\cdot\big(b(x,t)\nabla_{||}v(x,t) + d(x,t)\partial_tv(x,t)\big)dxdt\Big|
    \\
    &\qquad \qquad + \Big|\int_{\Omega^R}(\nabla_{||}\mathcal{P}_tf - \nabla_{||}f)(x)\eta_R(t)\cdot h_{||} + \partial_t \mathcal{P}_tf(x)\cdot h_{n+1}(x,t) dxdt\Big|
    \\
    &\qquad \lesssim  \Vert\mathcal{A}(\partial_t \mathcal{P}_tf)\Vert_{L^p}\Vert \mathcal{S}(v)\Vert_{L^{p'}} +  \Vert\tilde{\mathcal{N}}_1(\partial_t \mathcal{P}_tf)\Vert_{L^p}+\Vert\tilde{\mathcal{N}}_1(\nabla_{||} \mathcal{P}_tf)\Vert_{L^p}
    \\
    &\qquad\qquad + \Big|\int_{\Omega^R}\partial_t\eta_R(t)\mathcal{P}_tf(x)\cdot\big(b(x,t)\nabla_{||}v(x,t) + d(x,t)\partial_tv(x,t)\big)dxdt\Big|
    \\
    &\qquad\lesssim \Vert \nabla_{||}f\Vert_{L^p}+E_7^R.
\end{align*}
Here we used again \eqref{eq:BoundsOnV}, \reflemma{lemma:AreaFctBoundsInLp}, \refcor{cor:NontangentialBoundsOnW_1andW_2}, and that \(\Vert\tilde{\mathcal{N}}(\partial_t \mathcal{P}_tf)\Vert_{L^p}\lesssim \Vert\mathcal{A}(\partial_t\mathcal{P}_tf)\Vert_{L^p}\).
\medskip

To summerize, we have shown so far that
\[\Vert\tilde{\mathcal{N}}_1(\nabla u)\Vert_{L^p}\lesssim \Vert\nabla_{||}f\Vert_{L^p}+\lim_{R\to \infty}\sum_{i=1}^7E_i^R,\]
and thus it remains to bound the error terms such that they converge to zero as $R\to \infty$.
\medskip

\textbf{Bounding the error terms:}
First for $E_1^R$, we observe by Hölder that
\begin{align*}
    |E_1|&=\Big|\int_{\Omega^R}\nabla_{||}\partial_s\mathcal{P}_sf(x) \cdot V(x,s)\partial_s\eta_R(s)s\,dxds\Big|
    \\
    &\lesssim \int_{R/2}^R\frac{s}{R}\Vert\nabla_{||}\partial_s\mathcal{P}_s f\Vert_{L^2(\mathbb{R}^n)}\Big(\int_{\mathbb{R}^n}\big(\int_{R/2}^R|\nabla v|dt\big)^{2}dx\Big)^{1/2}ds
\end{align*}
We can continue with \refprop{prop:W2l2l2boundedness} and \refprop{prop:L2boundednessOfSemigroupOperators}, and get
\begin{align*}
    |E_1^R|&\lesssim \int_{R/2}^R\frac{s}{R^2}\Vert\partial_sA\Vert_{L^\infty}\Vert f\Vert_{L^2(\mathbb{R}^n)}\Big(\int_{\mathbb{R}^n}\big(\int_{R/2}^R|\nabla v|dt\big)^{2}dx\Big)^{1/2}ds
    \\
    &\lesssim \int_{R/2}^R\frac{s}{\sqrt{R}^3}\Vert\partial_sA\Vert_{L^\infty}\Vert f\Vert_{L^2(\mathbb{R}^n)}\Big(\int_{\mathbb{R}^n\times (R/2,R)}| \nabla v|^2dxdt\Big)^{1/2}ds
\end{align*}
Using the pointwise bound $|\partial_sA|\lesssim \frac{1}{s}$ yields
\begin{align*}
    E_1^R&\lesssim \frac{1}{\sqrt{R}}\Vert f\Vert_{L^2(\mathbb{R}^n)}\Vert v\Vert_{W^{1,2}(\Omega)}.
\end{align*}
Since $f$ was assumed to be smooth with compact support, $\Vert f\Vert_{L^2(\mathbb{R}^n)}<\infty$. Moreover $v\in W^{1,2}(\Omega)$, since $h\in L^\infty_c(\Omega)$. Thus $E_1^R\to 0$ as $R\to \infty$. Similarly, but using $|\partial_s V(x,s)|\lesssim |\nabla v(x,s)|$
\begin{align*}
    |E_2^R|&=\Big|\int_{\Omega^R}\nabla_{||}\mathcal{P}_sf(x) \cdot\partial_{s} V(x,s) \partial_s\eta_R(s)s \,dxds\Big|
    \\
    &\lesssim \Vert\nabla_{||}\mathcal{P}_sf\Vert_{L^2(\mathbb{R}^n\times (R/2,R))}\Vert\nabla v\Vert_{L^2(\mathbb{R}^n\times (R/2,R))}
    \lesssim \frac{1}{\sqrt{R}}\Vert f\Vert_{L^2(\mathbb{R}^n)}\Vert \nabla v\Vert_{L^2(\Omega)}.
\end{align*}
Hence $E_2^R\to 0$ as $R\to \infty$.
A similar argument to $E_1^R$ and $E_2^R$ also gives $E_3^R\to 0$ as $R\to \infty$. Also similarly, we get
\begin{align*}
    |E_4^R|\lesssim \frac{1}{R}\Vert\nabla_{||}\mathcal{P}_sf\Vert_{L^2(\mathbb{R}^n\times (R/2,R))}\Vert\nabla v\Vert_{L^2(\mathbb{R}^n\times (R/2,R))}\lesssim \frac{1}{\sqrt{R}}\Vert\nabla_{||}f\Vert_{L^2(\mathbb{R}^n)}\Vert v\Vert_{W^{1,2}(\Omega)},
\end{align*}
which converges to zero as $R\to \infty$. The term $E_5^R=-E_4^R$, and thus can be bounded in the same way. The error term $E_7^R$ can be dealt with analogous to $E_4^R$ with $\mathcal{P}_sf$ replacing $\nabla\mathcal{P}_sf$ and $f$ replacing $\nabla_{||}f$. Lastly with \refprop{prop:W2l2l2boundedness} and \refprop{prop:L2boundednessOfSemigroupOperators}, 
\begin{align*}
    |E_6^R|&=\Big|\int_{\Omega^R}\nabla_{||}\big(\int_0^ts\partial_s\eta_R(s)\partial_s\mathcal{P}_sf(x)-\partial_s\eta_R(s)\mathcal{P}_sf(x)ds\big)\cdot\big(\AP(x,t)\nabla_{||}v(x,t) + c(x,t)\partial_tv(x,t)\big) dxdt\Big|
    \\
    &\lesssim \int_{\Omega^R}\big(\int_{R/2}^R|\nabla_{||}\partial_s\mathcal{P}_sf(x)|+\frac{1}{R}|\nabla_{||}\mathcal{P}_sf(x)|ds\big)|\nabla v(x,t)| dxdt\Big|
    \\
    &\lesssim \int_{R/2}^R\big(\int_{R/2}^R\Vert\nabla_{||}\partial_s\mathcal{P}_sf\Vert_{L^2(\mathbb{R}^n)}+\frac{1}{R}\Vert\nabla_{||}\mathcal{P}_sf\Vert_{L^2(\mathbb{R}^n)}ds\big)\Vert\nabla v(\cdot,t)\Vert_{L^2(\mathbb{R}^n)}dt
    \\
    &\lesssim \int_{R/2}^R\big(\Vert\partial_sA\Vert_{L^\infty(\mathbb{R}^n)}+\frac{1}{R}\big)\Vert f\Vert_{L^2(\mathbb{R}^n)}\Vert\nabla v(\cdot,t)\Vert_{L^2(\mathbb{R}^n)}dt
\end{align*}
Using the pointwise bound $|\partial_sA|\lesssim \frac{1}{s}$ and Jensen's inequality implies
\begin{align*}
    |E_6^R|\lesssim \Vert f\Vert_{L^2(\mathbb{R}^n)}\fint_{R/2}^R\Vert\nabla v(\cdot,t)\Vert_{L^2(\mathbb{R}^n)}dt\leq  \Vert f\Vert_{L^2(\mathbb{R}^n)}\Vert\nabla v\Vert_{L^2(\mathbb{R}^n\times(R/2,R))}.
\end{align*}

Since $v\in W^{1,2}(\Omega)$, it follows that $\Vert\nabla v\Vert_{L^2(\mathbb{R}^n\times(R/2,R))}\to 0$ as $R$ tends to $\infty$.

This completes handling the error terms. Thus the proof of \refthm{MAINTHM} is complete.

\appendix
\section{Appendix}

\subsection{\(L^2\) area function bound for $sL^se^{-s^2L^s}$}

First, we need the following decay lemma, which is standard, if the operator $L^s$ and the function family $\phi^s$ are independent of $s$.

\begin{prop}\label{prop:L2NormGoesTo0}
If we have a compact set \(K\subset\mathbb{R}^n\), a family \(\phi_s\in W_0^{1,2}(K)\) with $\sup_s\Vert \nabla\phi_s\Vert_{L^2(\mathbb{R}^n)} < \infty$ and a family of uniformly elliptic operators \(L^s\) with the same ellipticity and boundedness constants then
\[\lim_{s\to \infty}\Vert e^{-s^2L^s}(L^s)^{1/2}\phi_s\Vert_{L^2(\mathbb{R}^n)}=0.\]
\end{prop}

\begin{proof}
By the Kato estimate \refprop{prop:KatoConjecture} and the \(L^2\) boundedness of \(s\nabla e^{-s^2L^s}\) from \refprop{prop:L2boundednessOfSemigroupOperators}, we obtain
\begin{align*}
    \Vert e^{-s^2L^s}(L^s)^{1/2}\phi_s\Vert_{L^2(\mathbb{R}^n)}\lesssim \Vert \nabla e^{-s^2L^s}\phi_s\Vert_{L^2(\mathbb{R}^n)}\lesssim \frac{1}{s}\Vert\phi^s\Vert_{L^2(K)}.
\end{align*}
By Poincar\'{e}, we get
\begin{align*}
    \frac{1}{s}\Vert\phi^s\Vert_{L^2(K)}\lesssim\frac{C(K)}{s}\Vert\nabla \phi^s\Vert_{L^2(K)}\lesssim \frac{1}{s}\to 0
\end{align*}
as $s$ tends to infinity.
\end{proof}

The following proof is a slight generalization of \cite[Lemma 7.9]{ulmer_solvability_2026}.
\begin{lemma}\label{lemma:L2AreaBoundGeneralCase}
    Assume that \eqref{cond:mixedL1LInftyCond} holds. Moreover, assume that $\phi^s:\mathbb{R}^n\to\mathbb{R}$ is a family of functions with the following two properties:
    \begin{enumerate}
        \item There exists a compact set $K\subset\mathbb{R}^n$ such that $\phi_s\in W^{1,2}_0(K)$;
        \item\label{item:propAnablaphi} \(\Vert\nabla\phi^s\Vert_{L^2(\mathbb{R}^n)}\lesssim C_1\); and
        \item\label{item:propBnablaphi} \(\int_0^\infty\Vert\nabla\partial_s\phi^s\Vert_{L^2(\mathbb{R}^n)}\lesssim C_2\).
    \end{enumerate}
    Then
    \[\Vert \mathcal{A}(sL^se^{-s^2L^s}\phi^s)\Vert_{L^2(\mathbb{R}^n)}^2\lesssim C_1^2+C_1C_2.\]
\end{lemma}

\begin{proof}
We define $f_s:=(L^s)^{1/2}\phi^s$, and note that by the Kato estimate \refprop{prop:KatoConjecture} \[\Vert L^s e^{-s^2 L^s}\phi^s\Vert_{L^2(\mathbb{R}^n)}=\Vert(L^s)^{1/2} e^{-s^2 L^s}f_{s}\Vert_{L^2(\mathbb{R}^n)}\lesssim\Vert\nabla e^{-s^2 L^s}f_{s}\Vert_{L^2(\mathbb{R}^n)},\]
and 
\[\Vert f_s\Vert_{L^2(\mathbb{R}^n)}\lesssim \Vert\nabla\phi^s\Vert_{L^2(\mathbb{R}^n)}\]

Hence, we obtain by ellipticity of $A$ and integration by parts
\begin{align}
s\Vert L^s e^{-s^2 L^s}\phi^s\Vert_{L^2(\mathbb{R}^n)}^2&\lesssim s\Vert\nabla e^{-s^2 L^s}f_{s}\Vert_{L^2(\mathbb{R}^n)}^2 \lesssim s\int_{\mathbb{R}^n}\AP \nabla e^{-s^2 L^{s}}f_{s}  \cdot \nabla e^{-s^2 L^{s}}f_{s} \, dx\nonumber
\\
&=s\int_{\mathbb{R}^n}L^{s}e^{-s^2 L^{s}}f_{s}  \cdot e^{-s^2 L^{s}}f_{s} \, dx. \label{eq:Proofw_teq1}
\end{align}

We would like to compute \(\partial_s\big(e^{-s^2 L^{s}}f_{s}\big)\) arguing exactly as in \(\partial_s\rho=\partial_s\big(e^{-s^2 L^{s}}\phi^s\big)\) in \cite[Section 4]{ulmer_solvability_2026} or in Section \ref{section:rho}. Although all the argument in Section~\ref{section:rho} carry over verbatim, first, we need to justify the notion $\partial_sf_s$ and that $\partial_sf_s\in L^2(\mathbb{R}^n)$ for every $s>0$. By \cite[Theorem 6.2]{bandara_square_2013}, we have for all $\tau$ sufficiently close to $0$ and \(g\in W^{1,2}(\mathbb{R}^n)\)
\begin{align*}
    \Vert(L^{s+\tau})^{1/2}g-(L^{s})^{1/2}g\Vert_{L^2(\mathbb{R}^n)}\lesssim \Vert A(\cdot,s+\tau)-A(\cdot,s)\Vert_{L^\infty(\mathbb{R}^n)}\Vert\nabla g\Vert_{L^2(\mathbb{R}^n)},
\end{align*}
which gives uniform boundedness in $\tau$.
By Banach-Alaoglu, this implies that for every $s>0$ there exists a sequence $(\tau^s_k)_{k\geq 1}$ with $\tau_k^s\to0$ as $k\to\infty$ such that $\lim_{k\to\infty}\frac{(L^{s+\tau^s_k})^{1/2}g-(L^{s})^{1/2}g}{\tau_k^s}$ converges weakly in $L^2(\mathbb{R}^n)$ and 
\begin{align}
    \lim_{k\to\infty}\Big\Vert\frac{(L^{s+\tau^s_k})^{1/2}g-(L^{s})^{1/2}g}{\tau_k^s}\Big\Vert_{L^2(\mathbb{R}^n)}\lesssim \Vert \partial_s A(\cdot,s)\Vert_{L^\infty(\mathbb{R}^n)}\Vert\nabla g\Vert_{L^2(\mathbb{R}^n)}.\label{eq:SeqDerivativeL^2Bounded}
\end{align}
For the rest of this proof, we define the sequential derivative of a function $f$ as 
\[\partial^{seq}_sf(s):=\lim_{k\to\infty}\frac{f(s+\tau_k^s)-f(s)}{\tau_k^s},\]
where the limit is taken with respect to the weak topology of $L^2(\mathbb{R}^n)$. Note also that the sequence $(\tau_k^s)_k$ depends on the point $s$, which makes this notion even weaker. However, for all functions $f$ that are weakly differentiable in $s$, we know that the difference quotient $\lim_{h\to 0}\frac{f(s+h)-f(s)}{h}$ converges weakly, and hence the sequential derivative agrees with the weak derivative almost everywhere. In particular, $\partial_s^{seq}A=\partial_sA, \partial_s^{seq}\rho=\partial_s\rho$ and $\partial_s^{seq}\phi^s=\partial_s\phi^s$, where $\partial_s$ denotes the weak derivative on $\mathbb{R}^{n+1}_+$. Now, \eqref{eq:SeqDerivativeL^2Bounded} implies that
\begin{align*}
&\Vert\partial_s^{seq}f_s\Vert_{L^2(\mathbb{R}^n)}= \Big\Vert\lim_{k\to\infty}\frac{(L^{s+\tau^s_k})^{1/2}\phi^{s+\tau_k^s}-(L^{s})^{1/2}\phi^s}{\tau_k^s}\Big\Vert_{L^2(\mathbb{R}^n)}
\\
&\qquad\leq \liminf_{k\to\infty}\Big\Vert\frac{(L^{s+\tau^s_k})^{1/2}\phi^{s}-(L^{s})^{1/2}\phi^s}{\tau_k^s}\Big\Vert_{L^2(\mathbb{R}^n)} + \Big\Vert\frac{(L^{s+\tau^s_k})^{1/2}(\phi^{s+\tau_k^s}-\phi^s)}{\tau_k^s}\Big\Vert_{L^2(\mathbb{R}^n)}
\\
&\qquad\lesssim \Vert\partial_sA(\cdot,s)\Vert_{L^\infty(\mathbb{R}^n)}\Vert\nabla\phi^s\Vert_{L^2(\mathbb{R}^n)} + \liminf_{k\to\infty}\Big\Vert\frac{(L^{s+\tau^s_k})^{1/2}(\phi^{s+\tau_k^s}-\phi^s)}{\tau_k^s}\Big\Vert_{L^2(\mathbb{R}^n)}.
\end{align*}
By the Kato estimate \refprop{prop:KatoConjecture}, 
\begin{align*}
\lim_{k\to\infty}\Big\Vert\frac{(L^{s+\tau^s_k})^{1/2}(\phi^{s+\tau_k^s}-\phi^s)}{\tau_k^s}\Big\Vert_{L^2(\mathbb{R}^n)}\lesssim \lim_{k\to\infty}\Big\Vert\frac{\nabla\phi^{s+\tau_k^s}-\nabla\phi^s}{\tau_k^s}\Big\Vert_{L^2(\mathbb{R}^n)}= \Vert\nabla\partial_s\phi^s\Vert_{L^2(\mathbb{R}^n)}.
\end{align*}
Thus, we obtain 
\begin{align}\Vert \partial_s^{seq}f_s\Vert_{L^2(\mathbb{R}^n)}\lesssim \Vert\partial_sA(\cdot,s)\Vert_{L^\infty(\mathbb{R}^n)}\Vert\nabla\phi^s\Vert_{L^2(\mathbb{R}^n)}+\Vert\nabla\partial_s\phi^s\Vert_{L^2(\mathbb{R}^n)},\label{eq:AppendixpartialSeqF}\end{align}
and hence $\partial_s^{seq}f_s\in L^2(\mathbb{R}^n)$.

Thus, verbatim to Section~\ref{section:rho} for a fixed $s$, the PDE
\begin{align*}
    \begin{cases}
        \partial_s^{seq}\partial_te^{-tL^s}f_s=\mathrm{div}(\partial_s^{seq}\AP\nabla e^{-tL^s}\phi^s)-L^s\partial_s^{seq}e^{-tL^s}f_s&, (x,t)\in\mathbb{R}^{n+1}_+
        \\
         \partial_s^{seq}e^{-tL^s}f_s=\partial_s^{seq}f_s&,x\in\mathbb{R}^n,t=0,
    \end{cases}
\end{align*}
holds. As in Section~\ref{section:rho} (or \cite[Section 4]{ulmer_solvability_2026}), we obtain by Duhamel's principle and the semigroup properties that the PDEs

\begin{align*}
    \begin{cases}
        \partial_s^{seq}W_{2}^{(1)}=-L^sW_2^{(1)}&, (x,t)\in\mathbb{R}^{n+1}_+
        \\
         W_{2}^{(1)}f(x)=\partial_s^{seq}f_s&,x\in\mathbb{R}^n,t=0,
    \end{cases}
\end{align*}
and
\begin{align*}
    \begin{cases}
        \partial_s^{seq}W_{2}^{(2)}=\mathrm{div}(\partial_s^{seq}\AP\nabla e^{-tL^s}\phi^s)-L^sW_2^{(2)}&, (x,t)\in\mathbb{R}^{n+1}_+
        \\
         W_{2}^{(2)}f(x)=0&,x\in\mathbb{R}^n,t=0,
    \end{cases}
\end{align*}
are solved by
\begin{align*}
    W_2^{(1)}f(x)&:=-e^{-s^2L^s}\partial_s^{seq}f_s\qquad \textrm{and}\\
    W_2^{(2)}f(x)&:=-\int_0^s2\tau e^{-(s^2-\tau^2)L^s}\mathrm{div}(\partial^{seq}_s\AP(\cdot, s)\nabla e^{-\tau^2L^s}f_s)d\tau.
\end{align*}
Consequently, using the notation \(W(x,t,s):=e^{-tL^s}f(x)\) we get
\begin{align*}
    \partial_s^{seq}e^{-s^2L^s}f(x)&=2s\partial_2W(x,s^2,s)+\partial_3W(x,s^2,s)
    \\
    &=2sL^se^{-s^2L^s}\phi^s-\int_0^s2\tau e^{-(s^2-\tau^2)L^s}\mathrm{div}(\partial^{seq}_s\AP(\cdot, s)\nabla e^{-\tau^2L^s}f_s)d\tau
    \\
    &\qquad-e^{-s^2L^s}\partial_s^{seq}f_s.
\end{align*}

Now, we can continue \eqref{eq:Proofw_teq1} with
\begin{align*}
&s\int_{\mathbb{R}^n}L^{s}e^{-s^2 L^{s}}f_{s}  \cdot e^{-s^2 L^{s}}f_{s} \, dx 
\\
&=-\int_{\mathbb{R}^n}\partial_s^{seq}\big(e^{-s^2 L^{s}}f_{s}\big)  \cdot e^{-s^2 L^{s}}f_{s}  \,dx
\\
&\qquad -\int_{\mathbb{R}^n}\Big(\int_0^s2\tau e^{-(s^2-\tau^2) L^{s}}\mathrm{div}(\partial_s^{seq}\AP\nabla e^{-\tau^2L^s}f_{s})d\tau\Big) e^{-s^2 L^s}f_{s} dx
\\
&\qquad -\int_{\mathbb{R}^n}e^{-s^2L^s}(\partial_s^{seq} f_{s})e^{-s^2L^2}f_{s}dx
\\
&=:I_1+I_2+I_3.
\end{align*}

\medskip

First, we can observe for \(I_3\) with \eqref{eq:AppendixpartialSeqF}, \eqref{item:propAnablaphi}, and \eqref{item:propBnablaphi} that
\begin{align*}
    I_3&\lesssim \Vert e^{-s^2L^s}(\partial_s^{seq}f_{s})\Vert_{L^2(\mathbb{R}^n)}\Vert e^{-s^2L^s}f_{s}\Vert_{L^2(\mathbb{R }^n)}
    \lesssim \Vert \partial^{seq}_sf_s\Vert_{L^2(\mathbb{R}^n)}\Vert f_s\Vert_{L^2(\mathbb{R}^n)}
    \\
    &\lesssim \big(\Vert\partial_sA(\cdot,s)\Vert_{L^\infty(\mathbb{R}^n)}\Vert\nabla\phi^s\Vert_{L^2(\mathbb{R}^n)}+\Vert\nabla\partial_s\phi^s\Vert_{L^2(\mathbb{R}^n)}\big)\Vert\nabla\phi^s\Vert_{L^2(\mathbb{R}^n)}
    \\
    &\lesssim C_1^2\Vert\partial_sA(\cdot,s)\Vert_{L^\infty(\mathbb{R}^n)}+C_1\Vert\nabla\partial_s\phi^s\Vert_{L^2(\mathbb{R}^n)}.
\end{align*}

For \(I_2\), the weak differentiability of $A$ yields $\partial_sA=\partial_s^{seq}A$. Thus, Minkowski's inequality, Cauchy-Schwarz, \refprop{prop:L2boundednessOfSemigroupOperators} and Poincar\'{e}'s inequality imply
\begin{align*}
    I_2&\lesssim \int_0^s2\tau \Vert e^{-(s^2-\tau^2) L^{s}}\mathrm{div}(\partial_s^{seq}\AP\nabla e^{-\tau^2L^s}f_{s})\Vert_{L^2(\mathbb{R}^n)} \Vert e^{-s^2 L^s}f_{s}\Vert_{L^2(\mathbb{R}^n)} d\tau
    \\
    &\lesssim \Big(\Vert \partial_sA(\cdot, s)\Vert_{L^\infty(\mathbb{R}^n)}\int_0^s2\frac{\tau}{\sqrt{s^2-\tau^2}}\Vert\nabla e^{-\tau^2L^s}f^s\Vert_{L^2(\mathbb{R}^n)}  d\tau\Big) \Vert f_s\Vert_{L^2(\mathbb{R}^n)}
    \\
    &\lesssim \Vert \partial_sA(\cdot, s)\Vert_{L^\infty(\mathbb{R}^n)}\Vert f_s\Vert_{L^2(\mathbb{R}^n)}^2\lesssim \Vert \partial_sA(\cdot, s)\Vert_{L^\infty(\mathbb{R}^n)}\Vert \nabla\phi^s\Vert_{L^2(\mathbb{R}^n)}^2
    \\
    &\lesssim C_1^2\Vert \partial_sA(\cdot,s)\Vert_{L^\infty(\mathbb{R}^n)}.
\end{align*}
Hence, condition \eqref{cond:mixedL1LInftyCond} yields 
\[\int_0^\infty I_2+I_3ds\lesssim C_1^2\int_0^\infty\Vert\partial_s A(\cdot,s)\Vert_{L^{\infty}(\mathbb{R}^n)}ds+C_1\int_0^\infty\Vert\nabla\partial_s \phi^s\Vert_{L^{2}(\mathbb{R}^n)}ds\lesssim C_1^2+C_1C_2.\]

Lastly, for \(I_1\) we can combine the weak differentiability of \(\Vert e^{-s^2L^s}f_{s}\Vert_{L^2(\mathbb{R}^n)}^2\) in $s$, \refprop{prop:L2NormGoesTo0} and \eqref{equation:L2ofnablaphiBounded} and get
\begin{align*}
     \int_0^\infty I_1 ds&=-\frac{1}{2}\int_0^\infty\partial_s^{seq}\big(\Vert e^{-s^2L^s}f_{s}\Vert_{L^2(\mathbb{R}^n)}^2\big)ds=-\frac{1}{2}\int_0^\infty\partial_s\big(\Vert e^{-s^2L^s}f_{s}\Vert_{L^2(\mathbb{R}^n)}^2\big)ds
     \\
     &=-\lim_{a\to 0}\frac{1}{2}\int_a^\infty\partial_s\big(\Vert e^{-s^2L^s}f_{s}\Vert_{L^2(\mathbb{R}^n)}^2\big)ds=
     \lim_{a\to 0}\frac{1}{2}\Vert e^{-a^2L^a}f_{a}\Vert_{L^2(\mathbb{R}^n)}^2
     \\
     &\lesssim \lim_{a\to 0}\Vert f_{a}\Vert_{L^2(\mathbb{R}^n)}^2\lesssim \lim_{a\to 0}\Vert \nabla\phi^a\Vert_{L^2(\mathbb{R}^n)}^2\lesssim C_1^2.
\end{align*}

This finishes the proof of \eqref{lemma:w_tSqFctBound1}.
\end{proof}

\section{Hodge decomposition}\label{appendix:HodgeDecomposition}

Let us fix a boundary ball \(\Delta\subset\partial\Omega\), and assume that $c(x,s)$ is a function with the following properties that
\begin{itemize}
    \item $c$ is weakly differentiable in $s$;
    \item there exists $m=m(s)>0$ such that 
    \begin{enumerate}[(i)]
        \item $\mathcal{M}(m)\in L^1_{loc}((0,\infty))$, and
        \item $\Vert c(\cdot,s)\Vert_{L^2(3\Delta)}<m(s)<\infty$.
    \end{enumerate}
\end{itemize}

For each \(s>0\) we can find the unique Hodge decomposition consisting of \(\phi^s\in W_0^{1,2}(3\Delta)\) and \( h^s\in L^2(\Omega)\), where \(h^s\) is \(x\)-divergence free and
\[c(x,s)\chi_{3\Delta}(x)=\AP(x,s)\nabla\phi^s(x)+h^s(x),\]
which means that for any test function \(v\in W^{1,2}_0(3\Delta)\)
\[\int_{3\Delta}\AP(x,s)\nabla\phi^s(x)\cdot\nabla v(x) dx=\int_{3\Delta}c(x,s)\cdot\nabla v(x)dx\]
holds. 
Analogous to \cite[Section~3]{ulmer_solvability_2026}, we discuss in the following this family of Hodge decomposition, and derive that both $\phi^s$ and $\nabla \phi^s$ are weakly differentiable in $s$ and satisfy certain bounds. 
\medskip

First, since we can test above equation by $v=\phi^s$, this implies with ellipticity of $A$
\begin{align}\Vert\nabla\phi^s\Vert_{L^2(3\Delta)}\lesssim \Vert c\Vert_{L^2(3\Delta)}<C(s).\label{equation:L2ofnablaphiBounded}\end{align}

Next, let us justify that $\phi^s$ is weakly differentiable in $s$. This key observation is due to the control of $\partial_sA$. We introduce $w^h(x,s):=\frac{\phi^{s+h}(x)-\phi^s(x)}{h}$. Then 
\begin{align}
    &\int_{3\Delta} \AP(x,s+h)\nabla w^h(x,s)\cdot\nabla v(x) dx\nonumber
    \\
    &=\frac{1}{h}\int_{3\Delta} \AP(x,s+h)\nabla \phi^{s+h}(x)\cdot\nabla v(x)+(\AP(x,s)-\AP(x,s+h))\nabla \phi^{s}(x)\cdot\nabla v(x)\nonumber
    \\
    &\qquad\qquad-\AP(x,s)\nabla \phi^{s}(x)\cdot\nabla v(x)dx\nonumber
    \\
    &=\int_{3\Delta} \frac{c(x,s+h)-c(x,s)}{h}\cdot\nabla v(x)+\frac{\AP(x,s)-\AP(x,s+h)}{h}\nabla \phi^{s}(x)\cdot\nabla v(x)dx.\label{eq:TakeLimitinSec4}
\end{align}
Plugging in $v=w^h(\cdot,s)$, using ellipticity of $A$ and \eqref{equation:L2ofnablaphiBounded} yields
\begin{align}
    &\Vert \nabla w^h(\cdot,s)\Vert_{L^2(\mathbb{R}^n)}^2\lesssim \int_{3\Delta}\AP(x,s+h)\nabla w^h(x,s)\cdot\nabla w^h(x,s) dx\nonumber
    \\
    &\qquad\leq \Big\Vert \frac{c(x,s+h)-c(x,s)}{h}\Big\Vert_{L^2(3\Delta)}\Vert \nabla w^h(\cdot,s)\Vert_{L^2(\mathbb{R}^n)}\nonumber
    \\
    &\qquad\qquad +\Big\Vert \frac{\AP(x,s+h)-\AP(x,s)}{h}\Big\Vert_{L^\infty(\mathbb{R}^n)}\Vert \nabla\phi^{s}\Vert_{L^2(\mathbb{R}^n)}\Vert \nabla w^h(\cdot,s)\Vert_{L^2(\mathbb{R}^n)}\nonumber
    \\
    &\qquad\leq \Big\Vert \frac{c(x,s+h)-c(x,s)}{h}\Big\Vert_{L^2(3\Delta)}\Vert \nabla w^h(\cdot,s)\Vert_{L^2(\mathbb{R}^n)}\nonumber
    \\
    &\qquad\qquad\lesssim \Big\Vert \frac{A(x,s+h)-A(x,s)}{h}\Big\Vert_{L^\infty(\mathbb{R}^n)}\Vert c\Vert_{L^2(3\Delta)}\Vert \nabla w^h(\cdot,s)\Vert_{L^2(\mathbb{R}^n)},\label{eq:analogboundSec4}
\end{align}
and consequently 
\begin{align*}\Vert \nabla w^h(\cdot,s)\Vert_{L^2(\mathbb{R}^n)}&\lesssim \Big\Vert \frac{c(x,s+h)-c(x,s)}{h}\Big\Vert_{L^2(3\Delta)}
\\
&\qquad+\Big\Vert \frac{A(x,s+h)-A(x,s)}{h}\Big\Vert_{L^\infty(\mathbb{R}^n)}\Vert c\Vert_{L^2(3\Delta)}.\end{align*}
Since $A$ and $c$ are weakly differentiable in $s$, $A$ and $c$ are absolutely continuous on almost every line in $s$ direction (cf. \cite[Theorem 1]{mazya_sobolev_2011}), and the Fundamental Theorem of Calculus applies on almost every line in $s$ direction. We obtain by the pointwise bound $|\partial_sA|s\leq C$ for a.e. $s$
\begin{align}\Big\Vert \frac{A(x,s+h)-A(x,s)}{h}\Big\Vert_{L^\infty(\mathbb{R}^n)}&\leq \fint_s^{s+h}\Vert \partial_tA(\cdot,t)\Vert_{L^\infty(\mathbb{R}^n)}dt\leq \frac{C}{s},\label{eq:LeftInequalityLimitSec4}\end{align}
and via assumption $\Vert \partial_s c\Vert_{L^2(3\Delta)}\lesssim m(s)$ and $\mathcal{M}[m]\in L^1_{loc}$
\begin{align}\Big\Vert \frac{c(x,s+h)-c(x,s)}{h}\Big\Vert_{L^2(3\Delta)}&\leq \Big(\fint_s^{s+h}\Vert \partial_tc(\cdot,t)\Vert_{L^2(3\Delta)}^2dt\Big)^{1/2}\leq \mathcal{M}(m)(s),\label{eq:LeftInequalityLimitSec4}\end{align}
and hence $\Vert \nabla w^h(\cdot,s)\Vert_{L^2(\mathbb{R}^n)}\lesssim \frac{C}{s}+\mathcal{M}(m)(s):=\tilde{C}(s)$ is uniformly bounded in $h$. Since \(\nabla w^h(\cdot,s)\in W^{1,2}_0(3\Delta)\), Poincar\'{e} in $x$ yields $\Vert w^h(\cdot,s)\Vert_{L^2(\mathbb{R}^n)}\lesssim l(\Delta)\tilde{C}$ which is a bound uniformly in $h$. Thus, for a fixed $s$, Banach-Alaoglu states that there exists a subsequence $w^{h_k}(\cdot,s)$ that converges weakly to a function $\tilde{w}(\cdot,s)\in W^{1,2}_0(3\Delta)$. 
Now, absolute continuity of $A$ and $c$ on almost every line in $s$-direction, the weak convergence of $w^{h_k}\rightharpoonup\tilde{w}$, and the weak differentiability of the coefficients of $A$ and $c$ allows us to take the limit in \eqref{eq:TakeLimitinSec4} and get
\begin{align}
\int_{3\Delta}\AP(x,s)\nabla \tilde{w}(x,s)\cdot\nabla v(x)dx
=\int_{3\Delta}\partial_sc(x,s)\cdot\nabla v(x)-\partial_s \AP(x,s)\nabla \phi^s\cdot\nabla v(x)dx.\label{equation:L2ofpartial_snablaphiBounded23}
\end{align}
By Lax-Milgram, this choice of $\tilde{w}$ is unique. Thus, any other convergent subsequence of $w^h(\cdot,s)$ converges to $\tilde{w}(\cdot,s)$. This implies that $w^h(\cdot,s)\rightharpoonup\tilde{w}(\cdot,s)$ as \(h\to 0\) weakly in \(W^{1,2}_0(3\Delta)\) without the need to pass to a subsequence.

We are now equipped to prove that $\phi^s$ and $\nabla\phi^s$ are weakly differentiable in $s$ on $\Omega$. Let $\psi\in C_c^\infty(\Omega)$. Then by a change of variables argument
\begin{align*}
    \int_\Omega w^h \psi \,dxds=\int_{\Omega}\phi^s(x)\frac{\psi(x,s-h)-\psi(x,s)}{h}\,dxds.
\end{align*}
Since $\psi$ is smooth, the right hand side converges to $-\int_{\Omega}\phi^s(x)\partial_s\psi(x)\,dxds$. For the left hand side, we use that $\psi\in L^\infty$, $\mathrm{supp}(\psi)\Subset\Omega$ and 
\[\Vert w^h\Vert_{L^1(\mathbb{R}^n)}\leq|\Delta|^{1/2}\Vert w^h\Vert_{L^2(3\Delta)}\lesssim |\Delta|^{1/2}\tilde{C}.\]
Then the dominated convergence theorem is applicable, since 
\[s\mapsto\tilde{C}(s)\Vert\psi\Vert_{L^\infty(\Omega)}\chi_{\mathrm{supp}(\psi)}= \Big(C\frac{\Vert\psi\Vert_{L^\infty(\Omega)}}{s}+\mathcal{M}[m](s)\Vert\psi\Vert_{L^\infty(\Omega)}\Big)\chi_{\mathrm{supp}(\psi)}\]
is integrable and dominates $s\mapsto\int_{\mathbb{R}^n}w^h \psi \,dx$, and we obtain
\begin{align*}
    \lim_{h\to0}\int_0^\infty \int_{\mathbb{R}^n} w^h \psi \,dxds=\int_0^\infty \lim_{h\to0}\int_{\mathbb{R}^n} w^h \psi \,dxds=\int_0^\infty\int_{\mathbb{R}^n} \tilde{w} \psi \,dxds.
\end{align*}
Here we used the weak convergence $w^h(\cdot,s)\rightharpoonup\tilde{w}(\cdot,s)$ in \(W^{1,2}(\mathbb{R}^n)\) for each fixed $s$ in the second equality. Analogously, by replacing $w^h$ and $\tilde{w}$ by $\nabla w^h$ and $\nabla\tilde{w}$, and letting $\psi$ be a vector valued function, we obtain that $\nabla \phi^s$ is weakly differentiable in $s$. Hence, we can write $\tilde{w}=\partial_s\phi^s$. From \eqref{equation:L2ofpartial_snablaphiBounded23}, we get
\begin{align*}\int_{3\Delta} \AP(x,s)\nabla\partial_s\phi^s(x)\cdot\nabla v(x) dx&= - \int_{3\Delta}\partial_s \AP(x,s)\nabla\phi^s(x)\cdot\nabla v(x) dx 
\\
&\qquad + \int_{3\Delta}\partial_s c(x,s)\cdot\nabla v(x).\end{align*}
Inserting \(v=\partial_s\phi^s\in W^{1,2}_0(3\Delta)\) leads to

\begin{align*}
\lambda_0\Vert \nabla\partial_s\phi^s\Vert_{L^2(3\Delta)}^2&\leq \Big|\int_{3\Delta} \AP(x,s)\nabla\partial_s\phi^s\cdot\nabla\partial_s\phi^s dx\Big|
\\
&\leq\Big|\int_{3\Delta}\partial_s \AP(x,s)\nabla\phi^s(x)\cdot\nabla\partial_s\phi^s(x) dx\Big|
\\
&\qquad+ \Big|\int_{3\Delta}\partial_s c(x,s)\cdot\nabla\partial_s\phi^s(x) dx\Big|
\\
&\lesssim\Vert\partial_s \AP(\cdot,s)\Vert_{L^\infty(3\Delta)}\Vert \nabla\phi^s\Vert_{L^2(3\Delta)}\Vert \nabla\partial_s\phi^s\Vert_{L^2(3\Delta)} 
\\
&\qquad + \Vert\partial_s c(\cdot,s)\Vert_{L^2(3\Delta)}\Vert \nabla\partial_s\phi^s\Vert_{L^2(3\Delta)},
\end{align*}

and hence 

\begin{align}
\Vert \nabla\partial_s\phi^s\Vert_{L^2(3\Delta)}&\lesssim \Vert\partial_s\AP(\cdot,s)\Vert_{L^\infty(3\Delta)}\Vert c \Vert_{L^2(3\Delta)}+\Vert\partial_s c(\cdot,s)\Vert_{L^2(3\Delta)}\nonumber
\end{align}

\bibliographystyle{alpha}
\bibliography{references} 
\end{document}